\documentclass[english]{article}

\usepackage{graphicx} % Required for inserting images
\usepackage[utf8]{inputenc} % Encodage des caractères
\usepackage[english]{babel} % Langue du document
\usepackage{amsmath, amssymb, amsthm} % Packages AMS pour les mathématiques
\usepackage{geometry}       % Pour régler les marges
\usepackage[colorlinks=true, allcolors=blue]{hyperref}    % Pour les liens hypertextes et les références
\usepackage{tocloft}        % Pour personnaliser le sommaire
\usepackage{braket}
\usepackage{cleveref}
\usepackage{tikz}
\usepackage{nicematrix}
\usepackage{bbm}
\usepackage{enumitem}
\usepackage{float}
\usepackage{microtype}

\usepackage[round]{natbib}
\usepackage[textsize=small]{todonotes}

\newcommand{\R}{\mathbb{R}}
\newcommand{\N}{\mathbb{N}}
\newcommand{\E}{\mathbb{E}}
\renewcommand{\P}{\mathbb{P}}

\usepackage{authblk}
\numberwithin{equation}{section}

\newtheorem{theorem}{Theorem}[section]
\newtheorem{lemma}[theorem]{Lemma}
\newtheorem{corollary}[theorem]{Corollary}

\newtheorem{proposition}[theorem]{Proposition}
\newtheorem{example}[theorem]{Example}

\newtheorem{assumption}{Assumption}[section]

\title{\bfseries
	Fake stationary rough Heston volatility: Microstructure-inspired foundations\footnote{Emmanuel Gnabeyeu and Mathieu Rosenbaum gratefully acknowledge support from the ILB Chair \textit{Artificial Intelligence and Quantitative Methods for Finance} at University Paris Dauphine-PSL.}}

\author[1]{Emmanuel Gnabeyeu\footnote{E-mail: {\tt emmanuel.gnabeyeu\_mbiada@sorbonne-universite.fr} }}
\author[2]{Gilles Pag\`es\footnote{E-mail: {\tt gilles.pages@sorbonne-universite.fr} }}
\author[3]{Mathieu Rosenbaum\footnote{E-mail: {\tt mathieu.rosenbaum@dauphine.psl.eu} }} 
\date{\today}
\affil[1,2]{Sorbonne Universit\'e and Universit\'e Paris Cit\'e, LPSM}
\affil[3]{Universit\'e Paris Dauphine-PSL, Ceremade}

\begin{document}
	
	\maketitle
	
\begin{abstract}
  % This paper investigates the asymptotic behavior and establishes scaling limit theorems for suitably time-modulated, renormalized, heavy-tailed, and nearly unstable Hawkes processes. 	
  This paper investigates the asymptotic behavior %and establishes scaling limit theorems for 
  of suitably time-modulated Hawkes processes with heavy-tailed kernels in a nearly unstable regime.
	We show that, under appropriate scaling, both the intensity processes and the rescaled Hawkes processes converge to a mean-reverting, time-inhomogeneous rough fractional square-root process and its integrated counterpart, respectively.
	In particular, when the original Hawkes process has a stationary first moment (constant marginal mean), the limiting process takes the form of a time-inhomogeneous rough fractional Cox–Ingersoll–Ross (CIR) equation with a constant mean-reversion parameter and a time-dependent diffusion coefficient. This class of equations is particularly appealing from a practical perspective, especially for the so-called \textit{fake stationary rough Heston} model. We further investigate the properties of such limiting scaled time-inhomogeneous Volterra equations, including moment bounds, path regularity and maximal inequality in the \(L^p\) setting for every \(p>0\). 
\end{abstract}



\section{Introduction}
Rough volatility models, where the volatility behave like fractional Brownian motion-driven SDEs with Hurst coefficient \(H\approx 0.1\), have attracted considerable attention in mathematical finance over the past decade, as they provide a consistent description of financial market dynamics across asset classes %\citep{bennedsen2022} 
and time scales, 
%ranging from high-frequency order book data to the pricing of derivative securities. 
with many applications such as volatility forecasting and the pricing of derivatives.
% Despite their inherently non-Markovian nature, these models have proven to be remarkably effective in capturing empirical features of volatility, most notably its rough behavior.
Among these models, the rough Heston model introduced in \cite{ElEuchR2018,el2019characteristic} as a natural extension  and rough version of the celebrated %classical
 \cite{Heston1993} model, plays a central role. It is particularly appealing as it admits a clear and convincing interpretation rooted in market microstructure.  More precisely, the rough Heston dynamics arise as the continuous-time scaling limit of suitably well-chosen Hawkes processes modeling high-frequency flows \citep{JaissonRosenbaum2015,JaissonRosenbaum2016,ElEuchFukasawaRosenbaum2018,el2019characteristic}. In this framework, the rough behavior of volatility emerge endogenously from the self-exciting nature of order arrivals. Moreover,
 the model admits semi-closed-form representations of its characteristic function \citep{el2019characteristic,abi2019affine}, which enable efficient numerical methods \citep{CallegaroGrasselliPages2020}, making the model suitable for practical applications such as option pricing and volatility fitting, while simultaneously providing a flexible framework for studying the joint statistical properties of prices and volatility, a key objective in financial econometrics.
 
\medskip
\noindent In this setting, the stock price process \( S \) evolves according to the stochastic differential equation ( SDE)
\begin{equation}\label{E:HestonlogS}
	\frac{dS_t}{S_t} = \sqrt{V_t}\,\big(\sqrt{1-\rho^2}\,dW^{(2)}_t + \rho\,dW^{(1)}_t\big), \quad \rho \in (-1,1), \quad
	\quad S_0 \in (0,\infty),
\end{equation}
where  \( (W^{(1)},W^{(2)}) \) is a two-dimensional independent standard Brownian motion and 
the variance process \( V \),
follows a fractional mean-reverting Cox--Ingersoll--Ross (CIR) dynamics,
\begin{equation}\label{E:HestonV}
	V_t = V_0 + \int_0^t K_{\alpha}(t-s)
	\left( (\theta(s) - \lambda V_s)\,ds + \nu \sqrt{V_s}\,dW^{(1)}_s \right),
	\qquad V_0 \perp\!\!\!\perp W^{(1)}.
\end{equation}
Here, $K_{\alpha}$ is a {nonnegative} locally square-integrable kernel i.e. $K_{\alpha}$ lies in $L^2_{\rm loc}(\mathbb{R}_+,\mathbb{R}_+)$, \( \theta \) a deterministic mean-reversion level and $\lambda, \nu \in \R_+$.
%Processes of this type appear in a wide range of applications. In particular, they not only arise as limits of suitably rescaled %non-Markovian 
%self-exciting Hawkes processes in financial mathematics \citep{ElEuchFukasawaRosenbaum2018, HorstXuZhang2024}, but also as continuous-time scaling limits of branching processes in population genetics \citep{dawson1994super, mytnik2015uniqueness}.
Processes of this type appear in a wide range of applications. In particular, they arise not only as limits of suitably rescaled self-exciting Hawkes processes in financial mathematics \citep{ElEuchFukasawaRosenbaum2018, HorstXuZhang2024},  but also as continuous-time scaling limits of branching processes in population genetics \citep{dawson1994super, mytnik2015uniqueness}.

\medskip
\noindent In financial applications, modeling volatility and pricing derivatives across multiple time horizons raises fundamental consistency challenges. This motivates the study of appropriate stationarity properties for fractional CIR-type equations, with the aim of providing a unified and consistent framework that captures both short- and long-maturity behaviors and enables robust calibration across the entire term structure.
%The mathematical properties that govern these processes~\eqref{E:HestonV} have profound implications in finance and become a major challenge when modeling volatility across multiple time horizons. As a result, the search for some suitable stationarity properties within fractional CIR-type equations can be of interest, with the aim of providing a unified and consistent framework for capturing both short- and long-maturity behaviors, and enabling robust calibration across the entire term structure.
In fact, a well-known drawback of the classical \cite{Heston1993} model is that it typically exhibits two distinct regimes: a short-maturity regime, where the initial condition (often chosen as the long-run mean) dominates and the variance remains close to zero; and a long-term regime, which may correspond to the stationary distribution of the process. This issue becomes even more intricate in the rough Heston model \citep{el2019characteristic}, and more generally, in affine Volterra models \cite*[Section~7]{abi2019affine}, since Volterra equations are \emph{inherently non-stationary}, as emphasized in the recent works \citep{Pages2024, EGnabeyeu2025}.
To overcome this limitation and provide a more consistent modeling framework, \cite{EGnabeyeu2025} introduced a notion of stabilization for Volterra equations with affine drift and suggested the possibility of stabilized versions of Volterra models, where the dynamics driving the asset’s volatility exhibit constant mean and variance over time. 
Their approach relies on a time-modulated diffusion coefficient designed to solve a suitable functional equation, ensuring that all marginal distributions of the variance process share the same mean and variance, despite the non-stationarity of the underlying Volterra dynamics. This provides examples of \emph{fake stationary Volterra processes} in the sense of \cite{Pages2024}. 

\medskip
\noindent In the rough Heston setting, this would result in the so-called \emph{fake stationary rough Heston model}, in which the variance process \(V\) is driven by  a stabilized fractional Cox–Ingersoll–Ross dynamics:
\begin{equation}\label{eq:Volterra}
	V_t = V_0\phi(t) + \int_0^t K_{\alpha}(t-s) \left( (\theta(s) - \lambda V_s)\,ds + \nu\, \varsigma(s) \sqrt{V_s}\, dW^{(1)}_s \right), \; \varsigma = \varsigma_{\alpha,\lambda,c},\; c > 0,\; V_0 \perp\!\!\!\perp W^{(1)}.
\end{equation}
Following the approach of \cite{EGnabeyeu2025}, the time-dependent function \(\varsigma\), appearing in the diffusion coefficient, should be the solution of the below functional equation~\eqref{eq:funcEq}, which is known to be unique and ensures that all marginal distributions of the process \(V\) have constant first and second-order moments over time:
\begin{equation}\label{eq:funcEq}
	\forall\, t\ge 0, \quad c \lambda^2\big(1-(\phi(t)-(f_{\alpha,\lambda} * \phi)_t)^2 \big) =  (f_{\alpha,\lambda}^2 * \varsigma_{\alpha,\lambda,c}^2)(t), \;\; \phi(t)   =1 - \lambda \int_0^t K_{\alpha}(t-s) \left( \frac{\theta(s)}{\mu_\infty} - 1 \right) \, ds.
\end{equation}
In this equation, \(\mu_\infty = \lim_{t \to \infty} \theta(t)\), and \(f_{\alpha,\lambda} := -R_{\alpha,\lambda}^\prime\), where \(R_{\alpha,\lambda}\) is the \(\lambda\)-resolvent of the kernel \(K_{\alpha}\), defined in~\eqref{eq:Resolvent}.
 As a result, for every \(t \ge 0\), 
\( \mathbb{E}[V_t] = \frac{\mu_\infty}{\lambda}\) and \( \text{Var}[V_t] = c\nu^2 \frac{\mu_\infty}{\lambda}\) (see e.g. \cite{EGnabeyeuPR2025}). 
%Such new model belongs to the class of diffusion-based models used in the industry to encode implied-volatility information, alongside non-parametric grids of implied volatilities with spline interpolation, direct modelling of the asset’s implied density and surface-level parametrizations followed by AI-driven fitting as illustrated in \cite{GnabeyeuKarkarIdboufous2024}.

\medskip
\noindent However, the formulation~\eqref{eq:Volterra} and~\eqref{eq:funcEq} naturally raises fundamental questions regarding the \emph{well-posedness} (at least in the sense of the theory of weak solutions) and the \emph{positivity} of solutions to such stabilized Volterra equations~\eqref{eq:Volterra} for any deterministic and positive functions \(\varsigma\) and \(\phi\). In this paper, we tackle these issues through a continuous-time scaling limit of suitably time-modulated rescaled, non-Markovian self-exciting linear Hawkes processes. 
As a consequence, our approach offers a probabilistic foundation, inspired by microstructure,
%our approach provides a probabilistic and microstructural foundation 
for the \emph{fake stationary rough Heston model}, thereby extending the applicability of the results from \cite{EGnabeyeuPR2025} on more general \emph{fake stationary affine Volterra processes} to this setting. The core developments of this paper build on Hawkes processes with nonzero initial conditions, as described below.

\subsection{Hawkes processes: background and related work} % Literature Review}

Let \( N \) be a simple point process on the real line \( \mathbb{R} \), defined by a family of integer-valued random variables \( \{ N(o) \}_{o \in \mathcal{B}(\mathbb{R})} \) taking values in \( \N \), indexed by the Borel \( \sigma \)-algebra \( \mathcal{B}(\mathbb{R}) \), where 
$
N(o) = \sum_{i \in \mathbb{N}} \mathbf{1}_o(\tau_i),
$
and the random variable $\tau_i$ denoting the arrival time of the $i$-th event, for each $i \in \mathbb N$ i.e. \( (\tau_i)_{i \in \mathbb{N}} \) is a sequence of real-valued random variables satisfying \( \tau_i < \tau_{i+1} \) for all \( i \in \mathbb{N} \). 
Let \( \mathcal{F}_t = \sigma(N(o), o \subset (0, t]) \) be the filtration generated by the random variables \( N(o) \), with \( o \in \mathcal{B}(\mathbb{R}) \) and which represents the history of the process.
 Let $\mu \in    L^1_{\rm loc} (\mathbb{R}_+;\mathbb{R}_+)$ be a non-negative and locally integrable function. We define the univariate \textit{linear Hawkes process} with \textit{baseline intensity }
 or \textit{immigration density} $\mu$ and some \textit{kernels} or \textit{exciting functions} \(\varphi \in L^1(\mathbb{R}_+; \mathbb{R}_+)\) as a simple random point process \((N, \Lambda)\) that models self-exciting arrivals of random events via \(N := \{N(t): t \geq 0\}\), and possesses the \( \mathcal{F}_t \)-\textit{stochastic intensity} \(\Lambda\) defined below, which depends on the past realizations of the process itself:
\begin{equation}\label{HawkesDensityLinear}
	\Lambda_t = \mu(t)+ \sum_{0<\tau_i<t} \varphi(t-\tau_i) = \mu(t)+ \sum_{i=1}^{N(t)} \varphi(t-\tau_i) =\mu(t) + \int_{(0,t)} \varphi(t-s)N(ds), 
\end{equation}
The \textit{baseline intensity } $\mu$ captures the immigration of exogenous events and the \textit{kernels} $\varphi$ captures the self-exciting impact of past events on the arrivals of future events (For instance, in finance, this enables the modeling of the self-exciting nature of the order flow).
As a result of the expression in equation ~\eqref{HawkesDensityLinear}, the univariate Hawkes process is sometimes referred to as a \emph{self-exciting point process} in the literature.

\medskip
\noindent The core idea behind Hawkes processes is to model the influence of historical events on the occurrence of future events. Setting the initial state condition to zero in~\eqref{HawkesDensityLinear}, as is commonly done in many studies on scaling limits in the literature, implies that no events have occurred before time zero. This assumption, however, introduces inconsistency in the model's dynamics, as it overlooks %neglects 
the fundamental dependency of future event arrivals on past events. In this regard, we are going to consider a standard Hawkes process \( N \) defined on the whole real line with intensity:
\begin{equation}\label{HawkesDensityLinear2}
	\Lambda_t = \mu(t)+ \sum_{\tau_i<t} \varphi(t-\tau_i) =\mu(t) + \int_{(-\infty,t)} \varphi(t-s)N(ds) = Z_0(t) + \mu(t) + \int_{(0,t)} \varphi(t-s)N(ds), 
\end{equation}
The function \( Z_0 \in L^1_{\text{loc}}(\mathbb{R}^+; \mathbb{R}^+) \) is a random function evolving deterministically for \(t\geq0\) i.e. for \(t>0, Z_0(t)\), is \(\mathcal F_0-\) measurable, where \(\mathcal F_0\) is the history of the process up to time 0.  The random process \(Z_0 \) represents the combined impact of all the events that arrived prior to time zero on future arrivals at time \( t > 0 \). It defines the initial distribution required to make the point process stationary. In that particular case where the first order mean of the intensity process is as least constant, the mean impact of the random function \(Z_0\) is given as follow:
\begin{equation}\label{ExpectationZ} \mathbb{E} [ Z_0(t)] =
	\mathbb{E} \left[ \int_{-\infty}^0 \varphi(t - s) \, dN(s) \right] = \int_{-\infty}^0 \varphi(t - s) \, \mathbb{E} [ \Lambda_s]ds = \mathbb{E} [ \Lambda_0]\int_{-\infty}^0 \varphi(t - s) \, ds = \mathbb{E} [ \Lambda_0] \cdot \Phi(t).
\end{equation}
Thus, \(Z_0 (t)\) is non-negative, converges almost surely to zero, is integrable almost surely over any finite interval \([0,T]\) and satisfy $\mathbb{E} [ Z_0(t)] \overset{t\to+\infty}{\to} 0$ if \((\Lambda_t)_{t}\) has at least constant mean.

\medskip
\medskip
\noindent The long-term behavior of the Hawkes process \((N,\Lambda) \)  is determined by the temporal release \(\mathit{\Phi}(t) := \int_t^\infty \varphi(s) \, ds, \quad t \geq 0.\)
Here, \(  a := \|\varphi\|_{L^1} = \mathit{\Phi}(0) \) represents the total impact of an event  on the arrival of future events, measured by its expected number of descendants, with \( \mathit{\Phi} \) describing its \textit{temporal release} of the total impact.
In \cite{BacryDelattreHoffmannMuzy2013}, it is shown that under the stability condition \(a = \mathit{\Phi}(0)= \int_0^{+\infty} \varphi(s) \, ds < 1 \), the asymptotic behavior of a Hawkes process resembles that of a Poisson process: \(N_t\) has asymptotically stationary increments, and \(\Lambda_t\) becomes asymptotically stationary as \(t \to \infty\). This behavior yields useful approximations and convergence guarantees for statistical inference procedures via functional limit theorems (FLTs) (FLLN and FCLT).
When the stability condition is violated, for nearly unstable Hawkes processes, neither a standard FLLN nor a standard FCLT  can
be expected.
In \cite{JaissonRosenbaum2015,ElEuchFukasawaRosenbaum2018,ElEuchR2018}, the authors study the scaling limit of Hawkes processes when the stability condition is nearly violated. This is motivated by empirical observations in financial data, where Hawkes processes are commonly used to model the clustered nature of order flows in markets. The near instability required to fit the data reflects the high degree of endogeneity in modern markets, driven by high-frequency trading. Essentially, a significant proportion of orders are triggered by other orders. In this context, \cite{JaissonRosenbaum2015,ElEuchFukasawaRosenbaum2018} proves that the limiting distribution of a sequence of nearly unstable Hawkes processes is that of an integrated Cox-Ingersoll-Ross (CIR) process and integrated fractional Cox-Ingersoll-Ross (CIR) process respectively. In these settings, the sequence of renormalized processes satisfies the stability condition but with a kernel \( \varphi_T \) that depends on the observation scale \( T \), such that \( \|\varphi_T\|_1 \underset{T \to +\infty}{\to} 1 \). These are referred to as nearly unstable Hawkes processes. We aim to extend these results to time-dependent settings and initial conditions.

\subsection{Our contribution}
Building upon the asymptotic limits established in \cite{JaissonRosenbaum2015,JaissonRosenbaum2016, ElEuchFukasawaRosenbaum2018}, we prove that, after applying an appropriate \textit{time-dependent scaling} in both time and space, to the intensity process~\eqref{HawkesDensityLinear2}, the intensity of a nearly unstable Hawkes process can be asymptotically approximated by a time-inhomogeneous stochastic rough fractional square-root process, denoted \( \Lambda^* \), and that the Hawkes process \(N\) behaves asymptotically as the integrated process of \(\Lambda^*\).
In particular, when the Hawkes process % is stationary (or 
has constant mean intensity over time, the limiting process is precisely a fractional Cox-Ingersoll-Ross (CIR) equation (constant mean-reversion parameter) with a time-dependent diffusion coefficient.
We focus on the weak convergence of the integrated rescaled intensity processes. 
Establishing the weak convergence of this sequence allows us to infer the limiting process 
of the original sequence of rescaled intensity processes by differentiating the limit 
of the integrated processes, provided that the original sequence is $\mathcal{C}$-tight.
Our main result states that under mild assumption the sequence of suitably time-modulated rescaled intensity of self-exciting linear Hawkes processes
converges to the \emph{unique (in law) positive
solution} of a stochastic rough fractional square-root equation with time-inhomogeneous drift and diffusion coefficients. 
We also derive  \emph{global moment bounds} for this resulting time-dependent Volterra square-root process. Moreover,
we show that its  \emph{H\"older increments are uniformly bounded in \(L^p\) for every \(p>0\)} and derive the  \emph{regularity} and \emph{maximal inequality} of the process.

\medskip
\noindent We assume throughout this paper that the Hawkes process \( (N, \Lambda) \) with kernel \( \varphi \) and immigration density \( \mu \) is defined on a complete probability space \( (\Omega, \mathcal{F}, \P) \), endowed with a filtration \( \{\mathcal{F}_t : t \geq 0\} \) that satisfies the standard or usual conditions.

\medskip
\noindent {\sc \textbf{Notations.}} 

\smallskip 
\noindent $\bullet$ For $p\in(0,+\infty)$, $L_{\mathbb H}^p(\P)$ or simply $L^p(\P)$ denote the set of  $\mathbb H$-valued random vectors $X$  defined on a probability space $(\Omega, {\cal A}, \P)$ such that $\|X\|_p:=(\E[\|X\|_{\mathbb H}^p])^{1/p}<+\infty$.
%\medskip

\noindent $\bullet$  For all \(t \in [0,T]\), \(\|\mathbf{x}\|_t = \sup_{s \in [0, t]} |\mathbf{x}_s| = \sup\Big\{ |\mathbf{x}_s|: s \leq t  \Big\}, \qquad \forall \, \mathbf{x} \in {\cal C}([0,T], \mathbb H)\) and \(| \cdot | := \| \cdot \|_{\mathbb{H}}.\)

\smallskip 
\noindent $\bullet$ For $f, g \!\in {\cal L}_{\R_+,loc}^1 (\R_+, {\rm Leb}_1)$, we define their convolution by $f*g(t) = \int_0^tf(t-s)g(s)ds$, $t\ge 0$.

\smallskip 
\noindent $\bullet$ For $f, g \!\in {\cal L}_{\R_+,loc}^2(\R_+, {\rm Leb})$ and $M$ a martingale, we define  their stochastic convolution by 

\(\qquad \hspace{3cm} f\stackrel{M}{*}g = \int_0^t f(t-s)g(s) dM_s, \quad t\ge 0.\)

\smallskip 
\noindent $\bullet$ For a random variable/vector/process $X$, we denote by $\mathcal L(X)$ or $[X]$ its law or distribution. 

\smallskip 
\noindent $\bullet$ $X\perp \! \! \!\perp Y$  stands for independence of random variables, vectors or processes $X$ and $Y$.  

%\smallskip 
\noindent $\bullet$ For a positive and measurable function \( \varphi: \mathbb{R}^+ \to \mathbb{R}^+ \), we denote: 

\(\qquad \hspace{1.5cm} \| \varphi \|_1 := \int_0^{+\infty} \varphi(u) \, du , \quad \| \varphi \|_2^2 := \int_0^{+\infty} \varphi(u)^2 \, du  \quad \text{and} \quad \displaystyle \|\varphi\|_{\sup} := \sup_{u\in \mathbb{R}^+}|\varphi(u)|.\)

\noindent $\bullet$ For \(t_0>0\), write \( \mathbb{D} := (D([0,t_0], \mathbb{R}), \text{J}_1) \) as  the usual Skorokhod
space \footnote{The Skorokhod space $D([0,t_0]  \times \Omega, \mathbb{R}) $ refers to the set of stochastic processes defined on $[0,t_0]$ that are right-continuous with left limits (càdlàg). This space generalizes the classical space $\mathcal C([0,t_0] \times \Omega, \mathbb{R}^d)$, which is equipped with the uniform convergence norm.A specific metric $d$ on the Skorokhod space induces a topology that is well-suited for handling discontinuities.} i.e. the space of  all $\R$-valued càdlàg functions or processes on \( [0,t_0] \) equipped with the Skorokhod topology \( J_1 \) (see, e.g., Billingsley \cite{Billingsley1999}). \( \overset{\mathcal L }{\Rightarrow} \) stands for the convergence in distribution as a process in the Skorokhod topology.
% \( \overset{\mathcal L - \mathcal S}{\Rightarrow} \) 
\section{Time-dependent fractional square root processes as limits of heavy-tailed and nearly unstable Hawkes processes}
Our goal in the sequel is to show that a suitably rescaled version of the intensity process \(\Lambda_t\) with a time-dependent baseline \(\mu_T\) asymptotically behaves as an inhomogeneous Volterra Cox-Ingersoll-Ross with time-dependent mean-reversion parameter.

\medskip
\noindent From now on, we set \(N((0, t]) = N_t\) and for a given $T$, let $(N_t^T)_{t\geq0}$ be a sequence of one-dimensional Hawkes processes  observed on the time interval $[0,T]$, indexed by \(T> 0\), going to infinity and defined on a probability space $(\Omega^T,\mathcal{F}^T,\mathbb{P}^T)$ equipped with the filtration $(\mathcal{F}_t^T)_{t\in[0,T]}$, where $\mathcal{F}_t^T$ is the $\sigma$-algebra generated by the random variables \( N^T(o) \), with \( o \in \mathcal{B}((0,t]) \).
Our asymptotic setting is that the observation scale $T$ goes to infinity. The intensity process $(\Lambda_t^T)$ is defined for $t\geq 0$ by
\begin{equation}\label{eq:intensity}
	\Lambda_t^T=Z_0^T(t) + \mu^T(t)+\int_0^t \varphi^T(t-s) dN^T_s
\end{equation}
where $(\mu^T (\cdot))_{T\geq0}$ is a sequence of positive real values functions and the $ (\varphi^T)_{T\geq0} $ are non-negative measurable functions on $\mathbb{R}^+$ which satisfies $\|\varphi^T\|_{1}<+\infty$. 
We assume that the sequence $(\mu^T (\cdot))_{T\geq0}$  is constant in the negative half-line, i.e.
$$\mu^T(t) = \left\{
\begin{array}{ll}
	& \mu^T(0) = \mu^T \text{ for } t \leq 0, \; \mu^T \text{ bounded for all } T>0 \\
	& \mu^T(t) \text{ if } t > 0 \text{ to be specified later }, 
\end{array} 
\right.$$
so that, the intensity has a constant mean over the the negative half-line, i.e. $\mathbb{E}[\Lambda^T_t] = \mathbb{E}[\Lambda^T_0] \quad \forall t \leq 0$ and thus for \(t\geq0\), equation~\eqref{ExpectationZ} gives :
$ \mathbb{E}[Z^T_0 (t)] = \mathbb{E}[\Lambda^T_0]\int_t^{+\infty} \varphi^T(s)ds = \mathbb{E}[\Lambda^T_0]\mathit{\Phi}^T(t)$. 

\medskip
\noindent From this point forward, we consider a family of one-dimensional Hawkes processes with intensity defined as follows. Let \(\Lambda_0^{*T} \in L^1(\R_+)\) be a positive real-valued random variable such that $\mathbb{E}[\Lambda_0^{*T}]$ is bounded for every \(T>0\). 
For a given time horizon \(T > 0\) and \(t_0> 0\), let \((N_t^T)_{t \geq 0}\) denote a sequence of Hawkes processes observed over the interval \([0, Tt_0]\), indexed by \(T\) and defined on the probability space \((\Omega^T, \mathcal{F}^T, \mathbb{P}^T)\), equipped with the filtration \((\mathcal{F}_t^T)_{t \in [0, Tt_0]}\). The associated intensity process \((\Lambda_t^T)_{t \geq 0}\) is defined for all \(t \geq 0\) by (note that, similar idea has been used in \cite{HorstXuZhang2024}):
\begin{align}
	\Lambda_t^T 
	= \frac{\Lambda_0^{*T}}{\mathbb{E}[\Lambda_0^T]} \mathbb{E}[Z_0^T(t)] + \mu^T(t) + \int_0^t \varphi^T(t-s) \, dN_s^T 
	= \Lambda_0^{*T} \, \Phi^T(t) + \mu^T(t) + \int_0^t \varphi^T(t-s) \, dN_s^T \label{eq:intensity2}
\end{align}
% Assumptions and intuitions for the results. 
%We describe in this section our asymptotic framework together with intuitions about our main results which are given in Section 3.
%\subsection{Intuition and Main Heuristic Convergence Results for Heavy-Tailed Nearly Unstable Hawkes Processes}\label{subsect:intuition}
\subsection{Assumptions and heuristic computations for the convergence of heavy-tailed nearly unstable Hawkes processes}\label{subsect:intuition}
We are interested in the asymptotic behavior of the renormalized intensity processes \( \Lambda^{T}_t \) for which we need to understand the long-term limit. Let us write using equation ~\eqref{eq:compensator}, 
\(M^T_t = N^T_t -N^T_0 - \int_0^t \Lambda^T_s \, ds\)
for the martingale associated with the point process \( N^T_t \). % Note that, we assume \(M^T_0 = 0\) up to rather consider the representation \(M^T_t = N^T_t-N^T_0 - \int_0^t \Lambda^T_s \, ds\)
We easily obtain owing to Theorem  \ref{MartRep} that the intensity process $\Lambda^T$ admits the representation: For every \(t >0\)
\begin{align}
	\Lambda^T_t 
	&= \Lambda_0^{*T} \, \mathit{\Phi}^T(t) + \mu^T(t) + \int_0^t R^{\prime T}_{-1}(t-s)\left(\Lambda_0^{*T} \, \mathit{\Phi}^T(s) + \mu^T(s)\right) \, ds + \int_0^t R^{\prime T}_{-1}(t-s) \, dM^T_s, \quad t > 0 \nonumber \\
	&= \Lambda_0^{*T} \left( \mathit{\Phi}^T(t) + (R^{\prime T}_{-1} * \mathit{\Phi}^T)_t \right) + \mu^T(t) + \int_0^t R^{\prime T}_{-1}(t-s)\mu^T(s) \, ds + \int_0^t R^{\prime T}_{-1}(t-s) \, dM^T_s \label{SVR2}
\end{align}
where the $-1$-resolvent $R^{\prime T}_{-1}$ of $\varphi^T$ is just the resolvent of the second kind $\Psi_{\varphi^T}$ of $\varphi^T$ (see Appendix~\ref{subsect-tools}).
Taking expectation after integration in equation~\eqref{SVR2} or recalling equation~\eqref{eq:mart} of Theorem \ref{MartRep} with the specific expression in \ref{eq:intensity2} , we can write using in the last equality that $
\int_0^{\cdot} (f * g)(s) \, ds = \left( \int_0^{\cdot} f(s) \, ds \right) * g
$:
\begin{align}
	\mathbb{E}[\mathcal{I}^{\Lambda^{T}}_t] 
	&= \int_0^t \left( \mathbb{E}[\Lambda_0^{*T}]\, \Phi^T(s) + \mu^T(s) \right) ds 
	+ \left( \left( \mathbb{E}[\Lambda_0^{*T}]\, \Phi^T(\cdot) + \mu^T(\cdot) \right) * \mathcal{I}^{R^{\prime T}_{-1}} \right)_t \nonumber \\
	&= \left(1 + (\mathbf{1} * R^{\prime T}_{-1})_t \right) 
	\int_0^t \left( \mathbb{E}[\Lambda_0^{*T}]\, \Phi^T(s) + \mu^T(s) \right) ds \;\; \text{where} \;\; \mathcal{I}^{R^{\prime T}_{-1}}_t := \int_0^t R^{\prime T}_{-1}(s) ds.
	\label{eq:IntLam}
\end{align}

\noindent {\bf Remark.}
	If $ \mu^T(t) \equiv C^{\textit{ste}} = \mu^T$ over the whole real line, we know from Proposition \ref{prop:stat} that, the Hawkes process has constant mean and thus the above equation reads:
	\begin{equation}\label{SVR3}
		\Lambda^T_t = \Lambda_0^{*T} \, (\mathit{\Phi}^T(t) + (R^{\prime T}_{-1}*\mathit{\Phi}^T)_t) + \mu^T (1 + \int_0^t R^{ T}_{-1}(s) ds) + \int_0^t R^{\prime T}_{-1}(t-s)dM^T_s ,\quad t\geq 0.
	\end{equation} 

\begin{assumption}[On nearly unstable heavy-tailed Hawkes processes]\label{assum:Lhawkes}\textcolor{white}{.}
	
\medskip
\noindent {\em (1).}  We are working in the nearly unstable  heavy-tailed
		case since $\mathit{\Phi}(0)=1$. We assume
		for $t\in \mathbb{R}^+$,
		$\varphi^T(t)=a_T\varphi(t),$ where $(a_T)_{T\geq 0}$ is a sequence of positive numbers converging to $1$ such that for all $T$, $a_T<1$ and $\varphi$ is a non-negative %measurable 
		completely monotone function such that $\mathit{\Phi}(0)=\|\varphi\|_{L_1}=1$ and \(\varphi(0)< \infty\). 
		\item Furthermore, the function \(\varphi\) has a regular varying tail, i.e. there exists  some $\alpha\in(0,1)$ so that as \(x \to +\infty\), $\mathit{\Phi}(x)=l(x)x^{-\alpha}$
		where \(l(x)\) is {\em slowly varying} at infinity; that is, \(\frac{l(cx)}{l(x)} \to 1\) for all \(c>0\).\\The reader is invited to refer to \cite{BiGoTe1989,DeHaanFerreira2006} for general results on regularly varying functions. 
		More specifically, we consider the case where there exists some positive constant $C > 0 $, (\(l(x) \equiv 1\)),  such that \(\underset{x\rightarrow+\infty}{\emph{lim}}\alpha x^{\alpha}\mathit{\Phi}(x)=C.\)
		
	\medskip
	\noindent {\em (2).}  Set \(\delta := \frac{C \Gamma(1-\alpha)}{\alpha}.\)
	For \(T > 0\), let \(\theta_0 \in (0,\infty)\), and consider a deterministic
	function \(\tilde{\mu}^T : [0, \frac{t_0}{T}] \to (0,\infty)\), assumed to be Lipschitz, continuous with a
	bounded first derivative and satisfying \(\tilde{\mu}^T (0) = \frac{\mu^T(0)}{\theta_0} = \frac{\mu^T}{\theta_0}\) so that for every \(t > 0\), \(\lim_{T\to +\infty} \frac{\mu^T}{\tilde{\mu}^T (\frac{t}{T})} = \theta_0.\)
	Finally, assume that there exist positive constants \(\lambda\) and \(\nu\),
	a non-negative bounded Borel function \(\varsigma\), and a positive continuous
	function $\phi \in {\cal L}^1([0,t_0],\text{Leb}_1)$  such that, uniformly as \(T \to \infty\),
	\vspace{-.3cm}	
	%Set \(\delta = \frac{C \Gamma(1-\alpha)}{\alpha}\), and for \(T>0\), consider \(\theta_0 \in \R^*_+\), then a deterministic function \(\tilde{\mu}^T : [0, Tt_0] \to \R_+^*\), lipstchiz with bounded derivative, uniformly continuous on the interval \([0, t_0]\) such that \(\tilde{\mu}^T (0) = \frac{\mu^T(0)}{\theta_0} = \frac{\mu^T}{\theta_0}\) so that \(\forall t>0, \lim_{T\to +\infty} \frac{\mu^T}{\tilde{\mu}^T (\frac{t}{T})} = \theta_0\) and assume there exists some positive constants \( \lambda, \nu \), a non-negative bounded borel function \(\varsigma\) and a  continuous positive function $\phi \in {\cal L}^1([0,t_0],\text{Leb}_1)$ such that, uniformly as $ T \to \infty$:
		\[
		T^\alpha(1 - a_T) \to \lambda\delta; \quad \forall t >0,\hspace{.1cm}  T^{1 - \alpha} \tilde{\mu}^T (\frac{t}{T})  \to \frac{\lambda}{\nu^2\delta \varsigma^2(t)} \quad \text{and}  \quad T^\alpha\frac{ a_T(1 - a_T)}{\tilde{\mu}^T (\frac{t}{T})} \to \phi(t).
		\]
		so that we recover up to the scaling by \(\theta_0\), the assumption of \cite{JaissonRosenbaum2016,ElEuchR2018} provided \( \tilde{\mu}^T \equiv C^\textit{ste} = \frac{\mu^T(0)}{\theta_0} = \frac{\mu^T}{\theta_0}\)in which case \(\varsigma\equiv 1\).
		
	\medskip
	\noindent {\em (3).}  There exist a real-valued random variable \(\Lambda_0^{*} \in L^1(\R_+)\), weak limit of the sequence of positive real-valued random variable $(T^{- \alpha}\Lambda_0^{*T})_{T\geq0}$ as \(T \to +\infty \) i.e. \(\lim_{T\to +\infty} T^{- \alpha}\Lambda_0^{*T} = \Lambda_0^{*}.\)
\end{assumption}

 \noindent {\bf Remark and discussions on assumption \ref{assum:Lhawkes}.}
	\begin{enumerate}
		\item {\em Example of sequence:} The sequence $(a_T)_{T\geq 0}$ can be defined for example as folow: For some $\alpha\in(0,1)$ and positive constants \( \lambda \)  defined in \ref{assum:Lhawkes} (2) and (3) respectively, $a_T = 1 - \frac{\lambda}{T^\alpha}$ and $T > \lambda^{\frac1\alpha}$
		\item {\em Processes of Mittag-Leffler type:} A Hawkes process \( (N, \Lambda) \) is said to be of Mittag-Leffler type with index \( (\alpha, \lambda) \) if \( \varphi = f_{\alpha, \lambda} \), so that  $C = \frac{\alpha}{\lambda\Gamma(1-\alpha)}$ since $\mathit{\Phi}(x) \underset{x \to \infty}{\sim} \frac{1}{\lambda \, \Gamma(1 - \alpha)} \, x^{-\alpha}$ (See \cite{GorMain2000, Mainardi2014}). In particular, for the case  \( \varphi = f_{\alpha, 1} \), we have $C = \frac{\alpha}{\Gamma(1-\alpha)}$.

		A typical example of a completely monotone kernel $\varphi$ with heavy tail $\mathit{\Phi}(0)=\|\varphi\|_{L_1}=1$ and 
		\(\varphi(0)< \infty\) is a power-law kernel of the form
		\begin{equation}
		\varphi(t) = \frac{\alpha b \tau^\alpha}{(\tau + b t)^{1+\alpha}}, \qquad \text{for some positive constants } \alpha \in \left(\tfrac{1}{2},1\right)
		\text{ and } \tau, b > 0.
	    \end{equation}
		We will consider this kernel in this work to account for the empirically well-documented long-range dependencies in order arrivals.
		The function $\varphi$ is a probability density function on $\mathbb{R}_{+}$ with
		tail distribution \(\Phi(t) := \int_{t}^{\infty} \varphi(s)\,ds
		= \tau^\alpha (1 + b t)^{-\alpha},
		\qquad t \ge 0.\)
		\item Since, \(\tilde{\mu}^T (0)  = \frac{\mu^T}{\theta_0}\) and \(\tilde{\mu}^T(t)>0\) for all \(t>0\), we have that \(\tilde{\mu}^T\) is uniformly bounded away from zero, i.e., there exists \( c > 0 \) such that: \(\qquad	\inf_{T > 0} \inf_{s \in [0, \frac{t_0}{T}]} \tilde{\mu}^T(s) \geq c.\)
		\item A function \( \varphi : [0, \infty) \to \mathbb{R} \) is said to be \emph{completely monotone} if it is \( C^\infty \) and its derivatives 
		alternate signs, i.e.
		
		\centerline{$(-1)^n \varphi^{(n)}(t) \geq 0 
			\qquad \text{for all } n \geq 0 \text{ and } t > 0.$}
		Completely monotone functions have the particular property of being Laplace transforms of nonnegative measures \cite{Bernstein1929}. 
		More precisely, there exists a nonnegative measure \( \mu \) on \( [0, \infty) \) such that \(\varphi(t) = \int_{[0,\infty)} e^{-xt} \, \mu(dx).\)	
		Intuitively, completely monotone functions are a natural choice for financial modeling,
		since the function decreases smoothly without oscillatory behavior.
	\end{enumerate}

\medskip
\noindent Let us first establish the convergence or  asymptotic result for the integrated rescaled resolvent function that will be key to our analysis of the long-run behavior of Hawkes processes.
\begin{proposition}[The limits for the integrated rescaled resolvent function]\label{Lem:Lhawkes}
	Under Assumption \ref{assum:Lhawkes}, and denoting $\mathcal{L}(f) :=L_{f}$ the Laplace transform for any functional f, we have the following claims where  \( f_{\alpha, \lambda} \) is the Mittag-Leffler density function  defined in Example \ref{Ex:frackernel} and \(R_{\alpha, \lambda}(t)\) denotes the corresponding resolvent function.
	\begin{enumerate}
		\item $\mathcal{L}\left(\int_0^{t} T(1 - a_T) R^{\prime T}_{-1}(Tu) \, du \right)(z) \underset{T \to +\infty}{\rightarrow} \frac{1}{z} L_{f_{\alpha, \lambda}}(z)$.
		\item  The finite measure on \( \mathbb{R}_+ \) with density
		\(m^{T}(ds) := T(1 - a_T) R^{\prime T}_{-1}(Ts)\,ds =: f^T(s)\,ds\) converges weakly to the finite measure on \( \mathbb{R}_+ \) with density \(m^*(ds) := f_{\alpha,\lambda}(s)\,ds\).
		
		Consequently,
		$\int_0^{t} T(1 - a_T) R^{\prime T}_{-1}(Tu) \, du$
		converges uniformly towards
		$ \int_0^{t}f_{\alpha, \lambda}(u) \, du,$ as $T \to +\infty$ i.e we have the following convergence of the rescaled integrated resolvent function: 
		\[\sup_{t\geq0}| \int_0^{t} T(1 - a_T) R^{\prime T}_{-1}(Tu) \, du  - \int_0^{t}f_{\alpha, \lambda}(u) \, du|\underset{T \to +\infty}{\rightarrow} 0\]
		% $\qquad \sup_{t\geq0}| \int_0^{t} T(1 - a_T) R^{\prime T}_{-1}(Tu) \, du  - \int_0^{t}f_{\alpha, \lambda}(u) \, du|\underset{T \to +\infty}{\rightarrow} 0 $
	%	In particular, over and $ T(1 - a_T) R^{\prime T}_{-1}(Tu) \underset{T \to +\infty}{\rightarrow} f_{\alpha, \lambda}(u)$ 
	\end{enumerate}

\end{proposition}
\noindent We postpone the proof of the above proposition to Appendix \ref{app:lemmata}. We also state the following stability Result for the density \(
m^{T}(ds) \).
%		We also state the following stability lemma for the density \(m^{T}(ds) \).
\begin{proposition}[Stability result for the density \(
	m^{T}(ds) \)]\label{prop:stability}
	Consider a sequence \( (g^T)_{T \geq 0} \) of c\`adl\`ag functions on \([0, t_0]\), converging in \( L^1([0, t_0]) \) to a continuous function \( g \). Let the sequence \( (G^T)_{T \geq 0} \) be defined by
	\[
	G^T(t) := \int_0^{t} g^T(t - s) \, m^{T}(ds), \quad t \leq t_0, \ T \geq 0,
	\]
	where \( m^{T} \) is defined in the claim (2) of Proposition \ref{Lem:Lhawkes}.
	If \( G^T \to G \) in \( L^1([0, t_0]) \) with \( G \) continuous, then
	\(G(t) = \int_0^t g(t - s) m^*(ds) = \int_0^t g(t - s)f_{\alpha,\lambda}(s)\, ds, \quad t \leq t_0.\)
\end{proposition}

\begin{corollary}[Uniform almost sure convergences]\label{Corol:LimitMeasure}
		Let \( f_{\alpha, \lambda} \) be the Mittag-Leffler density function  defined in Example \ref{Ex:frackernel}, \(R_{\alpha, \lambda}(t)\) denotes the corresponding resolvent function and \(\tilde{\mu}^T\) defined in Assumption~\ref{assum:mu}~(2).
		\begin{enumerate}
			\item 
			For a deterministic function \( \theta \), continuous on \( \mathbb{R}_+^* \), we have:
			\[\sup_{t\geq0}\Big|\int_0^{t} T(1 - a_T) R^{\prime T}_{-1}(T(t-s))\theta(s) \, ds
			- \int_0^t f_{\alpha,\lambda}(t-s)\theta(s)ds\Big|
			\underset{T \to +\infty}{\rightarrow} 0.\]
			Equivalently, for all \( t > 0 \),
			\(\int_0^{t} T(1 - a_T) R^{\prime T}_{-1}(T(t-s))\theta(s) \, ds
			\underset{T \to +\infty}{\rightarrow} \int_0^t f_{\alpha,\lambda}(t-s)\theta(s)ds.\)
			
			\item 
			The temporal release \(\mathit{\Phi}^T\) satisfies the integral equation
			\(\mathit{\Phi}^T(t) + (R^{\prime T}_{-1}*\mathit{\Phi}^T)(t)
			= a_T - (1 -a_T) \int_0^t R^{\prime T}_{-1}(s) ds;\)
			so that, under time and space rescaling, this function converges as follows:
			\[\sup_{t\geq0}\Big| T^\alpha\frac{ (1 - a_T)}{\tilde{\mu}^T (\frac{t}{T})}
			\big(\mathit{\Phi}^T(tT) + (R^{\prime T}_{-1}*\mathit{\Phi}^T)(tT)\big)
			- \big(\phi(t)- \int_0^t f_{\alpha,\lambda}(t-s)\phi(s)ds\big)\Big|
			\underset{T \to +\infty}{\rightarrow} 0.\]
			Equivalently, for all \( t > 0 \),
			\(T^\alpha\frac{ (1 - a_T)}{\tilde{\mu}^T (\frac{t}{T})}
			\big(\mathit{\Phi}^T(tT) + (R^{\prime T}_{-1}*\mathit{\Phi}^T)(tT)\big)
			\underset{T \to +\infty}{\rightarrow}
			\phi(t)- \int_0^t f_{\alpha,\lambda}(t-s)\phi(s)ds.\)
			
			\item 
			Moreover, as \( T \to \infty \), we have the uniform almost sure convergence on \([0,t_0]\):
			\(\forall\, t \in [0,t_0]\), it holds that
			\((1 - a_T)\frac{\mu^T}{\tilde\mu^T(\frac{t}{T})}
			\underset{T \to +\infty}{\rightarrow} 0.\)
		\end{enumerate}
\end{corollary}
\noindent We must now proceed with a rescaling in time and a well-chosen renormalization in space in order to obtain a non-degenerate limit. First, we rescale the intensity in the %unit
time interval \( [0, t_0] \) by setting:
{\small 
	\[
	\Lambda^T_{tT} = \Lambda_0^{*T} \, (\mathit{\Phi}^T(tT) + (R^{\prime T}_{-1}*\mathit{\Phi}^T)(tT)) + \mu^T(tT) + \int_0^{tT} R^{\prime T}_{-1}(Tt - s) \mu^T(s) \, ds + \int_0^{Tt} R^{\prime T}_{-1}(Tt - s) \, dM^T_s, \quad \forall t \in [0, t_0].
	\]
}
\noindent Next, in terms of scaling in space, we assume that the inhomogeneous Poisson intensity \( \mu^T(\cdot) \) is of order \( \mu^T(0) = \mu^T \). A natural multiplicative factor is \( \omega_T = \frac{1 - a_T}{\mu^T} \) as in \cite{JaissonRosenbaum2016, ElEuchFukasawaRosenbaum2018}. This choice is motivated by the fact that, in the stationary regime, if we approximate the expected number of descendants \( a = \mathit{\Phi}(0) \) in the stability condition by \( a_T \) (this choice is legitimate since the sequence \( (a_T)_T \) converges to 1), the expectation of \( \Lambda^T \) (and thus the order of magnitude of the intensity \( \Lambda^T \)) is \( \frac{\mu^T}{1 - a_T} \) (see e.g. Proposition~\ref{prop:stat}~(2)).

\medskip
\noindent However, we aim at inducing a time-modulated volatility in the limiting diffusion, encoded via a deterministic non-negative bounded borel function \(\varsigma\). To this end, we slightly modify the rescaling term by introducing a time-dependent multiplicative factor \( \omega_T(t) = \frac{1 - a_T}{\tilde\mu^T(\frac{t}{T})} \) in the normalization for a well chosen function \(\tilde\mu^T\) (see. Assumption~\ref{assum:mu} below). 

\medskip
\noindent Intuitively, from a financial modeling standpoint, this approach is motivated by extensive empirical evidence of time-varying market activity. To account for local fluctuations in liquidity and trading intensity, we therefore introduce a time-dependent normalization factor that reflects the slowly varying average level of market activity. Define \( \Lambda^{*T}_t := \omega_T(t) \Lambda^T_{tT} \) so that after obvious computations:

%\begin{align*}
%	&\Lambda^{*T}_t =  \Lambda_0^{*T} \, (\mathit{\Phi}^T(tT) + (R^{\prime T}_{-1}*\mathit{\Phi}^T)(tT))\frac{1 - a_T}{\tilde\mu^T(\frac{t}{T})} + \frac{\mu^T}{\tilde\mu^T(\frac{t}{T})}\int_0^{t} T(1 - a_T) R^{\prime T}_{-1}\left(T(t - s)\right) \frac{\mu^T(Ts)}{\mu^T} \, ds \\
%	&\quad \hspace{3cm} + (1 - a_T)\frac{\mu^T(tT)}{\tilde\mu^T(\frac{t}{T})} + \int_0^{t}(1 - a_T)R^{\prime T}_{-1}\left(T(t - s)\right) \, dM^T_{Ts}, \quad \forall t \in [0, t_0]\\
%	&=\Lambda_0^{*T} \, (\mathit{\Phi}^T(tT) + (R^{\prime T}_{-1}*\mathit{\Phi}^T)(tT))\frac{1 - a_T}{\tilde\mu^T(\frac{t}{T})} + \frac{\mu^T}{\tilde\mu^T(\frac{t}{T})}\int_0^{t} T(1 - a_T) R^{\prime T}_{-1}\left(T(t - s)\right) \frac{\mu^T(Ts)}{\mu^T} \, ds \\
%	&\quad \hspace{3cm}+(1 - a_T)\frac{\mu^T(tT)}{\tilde\mu^T(\frac{t}{T})}  + \int_0^{t}T(1-a_T)  R^{\prime T}_{-1}\left(T(t - s)\right) \, \sqrt{\Lambda^{*T}_s }\sqrt{\frac{\tilde\mu^T(\frac{s}{T})}{T(1 - a_T)}}\frac{dM^T_{Ts}}{\sqrt{T\Lambda^T_{sT}}},
%\end{align*}
\begin{align*}	&\Lambda^{*T}_t =  \Lambda_0^{*T} \, (\mathit{\Phi}^T(tT) + (R^{\prime T}_{-1}*\mathit{\Phi}^T)(tT))\frac{1 - a_T}{\tilde\mu^T(\frac{t}{T})} + \frac{\mu^T}{\tilde\mu^T(\frac{t}{T})}\int_0^{t} T(1 - a_T) R^{\prime T}_{-1}\left(T(t - s)\right) \frac{\mu^T(Ts)}{\mu^T} \, ds \\
	&\quad \hspace{3cm} + (1 - a_T)\frac{\mu^T(tT)}{\tilde\mu^T(\frac{t}{T})} + \int_0^{t}(1 - a_T)R^{\prime T}_{-1}\left(T(t - s)\right) \, dM^T_{Ts}, \quad \forall t \in [0, t_0]\\
	&=\Lambda_0^{*T} \, (\mathit{\Phi}^T(tT) + (R^{\prime T}_{-1}*\mathit{\Phi}^T)(tT))\frac{1 - a_T}{\tilde\mu^T(\frac{t}{T})} + \frac{\mu^T}{\tilde\mu^T(\frac{t}{T})}\int_0^{t} T(1 - a_T) R^{\prime T}_{-1}\left(T(t - s)\right) \frac{\mu^T(Ts)}{\mu^T} \, ds \\
	&\quad \hspace{3cm}+(1 - a_T)\frac{\mu^T(tT)}{\tilde\mu^T(\frac{t}{T})}  + \int_0^{t}T(1-a_T)  R^{\prime T}_{-1}\left(T(t - s)\right) \, \sqrt{\Lambda^{*T}_s }\sqrt{\frac{\tilde\mu^T(\frac{s}{T})}{T(1 - a_T)}}\frac{dM^T_{Ts}}{\sqrt{T\Lambda^T_{sT}}}, \\
	&= \Lambda_0^{*T} \, (\mathit{\Phi}(tT) + (R^{\prime T}_{-1}*\mathit{\Phi})(tT))\frac{a_T(1 - a_T)}{\tilde\mu^T(\frac{t}{T})}+ \frac{\mu^T}{\tilde\mu^T(\frac{t}{T})}\int_0^{t} T(1 - a_T) R^{\prime T}_{-1}\left(T(t - s)\right) \frac{\mu^T(Ts)}{\mu^T} \, ds \\
	&\quad \hspace{2.7cm} + (1 - a_T)\frac{\mu^T(tT)}{\tilde\mu^T(\frac{t}{T})}  +  \int_0^{t}\underbrace{\sqrt{\frac{\tilde\mu^T(\frac{s}{T})}{T(1 - a_T)(\tilde\mu^T(\frac{t}{T}))^2}}}_{\to \nu\frac{ \varsigma(s)}{\lambda}  } \underbrace{ T(1 - a_T)R^{\prime T}_{-1}\left(T(t - s)\right)}_{f^{T}(t-s) \to  f_{\alpha,\lambda}(t-s)} \, \sqrt{\Lambda^{*T}_s }dW^T_{s}
\end{align*}
where we set, $W^T_{t} := \frac{1}{\sqrt{T}}\int_0^{tT} \frac{dM^T_{s}}{\sqrt{\Lambda^T_{s}}}, \quad \forall t \in [0, t_0].$
The sequence of processes \( (W^T_\cdot)_T \) was selected specifically because its corresponding sequence of quadratic variations converges to the identity. In fact, computing the quadratic variation or bracket of $W^T$ give: \( \langle W^T \rangle_t = \frac{1}{T}\int_0^{tT} \frac{d\langle M^T \rangle_{s}}{\Lambda^T_{s}} = \frac{1}{T}\int_0^{tT} ds =t \). As a result, by Levy's characterization of Brownian motion \( W^T \) converges to a Brownian motion \(W\).
\begin{assumption}\label{assum:mu}
	\begin{enumerate}
		\item There exists a deterministic function \( \theta \), continuous on \( \mathbb{R}_+^* \), satisfying
		\begin{equation}\label{eq:theta-lower}
			\forall t > 0, \quad \theta (t) \geq \theta_0\left(1 - \frac{1}{\int_0^t \varphi^T(s) \, ds}\right), \; \text{and}\; \exists K > 0 \text{ such that } \forall u \in (0, t_0], \quad \theta (u) \leq K
		\end{equation}
		\item The baseline intensity $ \mu^T(t)$ is given by
		$\mu^T(t) = \mu^T \zeta^T(t)$ with $\zeta^T(t) =   1 -\int_0^t \varphi^T(t-u)\big(1-\frac{\theta(\frac{u}{T}) }{\theta_0}\big)  du.$
	\end{enumerate} 
\end{assumption}
 \noindent {\bf Remark and discussions on assumption \ref{assum:mu}.}
 
	\medskip
		\noindent $1.$ Assumption \ref{assum:mu} (1) on $\mu^T(t)$ guarantees that $\mu^T(t)$ is a positive function, ensuring that the intensity process $\Lambda_t^T$ defined in \eqref{eq:intensity} is well-posed. In fact, set \(\bar\theta(\frac{\cdot}{T})= \frac{\theta(\frac{\cdot}{T}) }{\theta_0}\) and assume that for some \(M>0\), \( \; 1-\bar\theta(t) < M \) \(\forall t >0\),  then we have
	%	\centerline{$\int_0^t \varphi^T(t - u)(1 - \bar\theta(\frac{u}{T})) \, du < M \int_0^t \varphi^T(t - u) \, du = M \int_0^t \varphi^T(s) \, ds.$}
		\(	\int_0^t \varphi^T(t - u)(1 - \bar\theta(\frac{u}{T})) \, du < M \int_0^t \varphi^T(t - u) \, du = M \int_0^t \varphi^T(s) \, ds.\)
		Consequently, \(\zeta^T(t) > 1 - M \int_0^t \varphi^T(s) \, ds.\)
		In order to ensure that $\zeta^T(t) > 0$, it suffices that
	
		$\hspace{2cm}\qquad 1 - M \int_0^t \varphi^T(s) \, ds > 0 \quad \text{which is equivalent to}\quad M < \frac{1}{\int_0^t \varphi^T(s) \, ds}.$
		
		\medskip
		\noindent Consequently, \(\forall t > 0, \quad \bar\theta(t) > 1 - \frac{1}{\int_0^t \varphi^T(s) \, ds} = -\frac{\mathit{\Phi}^T(t)}{1-\mathit{\Phi}^T(t)}> - \mathit{\Phi}(t) \equiv - K t^{-\alpha}.\) For a detailed explanation of why this particular function was chosen, the reader is referred to Appendix \ref{app:choose_mu}.
		
	\medskip
	\medskip
	\noindent $2.$ The upper bound in equation~\eqref{eq:theta-lower}, assures that there exists a constant C s.t. $\forall t > 0, \quad \zeta^T(t) \leq C.$
\medskip
\noindent
With this non-homogenous baseline intensity $\mu^T(t)$ satisfying Assumption \ref{assum:mu}, the time-modulated renormalized intensity becomes:
\begin{align}\nonumber
	&\ \Lambda^{*T}_t = \Lambda_0^{*T} \, (\mathit{\Phi}^T(tT) + (R^{\prime T}_{-1}*\mathit{\Phi}^T)(tT))\frac{1 - a_T}{\tilde\mu^T(\frac{t}{T})} + \frac{\mu^T}{\theta_0\tilde\mu^T(\frac{t}{T})}\int_0^{t} T(1 - a_T) R^{\prime T}_{-1}\left(T(t - s)\right) \theta(s) \, ds \\
	&\quad \hspace{3cm} + (1 - a_T)\frac{\mu^T}{\tilde\mu^T(\frac{t}{T})} + \int_0^{t}(1 - a_T)R^{\prime T}_{-1}\left(T(t - s)\right) \, dM^T_{Ts}, \quad \forall t \in [0, t_0] \label{eq:form1}
	\\ \nonumber
	 &\hspace{.5cm}=  \Lambda_0^{*T} \, (\mathit{\Phi}(tT) + (R^{\prime T}_{-1}*\mathit{\Phi})(tT))\frac{a_T(1 - a_T)}{\tilde\mu^T(\frac{t}{T})}  + \frac{\mu^T}{\theta_0\tilde\mu^T(\frac{t}{T})}\int_0^{t} T(1 - a_T) R^{\prime T}_{-1}\left(T(t - s)\right) \theta(s) \, ds \\ 
	&\quad \hspace{1 cm}+ (1 - a_T)\frac{\mu^T}{\tilde\mu^T(\frac{t}{T})} +  \int_0^{t}\underbrace{\sqrt{\frac{\tilde\mu^T(\frac{s}{T})}{T(1 - a_T)(\tilde\mu^T(\frac{t}{T}))^2}}}_{\to \nu\frac{ \varsigma(s)}{\lambda}  } \underbrace{ T(1 - a_T)R^{\prime T}_{-1}\left(T(t - s)\right)}_{f^{T}(t-s) \to f_{\alpha,\lambda}(t-s)} \, \sqrt{\Lambda^{*T}_s }dW^T_{s} \label{eq:form2}
\end{align}
From the above equation~\eqref{eq:form2}, we observe that the asymptotic behavior of the rescaled intensity $\Lambda^{*T}_t$ is closely tied to that of \( u \to T(1 - a_T) R^{\prime T}_{-1}(Tu) \), whose limiting behavior is provided by Proposition \ref{Lem:Lhawkes}.
Therefore, taking the limit as \( T \to \infty \) and owing to Assumption \ref{assum:Lhawkes} and Assumption \ref{assum:mu} we expect the limiting process \( \Lambda^{*}_t \) to be at least heuristically the solution of the following stochastic volterra integral equation:
\vspace{-0.4cm}
\begin{equation}\label{eq:limitInt} 
	\Lambda^{*}_t= \Lambda^{*}_0\big(\phi(t)- \int_0^t f_{\alpha,\lambda}(t-s)\phi(s)ds\big) + \int_0^t f_{\alpha, \lambda}(t-s)\theta(s)ds + \frac{\nu}{\lambda}\int_0^t f_{\alpha, \lambda}(t-s)\varsigma(s) \sqrt{\Lambda^{*}_s}dW_s.
\end{equation}
	As a straightforward consequence of Proposition~\ref{prop:wiener_hopf} in the particular case of the \( \alpha-\)fractional kernel \( K(t) = K_{\alpha}(t)\), any solution to the stochastic Volterra equation~\eqref{eq:limitInt} can be equivalently represented as
	\vspace{-0.3cm}
	\begin{equation}\label{eq:limitIntAffine2}
		\Lambda^{*}_t = \Lambda^{*}_0\phi(t) +\int_0^t K_{\alpha}(t-s)\lambda(\theta(s)- \Lambda^{*}_s)ds + \nu \int_0^t K_{\alpha}(t-s)\varsigma(s) \sqrt{\Lambda^{*}_s}dW_s, \quad \Lambda^{*}_0\perp\!\!\!\perp W,\quad t \geq 0.
	\end{equation}
	where $K_{\alpha}:= u\to \frac{u^{\alpha - 1}}{\Gamma(\alpha)} \mathbf{1}_{\mathbb{R}_+}(t), \;  \alpha \in (\frac12,1)$ is the \( \alpha-\)fractional integration kernel.\\

%\noindent {\bf Proof:} This is a direct consequence of Proposition \ref{prop:wiener_hopf} for the particular case of \( \alpha-\)fractional integration kernel \( K(t) = K_{\alpha}(t), \;  \alpha \in (\frac12,1)\).

\noindent {\bf Remark:} {\em Convergence of rescaled stationary Hawkes processes to the inhomogeneous Volterra CIR Equation.}
	By construction, when, $\theta = C^{\text{ste}} = \theta_0$, the baseline intensity becomes constant over time, $\mu^T(t) = C^{\text{ste}} = \mu^T$ for all $t > 0$. Consequently, as detailed in Proposition \ref{prop:stat}, Section \ref{subsec:PrelimStationary}, the resulting Hawkes process exhibits a constant mean. In this scenario, the limiting equation precisely corresponds to an inhomogeneous Volterra Cox-Ingersoll-Ross (CIR) process, which can be expressed as:
	\vspace{-0.3cm}
	\begin{equation}\label{eq:frac_CIR1}
		\Lambda^{*}_t= \Lambda^{*}_0\big(\phi(t)- \int_0^t f_{\alpha,\lambda}(t-s)\phi(s)ds\big) + \theta_0 (1- R_{\alpha, \lambda}(t) )+ \frac{\nu}{\lambda}\int_0^t f_{\alpha, \lambda}(t-s)\varsigma(s) \sqrt{\Lambda^{*}_s}dW_s.
	\end{equation}
	Or equivalenty,  still owing to the Equation ~\eqref{eq:limitIntAffine2} above:
	\begin{equation}\label{eq:frac_CIR2}
	\Lambda^{*}_t = \Lambda^{*}_0 \phi(t) + \frac{1}{\Gamma(\alpha)} \int_0^t (t - s)^{\alpha - 1} \lambda(\theta_0 - \Lambda^{*}_s) \, ds + \frac{\nu}{\Gamma(\alpha)} \int_0^t (t - s)^{\alpha - 1} \varsigma(s) \sqrt{\Lambda^{*}_s} \, dW_s.
    \end{equation}
	Thus, we conclude that a well-normalized sequence of stationary (at least constant mean) Hawkes processes converges to a time-inhomogeneous rough fractional CIR process. Notably, if the function \( \varsigma \equiv 1 \), this formulation aligns with results found in \cite{el2019characteristic, HorstXuZhang2024}.
	
\vspace{-0.3cm}
\subsection{A priori estimates and $\mathcal C$-tightness}

 We aim to show that the sequence of rescaled intensity processes \( \left\{ \Lambda^{*T}_\cdot  : T \in \mathbb{R}^+ \right\} \) converges weakly to a  diffusion denoted \( \Lambda^{*}:=\left\{ \Lambda^{*}_t  : t \geq 0 \right\} \) and refered to as time-dependent stochastic Volterra CIR process, and the point process \( N^T \) behaves asymptotically as the integrated diffusion process of \(\Lambda^{*}\) for large \( T \in \mathbb{R}^+ \), when suitably rescaled. We will need the \(\mathcal{C}\)-tightness of the sequence \( (\Lambda^{*T})_{T > 0} \) and \( (\tilde{\mathcal{I}}^{\Lambda^{*T}})_{T > 0} \)
 
  \medskip
  \noindent We introduce the family of rescaled processes:
  \(\; X^T := \left\{ X^T_t =\left( \tilde{\mathcal{I}}_t^{\Lambda^{*T}}, \tilde{N}^{*T}_t, \tilde{M}^{*T}_t,  \tilde{M}_t \right) : t \geq 0 \right\}\) where \( \forall t > 0\):
	{\small
	\[\tilde{N}^{*T}_t =\frac{(1 - a_T)}{ T} \int_{0}^{tT} \frac{1}{\tilde\mu^T(\frac{s}{T^2})}dN^T_{s}, \quad \tilde{\mathcal{I}}_t^{\Lambda^{*T}} = \frac{(1 - a_T)}{ T}\int_0^{tT} \frac{1}{\tilde\mu^T(\frac{s}{T^2})}\Lambda^T_s ds, \quad %\textit{and} \quad \tilde{M}^T_t % =\sqrt{\frac{T\mu^T}{(1 - a_T)}}(\tilde{N}^T_t-\tilde{\mathcal{I}}_t^{\Lambda^T}) 
	\tilde{M}^{*T}_t = \sqrt{\frac{1 - a_T}{T}}\int_{0}^{tT} \frac{dM^T_{s}}{\sqrt{\tilde\mu^T(\frac{s}{T^2})}}
	\]
    }  
    \noindent and \( \tilde M^{T}_t := \int_{0}^{t} \varsigma^T(s) \, d\tilde M^{*T}_s
    \) where  we set
    \( \varsigma^T(t) := \frac{\lambda}{\nu}\sqrt{\frac{1}{T(1 - a_T)\tilde\mu^T\left(\frac{t}{T}\right)}} \)  so that \(\tilde M^{T}_t =\frac{\lambda}{\nu T}\int_{0}^{tT} \frac{dM^T_{s}}{\tilde\mu^T(\frac{s}{T^2})}.\)
  
\medskip
\noindent In this section we establish the $ \mathcal C$-tightness\footnote{In this framework, a sequence of processes is said to be $\mathcal{C}$-tight if it does not escape to infinity and all of its possible weak limits have continuous trajectories. This property is crucial for establishing convergence to a regular limiting dynamic.
} of the sequence of rescaled  processes $X^T $ (i.e. tightness for the Skorokhod \(J_1\)
topology), with continuous limits and hence the
existence of a weak continuous accumulation point (Prokhorov's theorem). Before establishing the \( \mathcal{C} \)-tightness, note that \(\tilde{\mathcal{I}}^{\Lambda^{*T}}_t\) is obtained by integrating both side of equation~\eqref{eq:form1} over \([0, t]\) which yield:
\(\tilde{\mathcal{I}}^{\Lambda^{*T}}_t = \mathcal{I}^{I^T}_t + \mathcal{I}^{J^T}_t + \mathcal{I}^{K^T}_t \) with	
{\small 
	\begin{align*}
		&\, \mathcal{I}^{I^T}_t = \int_0^t \Lambda_0^{*T} \left(\Phi^T(sT) + (R^{\prime T}_{-1}*\Phi^T)(sT)\right) \frac{1 - a_T}{\tilde\mu^T(\frac{s}{T})} \, ds + \int_0^t (1 - a_T)\frac{\mu^T(sT)}{\tilde\mu^T(\frac{s}{T})} \, ds =\int_0^t (1 - a_T)\frac{\mu^T(sT)}{\theta_0\tilde\mu^T(\frac{s}{T})} \, ds  \\
		&\hspace{3cm}+ T^{-\alpha}\Lambda_0^{*T} \left(\int_0^t  T^{\alpha}\frac{a_T(1 - a_T)}{\tilde\mu^T(\frac{s}{T})}ds  - \int_0^t T^{\alpha}\frac{(1 - a_T)}{\tilde\mu^T(\frac{s}{T})} \int_0^s T (1 -a_T)R^{\prime T}_{-1}(T(s-u)) du ds\right); \\	
		& \mathcal{I}^{J^T}_t= \int_0^t \frac{\mu^T}{\theta_0\tilde\mu^T(\frac{s}{T})}  \int_0^s T(1 - a_T) R^{\prime T}_{-1}(T(s - u)) \theta(u) \, du  ds \;, \; \mathcal{I}^{K^T}_t= \int_0^t \frac{\mu^T}{\tilde\mu^T(\frac{s}{T})}  \int_0^s \frac{1 - a_T}{\mu^T} R^{\prime T}_{-1}(T(s - u)) \, dM^T_{Tu}  ds
	\end{align*}
}

\noindent For the term \(K^T_t\), we have the following decomposition (up to the multiplicative factor \(1 - a_T\)):
{\small
	\begin{align*}
		\frac{1}{\tilde\mu^T(\frac{t}{T})}  \int_0^t  R^{\prime T}_{-1}(T(t - s)) \, dM^T_{Ts} = \underbrace{ \int_0^t R^{\prime T}_{-1}(T(t - s))\big(\frac{1}{\tilde\mu^T(\frac{t}{T})}-\frac{1}{\tilde\mu^T(\frac{s}{T})}\big) \, dM^T_{Ts}}_{K^T_{1,t}} + \underbrace{\int_0^t  R^{\prime T}_{-1}(T(t - s))\frac{dM^T_{Ts}}{\tilde\mu^T(\frac{s}{T})} }_{K^T_{2,t}}.
	\end{align*}
}
\noindent so that  \(\mathcal{I}^{K^T}_t= \mathcal{I}^{K^T_1}_t + \mathcal{I}^{K^T_2}_t, \) 	with \(\mathcal{I}^{K^T_1}_t = \int_0^t (1 - a_T) \left[ \int_0^s  R^{\prime T}_{-1}(T(s - u)) \, \big(\frac{1}{\tilde\mu^T(\frac{s}{T})}-\frac{1}{\tilde\mu^T(\frac{u}{T})}\big) \, dM^T_{Tu} \,\right] ds\) and \(\mathcal{I}^{K^T_2}_t = \int_0^t \int_0^s (1 - a_T)  R^{\prime T}_{-1}(T(s - u))\frac{dM^T_{Tu}}{\tilde\mu^T(\frac{u}{T})}\,ds.\) We have the following uniform convergence on \([0,t_0]\):

\begin{proposition}\label{prop:uniformCvgce} As \( T \to \infty \), we have the uniform almost sure convergences on \([0,t_0]\): \(\quad \mathcal{I}^{K^T_1}_t \underset{T \to +\infty}{\rightarrow} 0\)
		
		\centerline{$
			\mathcal{I}^{I^{T}}_t \underset{T \to +\infty}{\rightarrow} 
			\int_0^t \Lambda^*_0 \left(\phi(s)- \int_0^s f_{\alpha,\lambda}(s-u)\,\phi(u)\,du \right)ds,\qquad
			\mathcal{I}^{J^{T}}_t \underset{T \to +\infty}{\rightarrow} 
			\int_0^t \int_0^s f_{\alpha,\lambda}(s-u)\,\theta(u)\,du\,ds.
			$}
		In particular, the sequence \(\{\mathcal{I}^{I^{T}} + \mathcal{I}^{J^{T}} + \mathcal{I}^{K^T_1}\}_{T>0}\) is \(\mathcal{C}\)-tight. % in \(\mathcal{C}([0,t_0])\).
\end{proposition}
\noindent {\bf Proof:}
The first convergence follows from the claim 2 of Lemma \ref{Lm:cvgcezero}. The last two claims are straightforward owing to Corollary \ref{Corol:LimitMeasure} and Assumption \ref{assum:Lhawkes} (1). The \(\mathcal{C}\)-tightness is straightforward thanks to \cite[Corollary~VI.3.33(1), p.353]{jacod2013limit}.
% Corollary~VI.3.33 (1) in \cite[p.353]{jacod2013limit}
\hfill $\Box$

\smallskip
\noindent In the proposition below, we provide a priori estimates that will be important to our subsequent analysis, particularly in proving the tightness of our sequence of processes.
\begin{proposition}[A priori estimates]\label{prop:estim}
	There exist constants \( C_1, C_2 > 0 \) such that both the renormalized Hawkes process 
	\(\tilde{N}^{*T}\) and the integrated renormalized intensity process \(\tilde{\mathcal{I}}^{\Lambda^{*T}}\) satisfy, for any \( t \geq 0 \),
	\[\sup_{T>0} \mathbb{E}\bigl[\tilde{N}^{*T}_t\bigr] 
	= \sup_{T>0} \mathbb{E}\bigl[\tilde{\mathcal{I}}^{\Lambda^{*T}}_t\bigr] 
	< C_1.\]
	\noindent Moreover, the stochastic integral terms \(\tilde{M}^{*T}\) and \(\tilde{M}^{T}\) satisfy the moment bound 
	\[\sup_{T>0} \mathbb{E}\Bigl[\sup_{t \leq t_0} \bigl| \tilde{M}^{*T}_t \bigr|\Bigr] + \sup_{T>0} \mathbb{E}\Bigl[\sup_{t \leq t_0} \bigl| \tilde{M}^{T}_t \bigr|\Bigr] 
	\leq C_2.\]
\end{proposition}
\noindent {\bf Proof:} Recalling equation~\eqref{eq:IntLam}, we have:

	 \(\qquad\hspace{3cm}\mathbb{E}[N^T_t] = \mathbb{E}[\mathcal{I}^{\Lambda^{T}}_t] 
= \left(1 + (\mathbf{1} * R^{\prime T}_{-1})_t \right) 
\int_0^t \left( \mathbb{E}[\Lambda_0^{*T}]\, \Phi^T(s) + \mu^T(s) \right) ds.\)\\
Futhermore, \(\; \forall t>0,\)
\[\mathbb{E}[\Lambda_0^{*T}]\, \Phi^T(t)+ \mu^T(t)= \mathbb{E}[\Lambda_0^{*T}]\, \mathit{\Phi}^T(t) +  \mu^T \zeta^T(t)  \leq K_1 a_T\mathit{\Phi}(0) +  K_2 \zeta^T(t) \leq K_1\mathit{\Phi}(0) +  K_2C=:K \]
where the first inequality comes from the fact that by assumption \(\mathbb{E}[\Lambda_0^{*T}]\) and \(\mu^T\)  are bounded for every \( T>0\), and the second follows from assumption \ref{assum:Lhawkes} (1) and the claim 2 in the Remark on Assumption \ref{assum:mu}.	Consequently, using that
$ \qquad\; 1 + (\mbox{\bf 1}*R^\prime_{-1})_t = \sum_{k\ge 0} (\mbox{\bf 1}*\varphi^{T*k})_t \leq \sum_{k \geq 0} \int_0^\infty (\varphi^T)^{*k} = \sum_{k \geq 0} (a_T)^k = \frac{1}{1-a_T},$	we obtain:
\begin{equation}\label{eq:bounds_I}
	\mathbb{E}[N_t]= \mathbb{E}[\mathcal{I}^{\Lambda^{T}}_t] \leq \frac{1}{1-a_T} \int_0^{t}\left( \mathbb{E}[\Lambda_0^{*T}]\, \Phi^T(s) + \mu^T(s) \right) ds \leq \frac{K}{1-a_T} t.
\end{equation}
Therefore since $\tilde\mu^T$ is uniformly bounded on \([0,t_0]\) (see the  the Remark (3) on assumption \ref{assum:Lhawkes}), we have some constant \(C_1\) such that
$$\mathbb{E}[\tilde{N}^{*T}_t]= \mathbb{E}[\tilde{\mathcal{I}}^{\Lambda^{*T}}_t]  \leq \frac{1-a_T}{T} \int_{0}^{tT} \frac{1}{\tilde\mu^T(\frac{s}{T^2})} \mathbb{E}\left[\Lambda^T_{s} \right] ds \leq \frac{1-a_T}{T c} \mathbb{E}[\mathcal{I}^{\Lambda^{T}}_{tT}]
= \frac{K }{c} t\leq \frac{K }{c} t_0 =: C_1.$$
Thus, $ \sup_{T>0} \mathbb{E}[\tilde{N}^{*T}_t]= \sup_{T>0} \mathbb{E}[\tilde{\mathcal{I}}^{\Lambda^{*T}}_t] <C_1.$ 
Moreover, since $\langle \tilde{M}^{*T}\rangle_t = \tilde{\mathcal{I}}^{\Lambda^{*T}}_t,$
the Burkholder--Davis--Gundy (BDG) inequality or the Doob’s maximal inequality ensures that
\[
\mathbb{E}\Bigl[\sup_{t \leq t_0} 
\bigl| \tilde{M}^{*T}_t \bigr|^2 \Bigr]
\leq 4\, \mathbb{E}\bigl[\langle \tilde{M}^{*T} \rangle_{t_0}\bigr]
= 4\, \mathbb{E}\bigl[\tilde{\mathcal{I}}^{\Lambda^{*T}}_{t_0}\bigr] \leq 4\,C_1 := C_1^* .\]
\noindent so that by Cauchy–Schwarz inequality, \(\mathbb{E}\Bigl[\sup_{t \leq t_0} | \tilde{M}^{*T}_t |\Bigr] \leq \sqrt{C_1^*}\).

\noindent Likewise, by Doob’s maximal inequality, and using the fact that the sequence 
$(\varsigma^T)_{T>0}$ is uniformly bounded on $[0,t_0]$ (which follows from 
Assumption~\ref{assum:Lhawkes} and Remark~(3) stated therein on the uniform boundedness of 
$\tilde\mu^T$ away from zero), we have
\[
\mathbb{E}\Bigl[\sup_{t \leq t_0} 
\bigl| \tilde{M}^{T}_t \bigr|^2 \Bigr]
\leq 4\, \mathbb{E}\bigl[\langle \tilde{M}^{T} \rangle_{t_0}\bigr]
= 4\, \mathbb{E}\bigl[\int_{0}^{t_0} (\varsigma^T)^2(s)\,d\tilde{\mathcal{I}}^{\Lambda^{*T}}_{s}\bigr] \leq 4\,\left(\frac{1}{T\,(1 - a_T)\,c}\right)\,\mathbb{E}\bigl[\tilde{\mathcal{I}}^{\Lambda^{*T}}_{t_0}\bigr] \leq  4\,C_\varsigma\, C_1 := C_2^* .\]
\noindent so that by Cauchy–Schwarz inequality, \(\mathbb{E}\Bigl[\sup_{t \leq t_0} | \tilde{M}^{T}_t |\Bigr] \leq \sqrt{C_2^*}\). 

\noindent Setting $C_2 := \sqrt{C_1^*} + \sqrt{C_2^*}$ then provides the desired uniform $L^1$-bound on both martingale terms. \hfill $\Box$

\medskip
\noindent We now turn to the main result of this section: establishing the tightness of our sequence of processes in the Skorokhod topology. By $\mathcal C$-tightness\footnote{A tight sequence of processes is called $\mathcal C$-tight if any weak accumulation point is continuous.}, we will be refering to tightness with almost surely continuous limiting processes; see \cite[Definition VI.3.25, p.351]{jacod2013limit}.

\begin{proposition}[$\mathcal{C}$-tightness of the process \( X^T = \left( \tilde{\mathcal{I}}^{\Lambda^{*T}}, \tilde{N}^{*T}, \tilde{M}^{*T}, \tilde{M}^{T} \right) \)]\label{prop:Lhawkes_tightness}
	Let \( t_0 > 0 \). Under Assumption~\ref{assum:Lhawkes}, the sequence \( X^T = \left( \tilde{\mathcal{I}}^{\Lambda^{*T}}, \tilde{N}^{*T}, \tilde{M}^{*T}, \tilde{M}^{T} \right) \) satisfies the following:
	\begin{enumerate}
		\item The sequence \( (X^T)_{T > 0} \) is $\mathcal{C}$-tight in the Skorokhod topology on \( \mathcal{D}([0,t_0], \mathbb{R}^2_+\times\mathbb{R}^2) \). As a result, any limiting process lies in \( \mathcal{C}([0,t_0], \mathbb{R}^2_+\times\mathbb{R}^2) \), i.e., it has continuous sample paths.
		
		\item The processes \( \tilde{\mathcal{I}}^{\Lambda^{*T}} \) and \( \tilde{N}^{*T} \) become asymptotically indistinguishable in probability:
		
		\(\qquad \hspace{4cm} \sup_{t \in [0, t_0]} \left| \tilde{\mathcal{I}}_t^{\Lambda^{*T}} - \tilde{N}^{*T}_t \right|  \underset{T \to +\infty}{\overset{\mathbb{P}}{\rightarrow}} 0.\)
		\item Futhermore let \( X = \left( \tilde{\mathcal{I}}^{\Lambda^*}, \tilde{N}^{*}, \tilde{M}^{*}, \tilde{M} \right) \) be any limit point of the sequence \( (X^T)_{T > 0} \), that is, \(X^T \overset{\mathcal L }{\Rightarrow} X\) in \(D([0,t_0], \mathbb{R}^2_+\times\mathbb{R}^2)\). Then \( \tilde{M}^{*} \) is a continuous martingale with quadratic variation \( \langle \tilde{M}^{*} \rangle = \tilde{N}^{*} = \tilde{\mathcal{I}}^{\Lambda^*} \). Moreover
		the sequence \( (\tilde M^{T})_{T > 0} \) weakly converges along a subsequence in \(\mathcal{C}([0,t_0],\R)\) to the continuous martingale \(\tilde M_t := \int_{0}^{t} \varsigma(s) \, d\tilde M^{*}_s\), whose quadratic variation is given by \( \langle \tilde{M} \rangle = \int_{0}^{\cdot} \varsigma^2(s) d\langle \tilde{M}^{*} \rangle_s = \tilde{\mathcal{I}}^{\varsigma^2\Lambda^*} \). 
	\end{enumerate}
\end{proposition}
\noindent The weak convergence of the sequence of integrated processes \( (\tilde{\mathcal{I}}^{\Lambda^{*T}})_{T > 0} \) does in general not
imply the convergence of the processes \((\Lambda^{*T})_{T > 0} \). 
When the initial condition  of the Hawkes process is not zero, without the \(\mathcal{C}\)-tightness of the rescaled intensity processes \( (\Lambda^{*T})_{T > 0} \), we can in general not expect to identify its weak limit by identifying the weak limit of the sequence of integrated processes \( (\tilde{\mathcal{I}}^{\Lambda^{*T}})_{T > 0} \) and proceeding by differentiation. The main issue as shown in \cite[Example 3.1]{HorstXuZhang2024} is the convergence of the initial state of the volatility process that cannot
always be inferred from the convergence of the integrated processes. This calls for a convergence result for the rescaled intensity process \( (\Lambda^{*T})_{T > 0} \) that we establish in the following.
 \begin{proposition}[\(\mathcal{C}\)-tightness of \( (\Lambda^{*T})_{T > 0} \)]\label{prop:tightIntensity}
	The sequence \( (\Lambda^{*T})_{T > 0} \) is uniformly integrable and \(\mathcal{C}\)-tight in the Skorokhod topology on \( \mathcal{D}([0,t_0], \mathbb{R}_+) \). As a result, any limiting process lies in \( \mathcal{C}([0,t_0], \mathbb{R}_+) \), i.e., it has continuous sample path.
\end{proposition}
 \noindent By Proposition~\ref{prop:Lhawkes_tightness}, %the sequence
  {\small $\left( \tilde{\mathcal{I}}^{\Lambda^{*T}}, \tilde{N}^{*T}, \tilde{M}^{*T}, \tilde{M}^{T} \right)_{T \geq 0}$} is tight. Hence, by Prokhorov's theorem, it admits a weakly convergent subsequence. The next section characterizes the corresponding accumulation points.

\section{Dynamics of the limit points and regularity}

We state the results about the dynamics of the limit points i.e. the weak limit of the sequence of integrated rescaled Hawkes processes and then differentiate that limit to  identify the weak limit of the rescaled  nearly unstable Hawkes processes, while also establishing the regularity properties of the limit points.

\subsection{Characterization of the limiting processes and moment control}

Thanks to Prokhorov's theorem, (see e.g. \cite[Theorem~2.4.7]{karatzas1991}), % we get relative compactness (\cite[Definition~2.4.6]{karatzas1991}) of the sequence of measure \( \big( \mathbb{P}_{X^T} \big)_{T \geq 0} \) in \(C([0,t_0];\R^3)\). 
the joint tightness of the sequence \( \big( X^T \big)_{T \geq 0} \) in \(D([0,t_0], \mathbb{R}^4)\) established in proposition \ref{prop:Lhawkes_tightness} ensures that there exists a converging subsequence \( \big( X^{T_k} \big)_{k \in \N^* } \) in \((D([0,t_0], \mathbb{R}^4), \text{J}_1(\mathbb{R}^4))\), i.e. the existence of a weakly convergent (in law) subsequence of \( \big( X^T\big)_{T \geq 0}:= \big( \left( \tilde{\mathcal{I}}^{\Lambda^{*T}}, \tilde{N}^{*T}, \tilde{M}^{*T}, \tilde{M}^{T} \right) \big)_{T \geq 0} \) to a certain X in \( D([0,t_0], \mathbb{R}^4) \), endowed with the \( J_1(\mathbb{R}^4) \) topology: \(X^T \overset{\mathcal L }{\Rightarrow} X\) in \(D([0,t_0], \mathbb{R}^4)\) . 
Let's now characterize the law of the limiting processes followed by the identification of the
stochastic equation that each weak accumulation point must satisfy.
\begin{theorem}[Characterization of the limits of the scaled nearly unstable Hawkes processes]\label{Thm:Limit}
	Let \(t_0 > 0\). As \(T \to +\infty\), under  Assumption \ref{assum:Lhawkes}, $(X^T_t)_{t \in [0,t_0] } \overset{\mathcal L }{\Rightarrow} (X_t)_{t \in [0,t_0] }$ in \((D([0,t_0], \mathbb{R}^2_+\times\mathbb{R}^2), \text{J}_1(\mathbb{R}^2_+\times\mathbb{R}^2))\) i.e. the sequence \((X^T_t)_{t \in [0,t_0] }\)
	converges in law for the Skorokhod topology on \([0,t_0]\) towards some processes $$
	X := \left\{ X_t =\left( \tilde{\mathcal{I}}_t^{\Lambda^*}, \tilde{N}^*_t, \tilde M^*_t, \tilde M_t \right) : t \geq 0 \right\} \quad	\textit{satisfying the following properties: }$$
	
	\begin{itemize}
		\item The convergence holds almost surely and uniformly (up to copies in law of the sequence \(\big( X^T\big)_{T \geq 0}\)on a new probability space) i.e.

		\begin{equation}\label{eqCvgce}
			\lim_{T \to \infty} \sup_{t \in [0, t_0]} \left( |\tilde{\mathcal{I}}^{\Lambda^{*T}}_t - \tilde{\mathcal{I}}^{\Lambda^{*}}_t| + |\tilde N^T_t - \tilde N^*_t | + |\tilde M^{*T}_t - \tilde M^*_t | + |\tilde M^{T}_t - \tilde M_t | \right) = 0 \quad \text{a.s.} 
		\end{equation}

		\item There exists a Brownian motion \( W \) such that: $
		\tilde M^*_t = \int_0^t \sqrt{\Lambda^*_s} dW_s
		$ and \(\tilde M_t := \int_{0}^{t} \varsigma(s) \, \sqrt{\Lambda^*_s} dW_s\), which are both continuous martingales.
		\item $\tilde{\mathcal{I}}_\cdot^{\Lambda^*}= \tilde{N}^*_\cdot =\int_0^\cdot\Lambda^*_s ds$ is a continuous, non-decreasing and nonnegative process, starting from \(0\). The derivative
		 \( \Lambda^* \) of \( \tilde{N}^* \) is an accumulation point of the rescaled process \( \{\Lambda^{*T}\}_{T \geq 0} \) in \( D([0, t_0); \mathbb{R}^+) \) and the unique (in law)  non-negative  continuous weak solution of the stochastic Volterra integral equation below
		on \([0, t_0]\):
		\begin{equation}\label{eq:limitThm}
		\Lambda^*_t = \Lambda^*_0\phi(t) +\int_0^t K_{\alpha}(t-s)\lambda(\theta(s)- \Lambda^*_s)ds + \nu \int_0^t K_{\alpha}(t-s)\varsigma(s) \sqrt{\Lambda^*_s}dW_s, \quad X_0\perp\!\!\!\perp W,
		\end{equation}
		where $K_{\alpha}:= u\to \frac{u^{\alpha - 1}}{\Gamma(\alpha)} \mathbf{1}_{\mathbb{R}_+}(t)$ is the fractional integration kernel. \textit{The solution to the above equation is unique in law. In particular, the sequence of rescaled Hawkes processes converges in law.}
		
		\item Furthermore, the process \( \Lambda^* \) is non-negative and we have the following moment estimate for any \(p>0\):
		\begin{equation}\label{eq:MargBounds}
			\sup_{t \in [0, t_0]} \mathbb{E} \left[ \left| \Lambda^*_t \right|^{p} \right] \leq C_{p,t_0} \cdot \left( 1 + \|\phi\|^{p}_{t_0}\mathbb{E}[|\Lambda^*_0|^{p}] \right).
		\end{equation}
	\end{itemize}
\end{theorem}

\noindent {\bf Remark: }
	1. {\em Strict Positivity of the weak solution:} If \(\phi(t)>0, \forall\, t\geq0\) and the initial sequence in Assumption \eqref{assum:Lhawkes} (4) and its limit are strictly positive real-valued random variable, then the weak solution \( \Lambda^* \) of the equation~\eqref{eq:limitThm} above is strictly positive. Indeed, if \(Z_0 \equiv 0\) in equations~\eqref{eq:intensity} and ~\eqref{eq:intensity2}, as shown in \cite{JaissonRosenbaum2016,ElEuchFukasawaRosenbaum2018}, the suitably rescaled version of intensity process asymptotically behaves as the variance(non-negative) process of a rough Heston model with initial variance equal to zero. With the non-zero starting value \(\Lambda^{*}_0\) in the limit in \eqref{eq:limitInt}, strict positivity is ensurely insofar as the limit \(\Lambda^{*}_0\) of the initial sequence \((T^{- \alpha}\Lambda_0^{*T})_{T\geq0}\) is strictly positive.
	
	\medskip
	\noindent 2.{\em Weak well-posedness of the fake stationary rough Heston model:}  Since \(S\) in Equation~\eqref{E:HestonlogS}  is fully determined by \(V\) in Equation~\eqref{eq:Volterra} , the existence of $S$ readily follows from that of $V$ whose weak existence and uniqueness in law is given by Theorem~\ref{Thm:Limit}.
	Consequently, the stochastic Volterra system~\eqref{E:HestonlogS}--~\eqref{eq:Volterra} admits a \([0,+\infty)-\)valued unique in law 
	continuous weak solution $(S, V)$ with values in $\mathbb{R} \times \mathbb{R}_+$, for any initial state $(S_0,V_0) \in \mathbb{R} \times \mathbb{R}_+$. Such new model belongs to the class of diffusion-based models used in the industry to encode implied-volatility information ( see~\cite{ElEuchGatheralRosenbaum2019}), alongside non-parametric grids of implied volatilities with spline interpolation, direct modelling of the asset’s implied density and surface-level parametrizations (see e.g.~\cite{Gatheral2012arbitrage}) followed by AI-driven fitting as illustrated in \cite{GnabeyeuKarkarIdboufous2024}.
	
	\medskip
	\noindent We now analyse the H\"older pathwise regularity of the sample paths of the variance process $V$. 
	\subsection{Regularity and maximal inequality}
	In this section we provide the proof of the H\"older continuity and the maximal inequality of the weak solutions $\Lambda^*$ to the two equivalent stochastic Volterra equations ~\eqref{eq:limitInt} and ~\eqref{eq:limitIntAffine2}.
	To this end, we recall that for a real-valued stochastic process $X$ on $[0, t_0]$, the Kolmogorov continuity theorem states that if for some constants $p, a, C > 0$,
	\begin{equation}
		\mathbb{E}\left[ \left| X_t - X_s \right|^p \right]
		\leq C \cdot |t - s|^{1 + a}, 
	\end{equation}
	uniformly in $0 \leq s, t \leq t_0$, then the process has a $\theta$-H\"older continuous modification for all $0 < \theta < a/p$. Hence we first need to establish the moment estimates for the increments of $\Lambda^*$.
	
	\begin{assumption}[Integrability and uniform H\"older continuity]\label{ass:int_holregul}
		Let \( \lambda, \alpha > 0 \). Assume the kernel \( K_\alpha \) is such that its \( \lambda \)-resolvent \( R_{\alpha, \lambda} \) and its derivative \( -f_{\alpha, \lambda} \) satisfy:
		\begin{enumerate}% [label=(\roman*)]
			\item[(i)] {\em Integrability:}
			\(\quad \big({\cal K }^{int}_{\beta}\big):\quad \int_0^{+\infty} f_{\alpha,\lambda}^{2\beta}(u) \, du < +\infty \quad \text{for some } \beta > 1,
			\)
			%	so that \( f_{\lambda} \in \mathcal{L}^2(\text{Leb}_1) \).
			
			\item[(ii)] {\em H\"older continuity:}  \( \exists \; \vartheta \in (0, 1] \), \( C < +\infty \) s.t. \(\max_{i=1,2} \left[ \int_0^{+\infty} |f_{\alpha,\lambda}(u + \bar{\delta}) - f_{\alpha,\lambda}(u)|^i \, du \right]^{\frac{1}{i}} \leq C \bar{\delta}^{\vartheta}.\)
			\item[(iii)] {\em Moment bound and regularity of the initial condition:} For some \( \delta > 0 \), for any \( p > 0 \) and \( t_0 > 0 \),
			\[
			\mathbb{E}\left(\sup_{t \in [0,t_0]} |\Lambda^*_0 \phi(t)|^p \right) < +\infty, \; \text{and} \;\; 
			\mathbb{E} \left|\Lambda^*_0 \phi(t') - \Lambda^*_0 \phi(t)\right|^p \leqslant C_{T,p} \left( 1 + \mathbb{E}\left[ \sup_{t \in [0,t_0]} |\Lambda^*_0 \phi(t)|^p \right] \right) |t' - t|^{\delta p}.
			\]
		\end{enumerate}
	\end{assumption}
	
	\begin{theorem}[Regularity and maximal inequality for inhomogeneous Volterra fractional diffusion]\label{Thm:regul}
		Let \( \Lambda^* \) be any weak solution to the stochastic Volterra equation~\eqref{eq:limitInt}. %(or ~\eqref{eq:limitIntAffine2}). 
		Then the following hold:
		\begin{enumerate}
			\item The process \( \Lambda^* \) is non-negative and has H\"older regularity \(\delta \wedge\vartheta\wedge \frac{\beta-1}{2\beta} - \epsilon \) for any sufficiently small \( \epsilon > 0 \) i.e. the sample paths of \( \Lambda^* \) are almost surely H\"older pathwise continuous of any order strictly less than \( \delta \wedge\vartheta\wedge \frac{\beta-1}{2\beta} \).

		%	 Moreover, the sample paths of $V$ are \( \left( \delta \wedge \vartheta \wedge \widehat \theta - \frac{1}{p} - \eta \right) \)-H\"older pathwise continuous (modulo  \( P \)-indistinguability) for sufficiently small \( \eta > 0 \).
			 
			\item For any \( a\!\in \big(0,\delta \wedge\theta\wedge \frac{\beta-1}{2\beta}\big)  \) and \( p \geq 0 \), there exists a constant \( C > 0 \) such that for any \( t_0 \geq 0 \),
			
			\begin{equation}\label{eq:Holderpaths}
				\left \| \sup_{s\neq t\in [0,t_0]}\frac{|\Lambda^*_t-\Lambda^*_s|}{|t-s|^{a}}\right\|_p^p = \mathbb{E}\left[\sup_{s\neq t\in [0,t_0]}\frac{|\Lambda^*_t-\Lambda^*_s|^p}{|t-s|^{ap}} \right]   <C_{a,p,t_0}  \left( 1 + \|\phi\|_{t_0}^{p}\mathbb{E}[|\Lambda^*_0|^{p}] \right)
			\end{equation}
			for some positive real constant $C_{a,p,t_0}=C_{a,\delta, \vartheta,\beta,p,t_0, f_{\alpha, \lambda}} $. 
			
			%	\[
			%	\mathbb{E} \left[ \left\| V^* \right\|^p_{C^{\kappa}_T} \right] \leq C \cdot (1 + T)^{p(\alpha - \kappa)}.
			%	\]

			\item In particular for each \( p \geq 0 \), there exists a constant \( C > 0 \) such that for any \( T > 0 \),
			\begin{equation}\label{eq:supLpbound}
				\left \| \sup_{t\in [0,t_0]} |\Lambda^*_t|\right\|_p^p =\mathbb{E}\left[ \sup_{t\in [0,t_0]} |\Lambda^*_t|^p\right] \le C'_{a,p,T}\left( 1 + \|\phi\|_{t_0}\mathbb{E}[|\Lambda^*_0|^p] \right).
			\end{equation}
		\end{enumerate}
		
	\end{theorem}
	
\noindent {\bf Remark.}
	1.	According to \cite[Proposition 5.1]{Pages2024} (see also \cite[Proposition 6.1.5]{EGnabeyeu2025} ), for the \(\alpha\)-fractional integration kernel,  \( \vartheta \in (0, \alpha - \frac{1}{2}) \) and hence owing to the above theorem, the process  \( \Lambda^* \) is almost surely H\"older continuous of any order strictly less than \( \delta \wedge\alpha - \frac{1}{2}\wedge \frac{\beta-1}{2\beta} \).
	One can check
	(see, e.g. \cite[Example~2.1]{RiTaYa2020}) that the condition $\big({\cal K }^{int}_{\beta}\big)$ of the assumption~\ref{ass:int_holregul} on \(f_{\alpha,\lambda}\) (or equivalently, on the kernel \( K_\alpha \) ) is satisfied with $\beta \in \bigl(1,\frac{1}{2(\alpha-1)^{-}}\bigr)$. Hence, as $\alpha \in \bigl(\tfrac12, 1\bigr)$, one has \(\frac{\beta-1}{2\beta} = \tfrac12 - \tfrac{1}{2\beta} \leq \tfrac12 - (\alpha-1)^{-}=\alpha-\frac12 \) so that, the final regualrity becomes \(\delta \wedge\alpha - \frac{1}{2}\).

 \medskip
 \noindent 2. In the literature, the bounds in~\eqref{eq:Holderpaths} and~\eqref{eq:supLpbound} hold only when the starting random variable \( \Lambda^*_0 \in L^p(\P) \) for sufficiently large \( p \) (say, \( p \geq 2 \) in most cases), depending on the integrability characteristics of the kernels. 
 Surprisingly, this unusual restriction can be circumvented thanks to \cite[Lemma D.1]{JouPag22}, which is applied to a specific functional, as illustrated in the proof of the above theorem~\ref{Thm:regul} in section~\ref{subsect:proofMresult2}.
	
   \section{Proofs of the main results}\label{sect:proofMresult}
   \subsection{Proofs of Propositions~\ref{prop:Lhawkes_tightness} and~\ref{prop:tightIntensity} }\label{subsect:proofMresult1} 
   
   To prove the tightness of \( X^T \), we use \cite[Theorem~7.3]{Billingsley1999}, 
   which states that a  sequence of càdlàg stochastic processes $\left( X^T\right)_T$ defined on $[0\,,t_0]$ is tight provided the following two conditions hold:
   \begin{itemize}
   	\item[$\bullet$] For each \( \eta > 0 \) there exists \( a > 0 \) such that 
   	\(\displaystyle \limsup_{T \to \infty} \mathbb{P}\bigl( | X^T_0 | \geq a \bigr) \leq \eta\).
   	\item[$\bullet$] For any \(\varepsilon > 0\), \(\lim_{\delta \to 0^+} \limsup_{T \to \infty} 
   	\mathbb{P}\bigl( w(X^T, \delta) \geq \varepsilon \bigr) = 0,\)\footnote{This condition, together with the requirement that $\lim\limits_{R \to + \infty}\, \limsup\limits_T\, \mathbb{P}\left(\sup_{0\leq t\leq t_0}|X^T_t| > R \right) = 0\,$ , implies the \(\mathcal{C}\)-tightness of the sequence by Proposition~VI.3.26 in \cite{jacod2013limit}. }
   	where the modulus of continuity  \(\omega(f; \delta)\) of \(f\) is defined by
   	
   	\centerline{$
   		w(f, \delta) := 
   		\sup_{\substack{0 \leq s \leq t \leq t_0 \\ |t - s| \leq \delta}}
   		| f(t) - f(s) |, 
   		\qquad f \in D([0,t_0]), \ \text{for}\;\delta > 0.$}
   \end{itemize}
   If the maximal jump of \( X^T \), that is \(\displaystyle \sup_{0 \le t \le T} |\Delta X^T_t|\) goes to zero in probability as \( T \to \infty \),
   then \( (X^T)_{T>0} \) is \(\mathcal{C}\)-tight (see e.g. \cite[Proposition~VI.3.26]{jacod2013limit}).\\
   
   \noindent {\bf Proofs of Proposition~\ref{prop:Lhawkes_tightness}:}
   \noindent  {\sc Step~1} {\em ($\mathcal C$-tightness  of  \( (\tilde{N}^T)_{T > 0} \) and  \( (\tilde{\mathcal{I}}^{\Lambda^{*T}})_{T > 0} \) ):} The first condition above clearly holds. Note that, it follows from Proposition \ref{prop:estim}, that $ \sup_{T>0} \mathbb{E}[\tilde{N}^{*T}_t]= \sup_{T>0} \mathbb{E}[\tilde{\mathcal{I}}^{\Lambda^{*T}}_t] <+\infty$ and consequently by Markov's inequality,
   \[
   \lim_{r \to +\infty} \limsup_{T} \mathbb{P} \left( \sup_{0\leq t\leq t_0}|\tilde{N}^{*T}_t| > r \right) = \lim_{r \to +\infty} \limsup_{T \to \infty} \mathbb{P} \left( \sup_{0\leq t\leq t_0}|\tilde{\mathcal{I}}^{\Lambda^{*T}}_t| > r \right)=0.
   \]
   We now prove that \(\tilde{\mathcal{I}}^{\Lambda^{*T}}_t := \mathcal{I}^{I^T}_t + \mathcal{I}^{J^T}_t + \mathcal{I}^{K^T_1}_t + \mathcal{I}^{K^T_2}_t \) verifies the second condition.
   Proposition~\ref{prop:uniformCvgce}	prove the uniform convergence of the term \( \mathcal{I}^{I^T}_t + \mathcal{I}^{J^T}_t + \mathcal{I}^{K^T_1}_t\) towards the process \(\int_0^t \Lambda^*_0 \left(\phi(s)- \int_0^s f_{\alpha,\lambda}(s-u)\phi(u)du \right)ds + \int_0^t \int_0^s f_{\alpha,\lambda}(s-u)\theta(u)du\,ds \) 
   and therefore is tight. We then focus on the term \(\mathcal{I}^{K^T_2}_t = \int_0^t \int_0^s (1 - a_T)  R^{\prime T}_{-1}(T(s - u))\frac{dM^T_{Tu}}{\tilde\mu^T(\frac{u}{T})}\,ds\) and we first apply the stochastic Fubini theorem \cite[ Theorem 2.6]{Walsh1986}, to write:
   \begin{align*}
   	\mathcal{I}^{K^T_2}_t &=\int_0^t (1 - a_T) \left[ \int_0^s  R^{\prime T}_{-1}(T(s - u)) \, \frac{dM^T_{Tu} }{\tilde\mu^T(\frac{u}{T})} \, \right] ds=\int_0^t \left[ \int_0^s \sqrt{ \frac{T(1 - a_T)}{\tilde\mu^T(\frac{u}{T})}}  R^{\prime T}_{-1}(T(s - u)) \, d\tilde{M}^{*T}_u \, \right] ds \\
   	&= \int_0^t\int_u^t  \sqrt{\frac{T(1 - a_T)}{\tilde\mu^T(\frac{u}{T})}}  R^{\prime T}_{-1}(T(s - u)) \,  ds \, d\tilde{M}^{*T}_u=\int_0^t\int_0^{t-u}  \sqrt{\frac{T(1 - a_T)}{\tilde\mu^T(\frac{u}{T})}}  R^{\prime T}_{-1}(Ts) \,  ds \, d\tilde{M}^{*T}_u\\
   	&=  \int_0^t \int_0^{t-u} \underbrace{\sqrt{\frac{1}{T(1 - a_T)\tilde\mu^T(\frac{u}{T})}}}_{\varsigma^T(u)}\, \underbrace{ T(1 - a_T)R^{\prime T}_{-1}(Ts)}_{f^T(s)}\, ds \, d\tilde{M}^{*T}_u = \int_0^t \int_0^{t-u} f^T(s) \, ds \,\varsigma^T(u)\, d\tilde{M}^{*T}_u \\
   	&=  \int_0^t f^T(t-u) \int_0^{u}\,\varsigma^T(u) \,d\tilde{M}^{*T}_s \,  du=\int_0^t  f^T(t-u)\,\tilde{M}^{T}_u \,  du.
   \end{align*}
   where the penultimate equality comes from an integration by parts so that for \( 0 \leq s \leq t \leq s+\delta \leq t_0 \)
   \begin{align*}
   	\bigl|\mathcal{I}^{K^T_2}_t - \mathcal{I}^{K^T_2}_s\bigr| 
   	&= \left|
   	\int_0^t f^T(t-u)\,\tilde{M}^{T}_u \,  du 
   	-
   	\int_0^s f^T(t-u)\,\tilde{M}^{T}_u \,  du \right|\\
   	&\leq \left( \int_s^t f^T(t - u)\, du 
   	+ \int_0^s \bigl|f^T(t - u) - f^T(s - u)\bigr|\, du \right)
   	\sup_{u \leq t_0} |\tilde{M}^{T}_u|.
   \end{align*}
   Under Assumption~\ref{assum:Lhawkes} (1), the kernel \(\varphi\) is completely monotone, 
   and it follows from Theorem~5.4 in \cite{gripenberg1990} that \(f^T\) is decreasing. 
   Since \(|t - s| \leq \delta\), we get
   \begin{align*}
   	\bigl|\mathcal{I}^{K^T_2}_t - \mathcal{I}^{K^T_2}_s\bigr| &\leq
   	\left( \int_0^\delta f^T(u)\, du 
   	+ \int_0^s \bigl|f^T(t -s + u) - f^T( u)\bigr|\, du \right)
   	\sup_{u \leq t_0} |\tilde{M}^{T}_u|\\
   	&\leq
   	\left( \int_0^\delta f^T(u)\, du + \int_0^s f^T(u)\, du
   	-\int_{t-s}^t f^T(u)\, \right)
   	\sup_{u \leq t_0} |\tilde{M}^{T}_u|\leq 2\,
   	\left( \int_0^\delta f^T(u)\, du \right)
   	\sup_{u \leq t_0} |\tilde{M}^{T}_u|
   \end{align*}
   Using Markov's inequality and owing to the second claim of Proposition~\ref{prop:estim}, we deduce that
   \[
   \mathbb{P}\bigl(\omega(\mathcal{I}^{K^T_2};\delta) \geq \varepsilon\bigr)
   \leq \frac2\varepsilon\,
   \mathbb{E}\Bigl[\sup_{u \leq t_0} |\tilde{M}^{T}_u|\Bigr]\,
   \int_0^\delta f^T(u)\, du
   \leq K \int_0^\delta f^T(u)\, du
   \]
   for some positive constant \( K \), and we conclude using that \(\lim_{\delta \to 0} \limsup_{T \to \infty} 
   \int_0^\delta f^T(u)\, du = 0.\)\\
   We proceed likewise for \(\tilde{N}^{*T}\) (or call upon \cite[Proposition 3.4]{HorstXuZhang2024}) and thus conclude the tightness of the sequences \( \left( (\tilde{N}^{*T}_t)_{t \in [0,t_0]} \right)_{T \geq 0} \) and \( \left( (\tilde{\mathcal{I}}^{\Lambda^{*T}}_t)_{t \in [0,t_0]} \right)_{T \geq 0} \). %, both of them being increasing. \\
   Additionally, since the maximum jump size of \( \tilde{N}^{*T}_t \) and \( \tilde{\mathcal{I}}^{\Lambda^{*T}}_t\), that is $\frac{(1 - a_T) }{T \, c}$ 
   (see the Remark (3) on assumption \ref{assum:Lhawkes}), tends to zero, the \( \mathcal{C} \)-tightness of \( (\tilde{N}^T_t, \tilde{\mathcal{I}}^{\Lambda^{*T}}_t) \) follows from Proposition VI-3.26 in \cite{jacod2013limit}.
   
   \medskip
   \noindent  {\sc Step~2} {\em ($\mathcal C$-tightness  of $\tilde{M}^{*T}$ and \(\tilde{M}^T\)):}
   
   \noindent {\em (1).}
   \(\tilde{M}^{*T}\) is a local martingale, since it can be written in the form $\tilde{M}^{*T} =  \int_{0}^{tT}H_sdM^T_{s}$ where $H_s$ is a deterministic, predictable and bounded function, see \cite[Chapter 2]{Protter}. According to Theorem VI-4.13 in \cite[ p.~358]{jacod2013limit}, the $\mathcal C$-tightness of the predictable quadratic variation process guarantees the tightness of the sequence \( (\tilde{M}^{*T})_{T>0} \).
   It is easy to check that
   $\langle \tilde{M}^{*T}\rangle_t = \tilde{\mathcal{I}}^{\Lambda^{*T}}_t,$ which is $\mathcal C$-tight. This gives the tightness of $\tilde{M}^{*T}$. Additionally, the maximum
   jump size of $\tilde{M}^{*T}$, that is $\sqrt{\frac{(1 - a_T) }{T \, c}}$, vanishing as \(T\) goes to infinity, we obtain that $\tilde{M}^{*T}$ is $\mathcal C$-tight and so its limit is continuous. It will remain to show that the limit is in fact a martingale. 
   
   \medskip
   \noindent {\em (2).}
   \(\tilde{M}^T\) is a local martingale, since owing to Remark (3) on assumption \ref{assum:Lhawkes}, $\varsigma^T$ is a deterministic, predictable and bounded function, (see e.g. \cite[Chapter 2]{Protter}).
   It is easy to check that
   $\langle \tilde{M}^T\rangle_t $ is $\mathcal C$-tight, which garantees still owing to Theorem VI-4.13 in \cite[ p.~358]{jacod2013limit}, the tightness of the sequence \( (\tilde{M}^T)_{T>0} \).
   Additionally, since the maximum jump size of \(\tilde{M}^T \), that is $\frac{1 }{\nu\,T \, c}$, tends to zero, the \( \mathcal{C} \)-tightness of \( \tilde{M}^T \) follows from Proposition VI-3.26 in \cite{jacod2013limit}.
   From Assumption \ref{assum:Lhawkes} (2), we have the deterministic pointwise convergence $
   \varsigma^T(t) \longrightarrow \varsigma(t),
   $ and from the above $\mathcal C$-tightness  of $\tilde{M}^{*T}$, let \(\tilde{M}^{*}\)
   be an accumulation point. We deduce that the pair \( (\varsigma^T, \tilde{M}^{*T}) \) converges jointly in distribution to a certain \( (\varsigma, \tilde{M}^{*}) \) in the Skorokhod space \( D([0,t_0],\mathbb{R}_+\times\mathbb{R}) \).
   Now, by the classical result on stability of integral processes (see e.g. \cite[ Theorem VI.6.22]{jacod2013limit} or \cite{Kurtz1991}),we conclude that
   \(
   \int_0^t \varsigma^T(s) \, d\tilde{M}^{*T}_s \overset{\mathcal L }{\Rightarrow} \int_0^t \varsigma(s)\, d\tilde{M}^*_s =: \tilde{M}_t\), as \(T \to +\infty \) i.e.
   in distribution in the Skorokhod space \( D([0,t_0],\mathbb{R}) \). 
   
   \medskip
   \noindent We may conclude by noticing that, the three sequences \((\tilde{\mathcal{I}}^{\Lambda^{*T}})_{T>0}, (\tilde{N}^{*T})_{T>0} , (\tilde{M}^{*T})_{T>0}, \text{ and }  (\tilde{M}^{T})_{T>0}\) are $J_1(\R)$-tight, which implies the joint $J_1(\R^4)$-tightness of \( \Big(X^T =\left( \tilde{\mathcal{I}}^{\Lambda^{*T}}, \tilde{N}^{*T}, \tilde{M}^{*T} \right)\Big)_{T>0} \), see \cite[Lemma 3.2]{Whitt2007} or the joint $\mathcal C$-tightness, thanks to \cite[Corollary VI.3.33]{jacod2013limit}.\\
   
   \noindent  {\sc Step~3} {\em (Martingale convergence of $\tilde{\mathcal{I}}_t^{\Lambda^{*T}} - \tilde{N}^{*T}_t $):} 
   We have
   $\tilde{N}^{*T}_t -\tilde{\mathcal{I}}_t^{\Lambda^{*T}} = \frac{1-a_T}{T } \int_{0}^{tT} \frac{1}{\tilde\mu^T(\frac{s}{T^2})}dM^T_{s},$ 
   which  owing to \cite[Chapter II]{Protter} is still a local martingale, with bounded predictible quadratic variation. From Doob's martingale inequality ($L^p$ -Burkholder-Davis-Gundy (BDG) inequality for \(p=2\)), we get using equation~\eqref{eq:bounds_I} in the last line that:
   {\small
   	\begin{align*}
   		&\ \mathbb{E}\left[\underset{t \in [0,t_0]}{\sup}|\tilde{\mathcal{I}}_t^{\Lambda^{*T}}-\tilde{N}^{*T}_t |^2 \Big| \mathcal{F}_0 \right] = \left(\frac{1-a_T}{T}\right)^2\mathbb{E}\left[\underset{t \in [0,t_0]}{\sup}|\int_{0}^{tT} \frac{1}{\tilde\mu^T(\frac{s}{T^2})}dM^T_{s} |^2 \Big| \mathcal{F}_0 \right]\\ 
   		&\hspace{2cm}\leq 4\left(\frac{1-a_T}{T }\right)^2 \mathbb{E}\left[|\int_{0}^{Tt_0} \frac{1}{\tilde\mu^T(\frac{s}{T^2})}dM^T_{s} |^2 \Big| \mathcal{F}_0 \right] \leq 4\left(\frac{1-a_T}{T}\right)^2 \mathbb{E}\left[\int_{0}^{Tt_0} \frac{1}{\tilde\mu^T(\frac{s}{T^2})^2} \Lambda^T_{s} ds\Big| \mathcal{F}_0 \right]\\
   		&\hspace{2cm}\leq 4\left(\frac{1-a_T}{T\,c}\right)^2 \mathbb{E}[\mathcal{I}^{\Lambda^{T}}_{Tt_0}]\leq 4\left(\frac{1-a_T}{T\,c }\right)^2 \frac{K }{1-a_T} T t_0 =\frac{4}{c^2}\,K\,t_0\, \frac{1-a_T}{T }
   	\end{align*}
   }
   \noindent since $\tilde\mu^T$ is uniformly bounded on \([0,t_0]\). We then deduce that:
   $\mathbb{E}\big[\underset{t \in [0,t_0]}{\sup}|\tilde{\mathcal{I}}_t^{\Lambda^{*T}}-\tilde{N}^{*T}_t |^2\big]\leq  \frac{C}{T}.
   $
   This gives the uniform convergence to zero in probability of $\tilde{N}^{*T}_t -\tilde{\mathcal{I}}_t^{\Lambda^{*T}}$.
   
   \noindent  {\sc Step~4} {\em (Limit of $\tilde{M}^{*T}$ and $\tilde{M}^T$):} 	If \( X =\left( \tilde{\mathcal{I}}^{\Lambda^*}, \tilde{N}^{*}, \tilde{M}^{*}, \tilde{M} \right) \) is a possible limit point of \( X^T =\left( \tilde{\mathcal{I}}^{\Lambda^{*T}}, \tilde{N}^{*T}, \tilde{M}^{*T}, \tilde{M}^T \right) \), we know that the processes  $\left( \tilde{\mathcal{I}}^{\Lambda^*}, \tilde{N}^{*}, \tilde{M}^{*}, \tilde{M} \right)$ have continuous sample paths. Moreover, by \cite[Corollary~IX.1.19]{jacod2013limit},
   $\tilde{M}^{*}$ is a local martingale, and therefore $\tilde{M}$ is also a local martingale.

   \medskip
   \noindent {\em (1).} However the local martingale \( \tilde M^{*T} \) has predictable quadratic variation \(	\langle \tilde{M}^{*T} \rangle_t = \tilde{\mathcal{I}}^{\Lambda^{*T}}_t = \int_0^t \Lambda^{*T}_s \, ds,\) \(t \geq 0.\) By the the continuous mapping theorem, together with \cite[Theorem VI-6.26]{jacod2013limit}, we get that as \( T \to +\infty \), $\tilde{M}^{*}$ is the limit of $\tilde{M}^{*T}$ and  $\langle \tilde{M}^{*} \rangle=\tilde{\mathcal{I}}^{\Lambda^*}$. Specifically, $\tilde{\mathcal{I}}^{\Lambda^{*}}_t$ is the limit of  $\langle \tilde{M}^{*T}  \rangle$ as \( T \to +\infty \), that is,
   \(\langle \tilde{M}^{*T}  \rangle_t  \overset{\mathcal L }{\Rightarrow} \tilde{\mathcal{I}}^{\Lambda^{*}}_t =\int_0^t \Lambda^{*}_s \, ds \quad \text{in } D([0,t_0], \mathbb{R}).\)
   Consequently, the (local) martingale \( \tilde{M}^{*} \) has predictable quadratic variation \(\langle \tilde{M}^{*} \rangle_t = \tilde{\mathcal{I}}^{\Lambda^{*}}_t =\int_0^t \Lambda^{*}_s \, ds, \quad t \geq 0.\) By Fatou's lemma, the expectation of $\langle \tilde{M}^{*} \rangle$ is finite and therefore $\tilde{M}^{*} $ is a martingale.
   
   \medskip
   \noindent {\em (2).}
   The sequence \( (\tilde{M}^T)_{T>0} \) weakly converges along a subsequence to a continuous local martingale \(\tilde M_t := \int_{0}^{t} \varsigma(s) \, d\tilde M^{*}_s\), with predictable quadratic variation:
   \( \langle \tilde{M} \rangle_t = \int_{0}^{t} \varsigma^2(s) d\langle \tilde{M}^{*} \rangle_s =\int_{0}^{t} \varsigma^2(s) \Lambda^*(s) ds.\)
   It is in fact a true martingale  (bounded quadratic variation as \(\varsigma\) is a bounded borel function). \hfill$\Box$\\

   \noindent {\bf Proof of Proposition~\ref{prop:tightIntensity}:} The uniform integrability follows from the first claim of Proposition~\ref{prop:estim}. From Corollary~\ref{Corol:LimitMeasure} and Lemma~\ref{Lm:cvgcezero}~(2), it follows that the sequence \((I^{T} + J^{T} + K^T_1)_{T>0}\) is \(\mathcal{C}\)-tight. Consequently, to establish the \(\mathcal{C}\)-tightness of the rescaled intensity processes \( (\Lambda^{*T})_{T > 0} \), \cite[Corollary~VI.3.33(1), p.353]{jacod2013limit} shows that it remains to establish the \(\mathcal{C}\)-tightness of the stochastic integral processes  \(( K^T_2)_{T>0}\). We observe that, \(( K^T_2)_{T>0}\) can be rewritten \(K^T_{2,t}=\frac{\nu}{\lambda}\int_0^t f^T(t-u)d\tilde{M}^T_u\). To prove the tigtness of \(K^T_{2}\), we observe that the resolvents and their first derivatives
   are smooth and bounded functions, so that the function \(f^T\) and satisfy
   
   \centerline{$
   	f^T(t - u) = f^T(0) + \int_u^t f^{\prime T}(r - u) \, dr, 
   	\quad T >0, \ t \ge u \ge 0.
   	$}
   \noindent with \(f^T(0) = T(1-a_T)a_T \varphi(0)\). Consequently, up to the multiplicative constant \(\frac{\nu}{\lambda}\), we may write
   \begin{align*}
   	K^T_{2,t} &= \int_0^t f^T(0)\;d\tilde{M}^{T}_u + \int_0^t \int_u^{t} f^{\prime T}(r - u)\,dr\, d\tilde{M}^{T}_u= f^T(0)\;\tilde{M}^{T}_t + \int_0^t \int_0^{t-u} f^{\prime T}(s) \, ds\, d\tilde{M}^{T}_u \\
   	&=  f^T(0)\;\tilde{M}^{T}_t + \int_0^t  f^{\prime T}(t-u)\,\tilde{M}^{T}_u \,  du =  f^T(0)\;\tilde{M}^{T}_t + K^T_{3,t}.
   \end{align*}
   where the penultimate equality follows from an integration by parts. It thus boils down from Proposition~\ref{prop:Lhawkes_tightness}~(1) to show the tightness of \(K^T_{3,t}\), and hence we consider, for \(0 \le s \le t \le s+\delta \le t_0\), the relevant increments:
   \begin{align*}
   	\bigl|K^T_{3,t} - K^T_{3,s}\bigr| 
   	&= \left|
   	\int_0^t f^{\prime T}(t-u)\,\tilde{M}^{T}_u \,  du 
   	-
   	\int_0^s f^{\prime T}(t-u)\,\tilde{M}^{T}_u \,  du \right|\\
   	&\leq \left( \int_s^t \left|f^{\prime T}(t - u)\right|\, du 
   	+ \int_0^s \bigl|f^{\prime T}(t - u) - f^{\prime T}(s - u)\bigr|\, du \right)
   	\sup_{u \leq t_0} |\tilde{M}^{T}_u|.
   \end{align*}
   Under Assumption~\ref{assum:Lhawkes} (1), as the kernel \(\varphi\) is completely monotone, it follows still from Theorem~5.4 in \cite{gripenberg1990} that \(f^{\prime T}\) is negative and non-decreasing. 
   Since \(|t - s| \leq \delta\), we get
   \begin{align*}
   	&\,\bigl|K^T_{3,t} - K^T_{3,s}\bigr| \leq
   	\left( \int_0^{t-s} \bigl|f^{\prime T}(u)\bigr|\, du 
   	+ \int_0^s \bigl|f^{\prime T}(t -s + u) - f^{\prime T}( u)\bigr|\, du \right)
   	\sup_{u \leq t_0} |\tilde{M}^{T}_u|\\
   	&\hspace{1cm}\leq
   	\left( \int_0^\delta \bigl|f^{\prime T}(u)\bigr|\, du - \int_0^s f^{\prime T}(u)\, du
   	+\int_{t-s}^t f^{\prime T}(u)\, du \right)
   	\sup_{u \leq t_0} |\tilde{M}^{T}_u|\\
   	&\hspace{1cm}\leq
   	\left( \int_0^\delta \bigl|f^{\prime T}(u)\bigr|\, du + \int_0^s \bigl|f^{\prime T}(u)\bigr|\, du
   	-\int_{t-s}^t \bigl|f^{\prime T}(u)\bigr|\, du \right)
   	\sup_{u \leq t_0} |\tilde{M}^{T}_u|\leq 2\,
   	\left( \int_0^\delta \bigl|f^{\prime T}(u)\bigr|\, du \right)
   	\sup_{u \leq t_0} |\tilde{M}^{T}_u|.
   \end{align*}
   Using Markov's inequality and owing to the second claim of Proposition~\ref{prop:estim}, we deduce that
   \[
   \mathbb{P}\bigl(\omega(K^T_3;\delta) \geq \varepsilon\bigr)
   \leq \frac2\varepsilon\,
   \mathbb{E}\Bigl[\sup_{u \leq t_0} |\tilde{M}^{T}_u|\Bigr]\,
   \int_0^\delta \bigl|f^{\prime T}(u)\bigr|\, du
   \leq K \int_0^\delta \bigl|f^{\prime T}(u)\bigr|\, du
   \]
   for some positive constant \( K \), and we conclude using that \(\lim_{\delta \to 0} \limsup_{T \to \infty} 
   \int_0^\delta \bigl|f^{\prime T}(u)\bigr|\, du = 0.\)
   This prove the tightness of \(K^T_3\) and hence that of the sequence \( \left( (K^T_{2,t})_{t \in [0,t_0]} \right)_{T > 0}.\)
   Additionally, since the maximum jump size of \(K^T_2\) is bounded by a constant multiple of that of \(\tilde{M}^T \), which tends to zero as \(T\to+\infty\), the \( \mathcal{C} \)-tightness of \(K^T_2\) follows from Proposition VI-3.26 in \cite{jacod2013limit}. Consequently, all limit points of \(K_2^T\) have continuous paths.
   
   \medskip
   \noindent Note that alternatively, calling upon Propositions~\ref{Lem:Lhawkes}~(2) and~\ref{prop:Lhawkes_tightness}~(1) and~(3), and reasoning similarly to Step~$(3)$ of the proof of Theorem~\ref{Thm:Limit}, we can formally deduce the weak convergence:
   \(
   \int_0^t f^T(t-u)d\tilde{M}^T_u \overset{\mathcal L }{\Rightarrow} \int_0^t  f_{\alpha,\lambda}(t-s)\, d\tilde{M}_s\), as \(T \to +\infty \)
   in distribution in the Skorokhod space \( D([0,t_0],\mathbb{R}) \). That is, the sequence \( (K^T_2)_{T>0} \) weakly converges (along a subsequence) to the process \( \int_{0}^{t} f_{\alpha,\lambda}(t-s)\, d\tilde M_s\) which is almost surely continuous. This also ensures that  \(( K^T_2)_{T>0}\) is \(\mathcal{C}\)-tight. \hfill $\Box$
   
   	\begin{lemma}\label{Lm:cvgcezero}
   	Let \(t_0 > 0\) and \(t \in [0,t_0]\). The following results hold: %we have the following almost sure convergences:
   	\begin{enumerate}
   		\item \(\int_0^{t} T (1-a_T) R^{\prime T}_{-1}(T(t-s)) 
   		\left( \frac{1}{\tilde{\mu}^T \left( \frac{s}{T} \right)} 
   		- \frac{1}{\tilde{\mu}^T \left( \frac{t}{T} \right)} \right) ds 
   		\;\underset{T \to +\infty}{\rightarrow}\; 0 \quad \text{a.s.}\)\\
   		Moreover, this convergence holds uniformly on \([0,t_0]\).
   		
   		\item 
   		Define \(K^T_{1,t} := \int_0^t (1 - a_T)  R^{\prime T}_{-1}(T(t - u)) 
   		\left( \frac{1}{\tilde\mu^T\left(\frac{t}{T}\right)} 
   		- \frac{1}{\tilde\mu^T\left(\frac{u}{T}\right)} \right) dM^T_{Tu}.\)
   		Then \(K^T_1\) is \(\mathcal{C}\)-tight, and the same holds for 
   		\(\mathcal{I}^{K^T_1}\). Moreover, for every \(t \in [0,t_0]\)
   		\[
   		\mathcal{I}^{K^T_1}_t:=\int_0^t (1 - a_T) 
   		\left[ \int_0^s R^{\prime T}_{-1}(T(s - u)) 
   		\left( \frac{1}{\tilde\mu^T\left(\frac{s}{T}\right)} 
   		- \frac{1}{\tilde\mu^T\left(\frac{u}{T}\right)} \right) dM^T_{Tu} \right] ds
   		\;\underset{T \to +\infty}{\rightarrow}\; 0 \quad \text{a.s.}
   		\]
   	\end{enumerate}
   \end{lemma}
   \noindent {\bf Proof:}
   \smallskip
   \noindent  {\sc Step~1} {\em Proof of $\int_0^{t} T (1-a_T) R^{\prime T}_{-1}(T(t-s)) 
   	\left( \frac{1}{\tilde{\mu}^T \left( \frac{s}{T} \right)} 
   	- \frac{1}{\tilde{\mu}^T \left( \frac{t}{T} \right)} \right) ds \underset{T \to +\infty}{\rightarrow} 0 \quad \text{almost surely}.$.} \\
   The map \(\tilde{\mu}^T\) is Lipschitz on \([0,t_0]\), and hence it is uniformly continuous.  
   Consequently, for any \(\epsilon > 0\), there exists \(\eta > 0\) such that for all \(s, t \in [0,t_0]\) with \(|s-t| < \eta\), i.e. \(s \in [t-\eta, t]\) we have and \(\big| \frac{1}{\tilde{\mu}^T \left( \frac{t}{T}\right)}-\frac{1}{\tilde{\mu}^T \left( \frac{s}{T}\right)} \big| \leq \epsilon. \) Moreover, for every \(s \in [0,t-\eta]\) there exists a constant \([\tilde{\mu}^{T}]_{\mathrm{Lip}}=:C >0\) such that \(\Big|\frac{1}{\tilde{\mu}^T \left( \frac{t}{T} \right)} - \frac{1}{\tilde{\mu}^T \left( \frac{t-s}{T} \right)} \Big| \leq \frac{C}{c^2}\frac{s}{T}\) (see Remark (3) on assumption \ref{assum:Lhawkes} ).
   We then have:
   {\small 
   	\begin{align*}&\,\Big|\int_0^{t} T (1-a_T) R^{\prime T}_{-1}(T(t-s)) 
   		\left( \frac{1}{\tilde{\mu}^T \left( \frac{s}{T} \right)} 
   		- \frac{1}{\tilde{\mu}^T \left( \frac{t}{T} \right)} \right) ds\Big| \leq \int_0^{t-\eta} T (1-a_T) R^{\prime T}_{-1}(Ts) 
   		\left( \Big|\frac{1}{\tilde{\mu}^T \left( \frac{t}{T} \right)} - \frac{1}{\tilde{\mu}^T \left( \frac{t-s}{T} \right)} \Big| \right) ds\\
   		&\hspace{1.9cm}+\epsilon \int_0^{\eta} T (1-a_T) R^{\prime T}_{-1}(Ts) ds\leq \epsilon \int_0^{\eta} T (1-a_T) R^{\prime T}_{-1}(Ts) ds +  \frac{C}{c^2}\,\frac{t_0}{T}\int_0^{t-\eta} T (1-a_T) R^{\prime T}_{-1}(Ts) ds \leq 2\epsilon 
   	\end{align*}
   }
   for large enough T since owing to the claim (2) of Proposition \ref{Lem:Lhawkes}, we have \(\int_0^a T(1 - a_T) R^{\prime T}_{-1}(Tu) \underset{T \to +\infty}{\rightarrow} \int_0^af_{\alpha, \lambda}(s)ds =1- R_{\alpha, \lambda}(a) < 1\) for the \(\alpha\)-fractional integration kernel.

   \smallskip
   \noindent  {\sc Step~2} {\em Proof of $\mathcal{C}-$ tightness.} First note that, setting \(	g^T(t, s) = f^T(t - s) \left( \frac{\tilde{\mu}^T\left( \frac{s}{T} \right)}{\tilde{\mu}^T\left( \frac{t}{T} \right)} - 1 \right)\) for all \( 0 \leq s \leq t \), we have \(K^T_{1,t} :=\int_{0}^{t} g^T(t,u)\, d\tilde M^{T}_u\). Owing to the fundamental theorem of calculus, we can write \(g^T(t, u) = g^T(s, u) + \int_s^t \partial_r g^T(r, u) \, dr\)
   for all \( 0 \leq u \leq s \leq t \). Consequently, noticing that \(g^T(u, u) =0\), we have for \(s=u\):
   \begin{align*}
   	K^T_{1,t} &=\int_{0}^{t} g^T(t,u)\, d\tilde M^{T}_u= \int_0^t g^T(u,u)\,d\tilde{M}^{T}_u + \int_0^t \int_u^{t} \partial_r g^T(r, u) \,dr\, d\tilde{M}^{T}_u\\
   	&= \int_0^t \int_0^{t-u} \partial_r g^T(u+r, u) \, dr\, d\tilde{M}^{T}_u = - \int_0^t \partial_u g^T(t, u) \,\tilde{M}^{T}_u \, du.
   \end{align*}
   where the penultimate equality comes from an integration by parts so that, the relevant increments reads:
   \begin{align*}
   	\forall s, t \in [0,t_0], \quad 	\bigl|K^T_{1,t} - K^T_{1,s}\bigr|
   	&= \left|
   	\int_0^t \partial_u g^T(t, u)\,\tilde{M}^{T}_u \,  du 
   	-
   	\int_0^s \partial_u g^T(s, u)\,\tilde{M}^{T}_u \,  du \right|\\
   	&\leq \left( \int_s^t \left|\partial_u g^T(t, u)\right|\, du 
   	+ \int_0^s \bigl|\partial_u g^T(t, u) - \partial_u g^T(s, u)\bigr|\, du \right)
   	\sup_{u \leq t_0} |\tilde{M}^{T}_u|.
   \end{align*}
   Define \(	C_1 := \sup_{T>0} \sup_{u \in [0,t_0]} 
   \left| \tilde{\mu}^T\!\left( \frac{u}{T} \right) \right|,
   \;
   C_2 := \sup_{T>0} \sup_{u \in [0,t_0]}
   \left| \tilde{\mu}^{\prime T}\!\left( \frac{u}{T} \right) \right|.\) 
   One checks that for
   \begin{align*}
   	&\, 0 \leq s,u \leq t \leq t_0, \quad  \bigl| \partial_u g^T(t,u) \bigr|
   	\leq
   	\frac{[\tilde{\mu}^T]_{\mathrm{Lip}}}{T c}
   	\, |u-t| \, \bigl| f^{\prime T}(t-u) \bigr|
   	+
   	\frac{C_2}{T c}
   	\, \bigl| f^T(t-u) \bigr|, \quad \text{and}\\
   	& \bigl| \partial_u g^T(t,u) - \partial_u g^T(s,u) \bigr|
   	\leq\;
   	\frac{[\tilde{\mu}^T]_{\mathrm{Lip}}}{T c}
   	\, |u-t| \,
   	\bigl| f^{\prime T}(t-u) - f^{\prime T}(s-u) \bigr| +
   	\frac{C_2}{T c}
   	\, \bigl| f^T(t-u) - f^T(s-u) \bigr| \\
   	&\hspace{6cm}+
   	\frac{[\tilde{\mu}^T]_{\mathrm{Lip}}}{T c^2}
   	\, |t-s|
   	\left(
   	C_1 \bigl| f^{\prime T}(t-u) \bigr|
   	+
   	\frac{C_2}{T} \bigl| f^T(t-u) \bigr|
   	\right).
   \end{align*}
   (See~\ref{assum:Lhawkes} and Remark~(3) stated therein on the uniform boundedness (non-degeneracy) of 
   $\tilde\mu^T$ away from zero).
   Consider the increments, for \(0 \le s \le t \le s+\delta \le t_0\) and let
   \(
   C := \max\{\frac{\sup_{T>0} [\tilde{\mu}^T]_{\mathrm{Lip}}}{c^2} (2ct_0 \vee C_1\vee C_2), \frac{C_2}{c} \}.
   \)
   \begin{align*}
   	&\, \bigl| K^T_{1,t} - K^T_{1,s} \bigr|
   	\leq \frac{C}{T}
   	\Biggl(
   	\int_s^t f^T(t-u)\, du
   	+ \int_0^s \bigl| f^T(t-u) - f^T(s-u) \bigr|\, du
   	+ \frac{\delta}{T} \int_0^s f^T(s-u)\, du
   	\Biggr)
   	\sup_{u \leq t_0} \bigl| \tilde{M}^T_u \bigr| \\
   	&\qquad\qquad\quad
   	+ \frac{C}{T}
   	\Biggl(
   	\int_s^t \bigl| f^{\prime T}(t-u) \bigr|\, du
   	+ \int_0^s \bigl| f^{\prime T}(t-u) - f^{\prime T}(s-u) \bigr|\, du
   	+ \delta \int_0^s \bigl| f^{\prime T}(s-u) \bigr|\, du
   	\Biggr)
   	\sup_{u \leq t_0} \bigl| \tilde{M}^T_u \bigr| \\
   	&\leq \frac{C}{T}
   	\Biggl(
   	\int_0^{t-s} \bigl( |f^{\prime T}(u)| + f^T(u) \bigr)\, du
   	+ \int_0^s \bigl( |f^{\prime T}(u)| + f^T(u) \bigr)\, du
   	- \int_{t-s}^t \bigl( |f^{\prime T}(u)| + f^T(u) \bigr)\, du 
   	+ \delta \int_0^s \left( |f^{\prime T}(u)| \right. \\
   	&\left. + \frac{1}{T} f^T(u) \right)\, du
   	\Biggr)
   	\sup_{u \leq t_0} \bigl| \tilde{M}^T_u \bigr| \leq \frac{2C}{T}
   	\left(
   	\int_0^{\delta} \bigl( |f^{\prime T}(u)| + f^T(u) \bigr)\, du
   	+ \delta \int_0^{t_0} \bigl( |f^{\prime T}(u)| + \frac{1}{T} f^T(u) \bigl)\, du
   	\right)
   	\sup_{u \leq t_0} \bigl| \tilde{M}^T_u \bigr|.
   \end{align*}
   where the last inequality follows by proceeding as in Step~(1) of the proof of Proposition~\ref{prop:Lhawkes_tightness} and in the proof of Proposition~\ref{prop:tightIntensity}, and by using that $|t - s| \leq \delta$.
   We obtain, by Markov's inequality and the second claim of
   Proposition~\ref{prop:estim}, that
   \[
   \mathbb{P}\bigl(\omega(K^T_1;\delta) \geq \varepsilon\bigr)
   \leq \frac{2C}{T\varepsilon}\,
   \mathbb{E}\Bigl[\sup_{u \leq t_0} |\tilde{M}^{T}_u|\Bigr]\,
   \left(
   \int_0^{\delta} \bigl( |f^{\prime T}(u)| + f^T(u) \bigr)\, du
   + \delta \int_0^{t_0} \bigl( |f^{\prime T}(u)| + \frac{1}{T} f^T(u) \bigl)\, du
   \right).
   \]
   Since \(\lim_{\delta \to 0} \limsup_{T \to \infty} \frac{1}{T} \left(
   \int_0^{\delta} \bigl( |f^{\prime T}(u)| + f^T(u) \bigr)\, du
   + \delta \int_0^{t_0} \bigl( |f^{\prime T}(u)| + \frac{1}{T} f^T(u) \bigl)\, du
   \right) = 0,\) and by\\ Proposition~\ref{prop:estim}, \(\mathbb{E}\Bigl[\sup_{u \leq t_0} |\tilde{M}^{T}_u|\Bigr]  < \infty,\) the tightness of $\{K^T_{1}\}_{T>0}$ is established.
   Additionally, since the maximum jump size of \(K^T_1\) is bounded by a constant multiple of that of \(\tilde{M}^T \), which tends to zero, the \( \mathcal{C} \)-tightness of \(K^T_1\) follows from Proposition VI-3.26 in \cite{jacod2013limit}.
   Since \(K_1^T\) is \(\mathcal{C}\)-tight, all its limit points have continuous
   paths. As the integration operator is continuous on
   \(\mathcal{C}([0,t_0])\), the \(\mathcal{C}\)-tightness of
   \(\mathcal{I}^{K_1^T}\) follows from the continuous mapping theorem\cite[Theorem~VI.1.14]{jacod2013limit}.
   
   \smallskip
   \noindent  {\sc Step~2} {\em Proof of $ \mathcal{I}^{K^T_1}_t= \int_0^t (1 - a_T) \left[ \int_0^s  R^{\prime T}_{-1}(T(s - u)) \, \big(\frac{1}{\tilde\mu^T(\frac{s}{T})}-\frac{1}{\tilde\mu^T(\frac{u}{T})}\big) \, dM^T_{Tu} \,\right] ds \underset{T \to +\infty}{\rightarrow} 0$.} Applying the stochastic Fubini theorem \cite[Theorem~2.6]{Walsh1986}, we may write
   \begin{align*}
   	&\, \mathcal{I}^{K_1^T}_t
   	= \int_0^t \int_0^{s} g^T(s,u)\, d\tilde M^{T}_u \, ds
   	= \int_0^t \int_0^{t-u} g^T(u+s,u)\, ds \, d\tilde M^{T}_u.\quad
   	\text{Moreover, we have } \\
   	& \text{the bound: }	\left| \int_0^{t-u} g^T(u+s,u)\, ds \right|
   	\leq [\tilde{\mu}^{T}]_{\mathrm{Lip}} \frac{t_0}{T c} \int_0^{t-u} f^T(s)\, ds
   	\leq [\tilde{\mu}^{T}]_{\mathrm{Lip}} \frac{t_0}{T c} (1-a_T) \sup_{s \le t_0} \bigl| R^{T}_{-1}(T s) \bigr|.
   \end{align*}
   By the Burkholder--Davis--Gundy inequality, it follows that for all
   \(t \in [0,t_0]\),
   \begin{align*}
   	\mathbb{E}\bigl[ |\mathcal{I}^{K_1^T}_t| \bigr]
   	&\leq 2 \int_0^t
   	\left( \int_0^{t-u} g^T(u+s,u)\, ds \right)^2
   	(\varsigma^T(u))^2 \, d\tilde{\mathcal{I}}^{\Lambda^{*T}}_u \leq [\tilde{\mu}^{T}]_{\mathrm{Lip}} \frac{2 t_0}{c^2}
   	\sup_{s \le t_0} \bigl| R^{T}_{-1}(T s) \bigr|
   	\left( \mathbb{E}\bigl[ \tilde{\mathcal{I}}^{\Lambda^{*T}}_{t_0} \bigr] \right)^{\frac12}
   	\frac{1}{T^2}.
   \end{align*}
   Consequently, by Markov’s inequality and the continuous Borel--Cantelli lemma, we conclude that the convergence \(\mathcal{I}^{K^T_1}_t \underset{T \to +\infty}{\rightarrow} 0\) holds almost surely. \hfill$\Box$ 
   
   \subsection{Proofs of Theorems~\ref{Thm:Limit} and~\ref{Thm:regul}}\label{subsect:proofMresult2}
   \medskip
   \noindent {\bf{Proof of the Theorem~\ref{Thm:Limit}:} \em{Dynamics of the limit points.}}
   The weak convergence of a sequence of integrated diffusion processes does not in general
   imply the convergence of those diffusion processes. Even if the sequence of integrated processes
   would converge, one cannot obtain the limit of the sequence \((\Lambda^{*T})_{T>0}\)  by differentiating the limit
   of the sequence \((\tilde{\mathcal{I}}^{\Lambda^{*T}})_{T>0}\) unless the $\mathcal C$-tightness of the former sequence has been established.
   
   \noindent  {\sc Step~1:} Based on all preceding results, we have \(\left( \tilde{\mathcal{I}}^{\Lambda^{*T}}, \tilde{N}^T, \tilde{M}^{*T},\tilde M^{T} \right) \overset{\mathcal L }{\Rightarrow} \left( \tilde{\mathcal{I}}^{\Lambda^*}, \tilde{N}^*, \tilde{M}^{*}, \tilde M \right)\) in the Skorokhod space \( D([0,t_0]; \mathbb{R}_+^2 \times \mathbb{R}^2).\)
   By the Skorokhod representation theorem \footnote{The Skorokhod representation theorem yields the existence of c\`adl\`ag processes \(\big(\hat{X}^{k}\big)_{k\geq1}\) and $\hat{X}$ all defined on a common probability space $(\hat{\Omega}, \hat{\mathcal{F}}, \hat{\mathbb{P}})$, such that for any \(k \ge 1\), $ \hat{X}^{k} \overset{d}{=} \hat{X}^{T_k}$, \(\hat{X} \overset{d}{=} X\) and for  any \(\omega \in \hat{\Omega}\), \(\hat{X}^{k}(\omega) \underset{k \to \infty}{\to} \hat{X}(\omega)\) in the Skorokhod topology, i.e. and $\hat{X}^{k} \to \hat{X}$
   	in $C([0, t_0]; \mathbb{R}^3)$ as $k \to \infty$, $\hat{\mathbb{P}}$-a.s.} \cite[Theorem 2.7]{IkedaWatanabe1989}, \cite[ Theorem 1.6.7]{Billingsley1999}, we may without loss of generality, assume that the preceding limit holds almost surely. By the continuity of the limiting processes \(X = \left( \tilde{\mathcal{I}}^{\Lambda^*}, \tilde{N}^*, \tilde M^*, \tilde M \right) \) the convergence holds almost surely and uniformly. That is, for any \( t_0 \geq 0 \), \(\sup_{t \in [0, t_0]} |\tilde{\mathcal{I}}^{\Lambda^{*T}}_t - \tilde{\mathcal{I}}^{\Lambda^*}_t| \underset{T \to +\infty}{\rightarrow} 0,\)
   \[
   \sup_{t \in [0, t_0]} |\tilde{N}^{*T}_t - \tilde{N}^*_t| \underset{T \to +\infty}{\rightarrow} 0,\; \sup_{t \in [0, t_0]} |\tilde{M}^{*T}_t - \tilde M^*_t| \underset{T \to +\infty}{\rightarrow} 0\;
   \textit{and} \; 
   \sup_{t \in [0, t_0]} |\tilde{M}^{T}_t - \tilde M_t| \underset{T \to +\infty}{\rightarrow} 0 \quad \text{a.s.}.
   \]	
   \noindent  {\sc Step~2} {\em (Martingale):} 
   In addition to the limit of the sequence \( (\tilde M^{T})_{T >0} \) established above, the Dambis-Dubin-Schwarz theorem gives the
   existence of a Brownian motion \(W\)\footnote{Note that even if $\langle \tilde{M}^* \rangle_\infty$ is not equal to $\infty \quad \text{a.s}$ , this Brownian motion exists up to an enlargement of the filtration.} such that \(\tilde M^*_t = W_{\tilde{\mathcal{I}}^{\Lambda^{*}}_t} \). By the martingale representation theorem (see Theorem 7.1 in \cite[ p.~84]{IkedaWatanabe1989}), the continuous martingale \(\tilde M^*_t \) can be represented as \(\tilde M^*_t = \int_0^t \sqrt{\Lambda^*_s} \, dW_s, \; t \geq 0.\) Moreover, as \( \langle \tilde{M} \rangle_t = \int_{0}^{t} \varsigma^2(s) \Lambda^*(s) ds\), still by the martingale representation theorem, the continuous martingale \(\tilde{M}\) can be represented as  \(\tilde{M}_t = \int_0^t \varsigma(s) \sqrt{\Lambda^*_s} \, dW_s, \; t \geq 0.\)\\
   
   \noindent  {\sc Step~3} {\em (Limit identification):} 
   The first claim of this third point comes directly from the definition of $\mathcal C$-tightness and the fact that each \(\tilde{\mathcal{I}}^{\Lambda^{*}}\) is non-decreasing,
   nonnegative, and starts from 0 .
   
   \noindent We now prove the weak convergence of the sequence \( (\tilde{\mathcal{I}}^{\Lambda^{*T}})_{T>0} \). 
   With the almost sure convergence established in Proposition~\ref{prop:uniformCvgce} above, we are now in a position to show that \( \mathcal{I}^{K^{T}} \) converges almost surely in \( \mathcal{C}([0,t_0]; \mathbb{R}) \) to the limit
   \[ Y_t := \frac\nu\lambda\int_0^t f_{\alpha,\lambda}(t - s) \tilde M_s \, ds = \frac\nu\lambda\int_0^t f_{\alpha,\lambda}(t - s) \left( \int_0^s \varsigma(u)\sqrt{\Lambda^*_u} \, dW_u \right)\, \quad t \geq 0.\]
   Setting \(F_{\alpha,\lambda}(t):= \int_0^t f_{\alpha,\lambda}(s) ds \), the above can by rewritten, owing to an integration by part (IbP):
   \[Y_t = \frac\nu\lambda\int_0^t f_{\alpha,\lambda}(t - s) \left( \int_0^s \varsigma(u)\sqrt{\Lambda^*_u} \, dW_u \right)\,ds = \frac\nu\lambda\int_0^t F_{\alpha,\lambda}(t - s) \varsigma(s)\sqrt{\Lambda^*_s} \, dW_s \quad t \geq 0.\]
   For the term \( K^T \), since \( \mathcal{I}^{K^T}_t = \mathcal{I}^{K^T_1}_t + \mathcal{I}^{K^T_2}_t \) with 
   \( \mathcal{I}^{K^T_1}_t \underset{T \to +\infty}{\rightarrow} 0 \) still by proposition~\ref{prop:uniformCvgce}, 
   it remains to prove that 
   \( \mathcal{I}^{K^T_2} \underset{T \to +\infty}{\overset{\text{a.s.}}{\rightarrow}} Y \) 
   in \( \mathcal{C}([0,t_0]; \mathbb{R}) \). We write:
   \begin{align*}
   	\mathcal{I}^{K^T_2}_t &= \int_0^t (1 - a_T) \left[ \int_0^s  R^{\prime T}_{-1}(T(s - u)) \, \frac{dM^T_{Tu}}{\tilde\mu^T(\frac{u}{T})} \,\right] ds = \int_0^t \int_0^s \frac{1 - a_T}{\sqrt{\tilde\mu^T(\frac{u}{T})}} R^{\prime T}_{-1}(T(s - u)) \cdot \frac{dM^T_{Tu}}{\sqrt{\tilde\mu^T(\frac{u}{T})}}ds \\
   	&= \frac\nu\lambda\int_0^t\int_0^s  T(1 - a_T) R^{\prime T}_{-1}(T(s - u)) \varsigma^T(u)d\tilde M^{*T}_u ds =\frac\nu\lambda \int_0^t\int_0^s  T(1 - a_T) R^{\prime T}_{-1}(T(s - u)) d\tilde M^{T}_u ds
   \end{align*}
   Then, applying the stochastic Fubini theorem \cite[ Theorem 2.6]{Walsh1986},\cite{Veraar2012}, we write:
   \begin{align*}
   	\mathcal{I}^{K^T_2}_t &=\frac\nu\lambda \int_0^t\int_0^s  T(1 - a_T) R^{\prime T}_{-1}(T(s - u)) d\tilde M^{T}_u ds =\frac\nu\lambda \int_0^t\int_u^t  T(1 - a_T) R^{\prime T}_{-1}(T(s - u))  ds d\tilde M^{T}_u\\
   	&=  \frac\nu\lambda\int_0^t \int_0^{t-u}  T(1 - a_T)R^{\prime T}_{-1}(Ts)ds d\tilde M^T_{u}  =\frac\nu\lambda \int_0^t F^T(t-s) d\tilde M^T_{s},  \; \text{where we set}
   \end{align*}
   \(F^T(t):= \int_0^t T(1 - a_T)R^{\prime T}_{-1}(Ts)ds \). By integration by parts, from the end of the first line, we also write:
   \begin{align*}
   	\mathcal{I}^{K^T_2}_t = \frac\nu\lambda\int_0^t\int_u^t  T(1 - a_T) R^{\prime T}_{-1}(T(s - u))  ds d\tilde M^{T}_u=  \frac\nu\lambda\int_0^t  T(1 - a_T)R^{\prime T}_{-1}(T(t-s)) \tilde M^T_{s} ds
   \end{align*}
   By the triangle inequality, we have \(\left| \mathcal{I}^{K^T}_t - Y_t \right| 
   \leq \left| \mathcal{I}^{K^T_1}_t \right| + \left| \delta^{T, \text{approx}}_t \right| 
   + \left| \delta^{T, \text{mart}}_t \right|,\)
   where the individual components of the discrepancy or error terms are defined as follows (\(f^{T}(\cdot):=T(1 - a_T) R^{\prime T}_{-1}(T\cdot)\)):
   {\small
   	\begin{align*}
   		\delta^{T, \text{approx}}_t
   		:= \frac\nu\lambda\int_0^t \left( f^{T}(t - s) - f_{\alpha,\lambda}(t - s) \right) \tilde M_s ds \quad \text{and}\quad
   		\delta^{T, \text{mart}}_t 
   		:= \frac\nu\lambda\int_0^t f^{T}(t - s) \left( \tilde M^T_s - \tilde M_s \right) ds
   	\end{align*}
   }
   \noindent Using the second claim of Proposition~\ref{Lem:Lhawkes} and the uniform almost sure convergence from Equation~\eqref{eqCvgce}, we obtain for $\delta^{T, \text{mart}}_t$, as $T \to \infty$,
   \begin{equation}
   	\sup_{t \in [0, t_0]} \left|\delta^{T, \text{mart}}_t \right| 
   	\leq \frac\nu\lambda\left( \int_0^{t_0} f^{T}(t - s) \, ds \right)
   	\cdot \sup_{t \in [0, t_0]} \left| \tilde M^T_t - \tilde M_t \right| \underset{T \to +\infty}{\overset{\text{a.s.}}{\rightarrow}} 0.
   \end{equation}
   For the second error term $\delta^{T,\text{approx}}$, using the fact that $\tilde M$ is also $L^1$-bounded and continuous, and reasoning as in Step~1 of the proof of Proposition~\ref{prop:Lhawkes_tightness}, we prove the tightness of $(\delta^{T,\text{approx}})_{T \geq 0}$. 
   Note that at this stage, we already have at least the \(L^2-\) convergence of \(\delta^{T, \text{approx}}_t\) to \(0\). Indeed, by an integration by parts, we can write \(\delta^{T, \text{approx}}_t = \frac\nu\lambda\int_0^t (F^T(t-s)-F_{\alpha,\lambda}(t-s)) d\tilde M_{s}.\) Then,
   
   \[
   \mathbb{E} \left[ \left( \int_0^t \left(F^T(t - s)-F_{\alpha,\lambda}(t-s) \right) d\tilde M_{s} \right)^2 \right]
   \leq \|\varsigma\|^2_{t_0} \mathbb{E}[\mathcal{I}^{\Lambda^{*}}_{t_0}] \| F^T-F_{\alpha,\lambda}\|^2_{t_0} \underset{T \to +\infty}{\rightarrow} 0
   \]
   thanks to the claim (2) of the Proposition \ref{Lem:Lhawkes} and the first claim of Proposition \ref{prop:Lhawkes_tightness} (tightness of \(\mathcal{I}^{\Lambda^{*T}}_t\)).
   To prove the almost sure convergence for each \( t \geq 0 \), we know from Proposition \ref{Lem:Lhawkes} that the finite measure \(m^T\) with density function \(f^{T}(\cdot) =T(1 - a_T) R^{\prime T}_{-1}(T\cdot)\) converges weakly to the finite measure \(m^*\)
   with density function $f_{\alpha, \lambda}$, i.e. $m^{T}(ds)\Rightarrow m^*(ds)$ weakly and hence, by the continuity of $\tilde M$, we have
   \begin{eqnarray*}
   	\int_0^t f^{T}(Ts) \tilde{M}(t-s)\,ds 
   	= \int_0^t \tilde{M}(t-s) m^{T}(ds) 
   	\overset{\rm a.s.}\longrightarrow \int_0^t \tilde{M}(t-s) m^*(ds) 
   	= \int_0^t f_{\alpha, \lambda} \tilde{M}(t-s)\,ds
   \end{eqnarray*}
   which immediately yields that $\delta^{T, \text{approx}}_t \overset{\rm a.s.}\to 0$ as $T\to\infty$ and hence $\delta^{T, \text{approx}} \Rightarrow 0$ weakly in the space $\mathcal{C}([0,t_0];\mathbb{R})$. 
   Based on preceding results, we conclude that $\mathcal{I}^{K^T} \Rightarrow Y$ weakly in $\mathcal{C}([0,t_0];\mathbb{R})$ and also \(\sup_{t\in[0,T]} \big| \mathcal{I}^{K^T}_t- Y_t \big| \overset{\rm a.s.}\to 0 .\)
   From this and the definition of \( \mathcal{I}^{K^*} \), we conclude that:
   \begin{align*}
   	\int_0^t K^*_s \, ds =: \mathcal{I}^{K^*}_t &\overset{\text{a.s.}}{=}
   	\frac1\lambda \int_0^t f_{\alpha,\lambda}(t - s) \left( \int_0^s \varsigma(u)\sqrt{\Lambda^*_u} \, dW_u \right)\, ds := Y_t \\
   	&=\frac1\lambda\int_0^t F_{\alpha,\lambda}(t - s) \varsigma(s)\sqrt{\Lambda^*_s} \, dW_s = \frac1\lambda \int_0^t \int_0^s f_{\alpha,\lambda}(s - r) \varsigma(r)\sqrt{\Lambda^*_r} \, dW_r \, ds.
   \end{align*}
   where, the penultimate equality comes from an integration by part and the last equality is deduced in the same way as for the proof of Theorem 3.2 in \cite{JaissonRosenbaum2016}.
   
   \noindent Therefore, letting \( T \to \infty \) in the equation \(\tilde{\mathcal{I}}^{\Lambda^{*T}}_t = \mathcal{I}^{I^T}_t + \mathcal{I}^{J^T}_t + \mathcal{I}^{K^T}_t\), and using Proposition \ref{prop:Lhawkes_tightness} (2), it is not difficult to see that \(\tilde{\mathcal{I}}^{\Lambda^{*T}}\) converges weakly to \( \tilde{\mathcal{I}}^{\Lambda^{*}}\) in \( D([0,t_0]; \mathbb{R}_+) \), where \(  \tilde{\mathcal{I}}^{\Lambda^{*}} \) is given as follows: %is the unique solution of the ODE: 
   \begin{equation} \label{eq:LimitIntegratedProcess}
   	\tilde{\mathcal{I}}^{\Lambda^{*}}_t = \int_0^t \left( \phi(s) - (f_{\alpha,\lambda}* \phi)(s )\right) ds 
   	+ \int_0^t \left( \int_0^s f_{\alpha,\lambda}(s - u) \theta(u) \, du \right) ds 
   	+ \frac{\nu}{\lambda} \int_0^t \int_0^s f_{\alpha,\lambda}(s - r) \varsigma(r)\sqrt{\Lambda^*_r} \, dW_r \, ds.
   \end{equation}
   We next want to identify the desired limit of the integrand processes by diffferentiation.
   More precisely, suppose that the sequence \( \{\Lambda^{*T}\}_{T \geq 1} \) converges weakly in \( D([0, t_0]; \mathbb{R}_+) \) to a process \( \Lambda^{*} \). Then, this implies weak convergence in \( L^1_{\mathrm{loc}}([0, t_0]; \mathbb{R}_+) \), and consequently, the sequence of integrated processes \(\tilde{\mathcal{I}}^{\Lambda^{*T}}_t := \int_0^t \Lambda^{*T}_s \, ds, \quad t \in [0, t_0],\)
   also converges weakly in \( D([0, t_0]; \mathbb{R}_+) \) to the random variable \(\tilde{\mathcal{I}}^{\Lambda^{*}}_t := \int_0^t \Lambda^{*}_s \, ds, \quad t \in [0, t_0].\)
   That is, as \( T \to \infty \),
   \begin{equation}
   	\Lambda^{*T} \overset{\mathcal L }{\Rightarrow} \Lambda^{*} \quad \text{in } D([0, t_0]; \mathbb{R}_+) \quad \Rightarrow \quad \tilde{\mathcal{I}}^{\Lambda^{*T}} \overset{\mathcal L }{\Rightarrow} \tilde{\mathcal{I}}^{\Lambda^{*}} := \int_0^{\cdot} \Lambda^{*}_s \, ds \quad \text{in } D([0, t_0]; \mathbb{R}_+). 
   \end{equation}
   Moreover, as the integrand or limiting process \( \Lambda^{*} \) is continuous as follows from the \( \mathcal{C} \)-tightness established in Proposition~\ref{prop:tightIntensity}, it can be recovered by differentiating the limit \( \tilde{\mathcal{I}}^{\Lambda^{*}} \), i.e. \(	\Lambda^{*}_t = \frac{d}{dt} \tilde{\mathcal{I}}^{\Lambda^{*}}_t, \; t \in [0, t_0].\)
   Consequently, the equality  or SDE in ~\eqref{eq:limitInt} satisfied by \(\Lambda^{*}\) follows by differentiating both sides of the preceding equality~\eqref{eq:LimitIntegratedProcess} with respect to \( t \). We deduce the expression~\eqref{eq:limitThm} stated in the theorem directly as a consequence of Proposition~\ref{prop:wiener_hopf} in the particular case of the \(\alpha\)-fractional integration kernel (see also~\eqref{eq:limitIntAffine2}).
   
   \medskip
   \noindent The weak uniqueness in law follows directly from \cite[Theorem 3.6]{EGnabeyeuPR2025}, which provides the exponential-affine transform formula for the measure Laplace transform of \(\Lambda^{*}\). Applying this result to the specific choice \(\mu_{t_1,\dots,t_n}(ds) = \sum_{i=1}^{n} u_i \delta_{t_n - t_i}(ds)\), where \(0 \leq t_1 < \dots < t_n\), \(u_1, \dots, u_n \in \mathbb{R}_-\), and \(n \in \mathbb{N}\), yields the Laplace transform \(\mathbb{E}\left[\exp\left(\sum_{i=1}^{n} u_i\, \Lambda^{*}_{t_i}\right)\right]\)
   which uniquely determines the law of \(\Lambda^{*}\).
   
   \medskip
   \noindent  {\sc Step~4} {\em (Moment control):}  Let's now control the moment or prove the moment estimate for \(\Lambda^*\):\\
   Taking expectations on both sides of equation\eqref{eq:limitInt} and using linearity and boundedness assumptions on the kernels, we get:
   \[\mathbb{E} \left[ |\Lambda^*_t| \right] = \mathbb{E} \left[ \Lambda^*_t \right] =  \mathbb{E} \left[ \Lambda^*_0 \right]\big(\phi(t)- \int_0^t f_{\alpha, \lambda}(t-s)\phi(s)ds\big) + \int_0^t f_{\alpha, \lambda}(t-s)\theta(s)ds\,, \quad t \geq 0.\]
   The integrability of \(f_{\alpha, \lambda}\) is sufficient for the boundedness of the first moment.
   If  \(\lim_{t \to +\infty} \theta (t) = \theta_\infty \) and \(\lim_{t \to +\infty} \phi (t) = \phi_\infty \), then, owing to the fact that \(f_{\alpha, \lambda}\)  is a probability density, we have :
   \(\lim_{t \to +\infty} \mathbb{E} \left[ \Lambda^*_t \right] = \theta_\infty\)
   and 
   \begin{align*}
   	\sup_{t \in [0, t_0]} \mathbb{E} \left[ |\Lambda^*_t| \right] &\leq  \mathbb{E} \left[ |\Lambda^*_0| \right] \|\phi\|_{t_0} \sup_{t \in [0, t_0]} | 1+(f_{\alpha,\lambda} * 1)(t)|  + \sup_{t \in [0, t_0]}  \int_0^t f_{\alpha, \lambda}(t-s)|\theta(s)|ds\,\\
   	&\leq  C (1 + \|\phi\|_{t_0} \mathbb{E} \left[ |\Lambda^*_0| \right]  )
   \end{align*}
   where \( C > 0 \) is a constant depending on \( \theta \) and \( f_{\alpha,\lambda} \). Let \(p\geq2\), by the triangle inequality, we write:
   \[\Big\||\Lambda^*_{t}|\Big\|_{p} \leq \Big\| |\big((\phi - f_{\alpha,\lambda} * \phi)(t)\big) \Lambda^*_0| \Big\|_p + \big|\int_0^t f_{\alpha, \lambda}(t-s)\theta(s)ds\,\big| + \left\| |\int_0^{t}f_{\alpha,\lambda}(t-u)\varsigma(u)\sqrt{\Lambda^*_{u}}\mathrm{d}W_u |\right\|_p\]  
   where, the second term come from the non-randomness of the term involved and its continuity under assumption \ref{assum:mu} implies its boundeness  in \([0, t_0]\) by a positive constant K and we have:
   $$\sup_{t \in [0, t_0]}\Big\| |\big((\phi - f_{\alpha,\lambda} * \phi)(t)\big) \Lambda^*_0| \Big\|_p \leq 	\Big\| | \Lambda^*_0| \Big\|_p \|\phi\|_{t_0}  \sup_{t \in [0, t_0]}| 1+(f_{\alpha,\lambda} * 1)(t)|  \leq C^{\prime} \|\phi\|_{t_0} \Big\| | \Lambda^*_0| \Big\|_p  $$
   Combining the $L^p$-BDG and the generalized Minkowski inequality yields for the last term,
   \begin{align*}
   	&\; \left\| |\int_0^{t}f_{\alpha,\lambda}(t-u)\varsigma(u)\sqrt{\Lambda^*_{u}}\mathrm{d}W_u |\right\|_p \le C^{BDG}_p \Big(\int_0^tf^2_{\alpha,\lambda}(t-u) \varsigma^2(u)\big\||\sqrt{\Lambda^*_{u}}|^2\big\|_\frac{p}{2} du\Big)^{1/2} \\
   	&\hspace{3cm}\le C^{BDG}_p\|\varsigma^2\|_{\infty}\sup_{u \in [0, t_0]}\Big\||\Lambda^*_{u}|\Big\|^{\frac12}_\frac{p}{2} \Big(\int_0^{+\infty}f^2_{\alpha,\lambda}(u) du\Big)^{1/2} \le C^{BDG}_p C_{p,\varsigma} \sup_{u \in [0, t_0]}\Big\||\Lambda^*_{u}|\Big\|^{\frac12}_\frac{p}{2}.
   \end{align*}
   where in the second inequality we used that \(f_{\alpha, \lambda} \in L^2(\text{Leb}_1)\) and the fact that \(\varsigma\) is a borel bounded function. Combining all bounds, we obtain:
   \begin{equation}\label{eq:final_bd}
   	\sup_{t \in [0, t_0]}	\Big\||\Lambda^*_{t}|\Big\|_{p} \leq C^{\prime}\|\phi\|_{t_0} \Big\| | \Lambda^*_0| \Big\|_p + K + C^{\prime}(p)\sup_{u \in [0, t_0]}\Big\||\Lambda^*_{u}|\Big\|^{\frac12}_\frac{p}{2}
   \end{equation} 
   We can 
   now proceed either by induction or alternatively, noticing that from equation~\eqref{eq:final_bd},
   there exists a constant \(C(p)\) such that:
   
   \(\qquad \hspace{2cm}1 + \sup_{t \in [0, t_0]}\mathbb{E} \left[ \left| \Lambda^*_t \right|^{p} \right]
   \leq \left( 1 + \|\phi\|^p_{t_0}\mathbb{E} \left[ \left| \Lambda^*_0 \right|^{p} \right] \right) C(p) \left(1 + \sup_{t \in [0, t_0]}\mathbb{E} \left[ \left| \Lambda^*_t \right|^{\frac{p}{2}} \right] \right)\)
   
   \noindent Letting \( n \in \mathbb{N} \) be the smallest integer such that \(1 \leq  \frac{p}{2^n} \leq 2 \), we obtain
   \begin{align*}
   	1 + \sup_{t \in [0, t_0]}\mathbb{E} \left[ \left| \Lambda^*_t \right|^{p} \right] &\leq \prod_{k=0}^{n-1} \left( 1 + \|\phi\|^{\frac{p}{2^k}}_{t_0}\mathbb{E} \left[ \left| \Lambda^*_0 \right|^{\frac{p}{2^k}} \right] \right)C\left(\frac{p}{2^k}\right) \left(1 + \sup_{t \in [0, t_0]}\mathbb{E} \left[ \left| \Lambda^*_t \right|^{p/2^n}\right]\right)\\
   	&\leq \prod_{k=0}^{n-1} \left( 1 + \|\phi\|^{\frac{p}{2^k}}_{t_0}\mathbb{E} \left[ \left| \Lambda^*_0 \right|^{\frac{p}{2^k}} \right] \right)C\left(\frac{p}{2^k}\right) \left(1 + \sup_{t \in [0, t_0]}\left(\mathbb{E} \left[ \left| \Lambda^*_t \right|^{2}\right]\right)^{\frac{p}{2^{n+1}}}\right)\\
   	&\leq \prod_{k=0}^{n-1} \left( 1 + \left(\|\phi\|^{p}_{t_0}\mathbb{E} \left[ \left| \Lambda^*_0 \right|^{p} \right] \right)^\frac{1}{2^k} \right)C\left(\frac{p}{2^k}\right) \left(1 + \sup_{t \in [0, t_0]}\mathbb{E} \left[ \left| \Lambda^*_t \right|^{2}\right]\right)
   \end{align*}
   Thus, we can use the previous estimate to find
   \begin{align*}
   	&\;	1 + \sup_{t \in [0, t_0]}\mathbb{E} \left[ \left| \Lambda^*_t \right|^{2}\right] \leq ( 1 + \|\phi\|^2_{t_0}\mathbb{E} [ \left| \Lambda^*_0 \right|^{2} ] ) C(2)(1 + \sup_{t \in [0, t_0]}\mathbb{E} [ \left| \Lambda^*_t \right|]) \leq K C(2)(1+C) (1 + \|\phi\|_{t_0} \mathbb{E} \left[ |\Lambda^*_0| \right]  )\\
   	&\hspace{1.8cm}\leq K \cdot C(2)(1+C) \left(1 + \left( \|\phi\|_{t_0}^p \mathbb{E} \left[ |\Lambda^*_0|^p \right] \right)^\frac1p  \right) \leq K\cdot C(2)(1+C) \left(1 + \left( \|\phi\|_{t_0}^p \mathbb{E} \left[ |\Lambda^*_0|e the^p \right] \right)^\frac{1}{2^n}  \right) \\
   	&\,\text{Consequently,} \;	1 + \sup_{t \in [0, t_0]}\mathbb{E} \left[ \left| \Lambda^*_t \right|^{p} \right]
   	\leq \prod_{k=0}^{n} \left( 1 + \left(\|\phi\|^{p}_{t_0}\mathbb{E} \left[ \left| \Lambda^*_0 \right|^{p} \right] \right)^\frac{1}{2^k} \right) \left( K \cdot C(2)(1+C) \prod_{k=0}^{n-1} C\left(\frac{p}{2^k}\right) \right) \\
   	&\hspace{5.2cm}\leq \tilde C\left(p\right) \left( 1 + \|\phi\|^{p}_{t_0}\mathbb{E} \left[ \left| \Lambda^*_0 \right|^{p} \right]  \right)
   \end{align*}	
   Hence the moment bound holds for all \(p \ge 2\).
   The \(L^{p}\)-marginal bound for all \(p>0\) extends by noticing that (in particular this is valid for \(0<p<2\)) \(\sup_{t \in [0,t_0]} \mathbb{E}\!\left[\, |\Lambda_t^{*}|^{p} \,\right]
   \le
   \mathbb{E}\!\left[\, \sup_{t \in [0,t_0]} |\Lambda_t^{*}|^{p} \,\right],\)
   which is a direct consequence of Jensen's inequality (or monotonicity of expectation),
   and by using the claim~(3) of Theorem~\ref{Thm:regul}.
   %	applying the splitting argument from \cite[Lemma D.1 (Splitting lemma)]{JouPag22} with $\bar{p}=2$ and the functional  $\Phi(x,y) = \sup_{t \in [0,t_0]} |x(t)|$. 
   This proves the assertion and we are done.\hfill $\Box$
   
   \medskip
   \noindent {\bf Proof of Theorem \ref{Thm:regul}:}.
   As a consequence of Theorem \ref{Thm:Limit}, we derive that:
   $$ \sup_{t \geq 0} \Big\| |\Lambda^*_t | \Big\|_p \leq C_{p,t_0} \cdot \left( 1 + \|\phi\|_{t_0}\Big\| |\Lambda^*_0 | \Big\|_p \right) < +\infty. $$
   \noindent  {\sc Step~1} {\em (Kolmogorov's continuity criterion).} 
   Now, we can establish H\"older continuity by the Kolmogorov's Continuity theorem. Let \( p\geq2 \) be given. One writes  for $s,\,t\ge 0$ with $s\le t$ and owing to equation ~\ref{eq:limitInt}:
   \[\Lambda^*_t-\Lambda^*_s = \big((\phi - f_{\alpha,\lambda} * \phi)(t)-(\phi - f_{\alpha,\lambda} * \phi)(s)\big) \Lambda^*_0 + \frac{1}{\lambda}\Big(J(t)-J(s)\Big) + I(t)-I(s)\]
   where we set: $J(t):=\int_0^tf_{\alpha,\lambda}(t-u)\varsigma(u)\sqrt{\Lambda^*_{u}}dW_u$ and $ I(t) =  \int_0^t f_{\alpha,\lambda}(t-u) \theta(u)  \, du$
   
   \noindent On the first hand,
   \begin{align*}
   	&\, \Big|(f_{\alpha,\lambda} * \phi)(t) - (f_{\alpha,\lambda} * \phi)(s) \Big|
   	= \Big|\int_0^s \left[ f_{\alpha,\lambda}(t - u) - f_{\alpha,\lambda}(s - u) \right] \phi(u) \, du 
   	+ \int_s^t f_{\alpha,\lambda}(t - u) \phi(u)\, du \Big| \\
   	&\hspace{3.5cm}\leq \|\phi\|_{t_0} \times \Bigg( \int_{0}^{s} |\left( f_{\alpha,\lambda}(t-u) - f_{\alpha,\lambda}(s-u) \right)|du +  \int_{s}^{t} |f_{\alpha,\lambda}(t-u)| du   \Bigg)\quad \text{	so that:}\\
   	%		\end{align*}
   %		\begin{align*}
   	&\Big\| |\big((\phi - f_{\alpha,\lambda} * \phi)(t)-(\phi - f_{\alpha,\lambda} * \phi)(s)\big) \Lambda^*_0| \Big\|_p \le 	\Big\| \Lambda^*_0| \Big\|_p  \Big(|(f_{\alpha,\lambda} * \phi)(s)-(f_{\alpha,\lambda} * \phi)(t) | + |\phi(t)-\phi(s) | \Big) \\
   	&\leq \Big\| |\Lambda^*_0| \Big\|_p \|\phi\|_{t_0} \times \Bigg( C\;|t-s|^{\vartheta} +  \left( \int_0^{+\infty}f_{\alpha,\lambda}^{2\beta}(u)du\right)^{\frac{1}{2\beta}}|t-s|^{1-\frac{1}{2\beta}}  \Bigg)+ C'_{t_0,p} \left( 1 + \|\phi\|_{t_0} \left\| |\Lambda^*_0| \right\|_p \right) |t - s|^{\delta} \\
   	&\leq C_{p,t_0, \beta, f_{\alpha,\lambda}} \left( 1 + \|\phi\|_{t_0} \left\| |\Lambda^*_0| \right\|_p \right) |t-s|^{\vartheta \wedge (1-\frac{1}{2\beta}) \wedge \delta}.
   \end{align*}
   Next we write, \(I(t)-I(s) =\int_0^s \left( f_{\alpha,\lambda}(t-u) - f_{\alpha,\lambda}(s-u) \right)\theta(u)  \, du +\int_s^{t}f_{\alpha,\lambda}(t-u)\theta(u)  \, du\), which is deterministic and by using generalized Minkowski inequalities, one gets from its non-randomness:
   \begin{align*}
   	\Big\| |I(t)-I(s)| \Big\|_p &\le \|\theta\|_{\infty} \times \Bigg( \int_{s}^{t} |f_{\alpha,\lambda}(t-u)| du + \int_{0}^{s} |\left( f_{\alpha,\lambda}(t-u) - f_{\alpha,\lambda}(s-u) \right)|du  \Bigg) \\
   	&\le \|\theta\|_{\infty} \times \Bigg( \Big(\int_0^{+\infty} |f_{\alpha,\lambda}|^{2\beta}(u)du \Big)^{\frac{1}{2\beta}}  (t-s)^{1-\frac{1}{2\beta}} + \int_0^s | f_{\alpha,\lambda}(t-u) - f_{\alpha,\lambda}(s-u) |du \Bigg)\\
   	&\le C_{p, \theta, f_{\alpha,\lambda}}|t-s|^{ (\vartheta\wedge(1-\frac{1}{2\beta}))} < +\infty ; \; \text{so that} \; \Big\| |I(t)-I(s)| \Big\|_p < +\infty \quad \text{uniformly in} \; t_0 \geq 0.
   \end{align*}
   Likewise, as $J(t)-J(s) =\int_0^s \left( f_{\alpha,\lambda}(t-u) - f_{\alpha,\lambda}(s-u) \right)\varsigma(u)\sqrt{\Lambda^*_{u}}dW_u +\int_s^{t}f_{\alpha,\lambda}(t-u)\varsigma(u)\sqrt{\Lambda^*_{u}}\mathrm{d}W_u $, 	combining the $L^p$-BDG, the generalized Minkowski inequality and the growth of $L^p$-norm wrt p \footnote{For every $p, r \in (0, \infty)$, with $r \leq p$, we have:
   	\(
   	\left( \int |X_t|^r \, d\mu(t) \right)^{1/r} 
   	\leq 
   	\left( \int |X_t|^p \, d\mu(t) \right)^{1/p}.
   	\)} yields:
   \begin{align*}
   	&\ \Big\| |J(t)-J(s)| \Big\|_p \le \left\| |\int_s^{t}f_{\alpha,\lambda}(t-u)\varsigma(u)\sqrt{\Lambda^*_{u}}\mathrm{d}W_u |\right\|_p + \left\| |\int_0^s \left( f_{\alpha,\lambda}(t-u) - f_{\alpha,\lambda}(s-u) \right)\varsigma(u)\sqrt{\Lambda^*_{u}}dW_u | \right\|_p\\
   	&\leq C_p \|\varsigma^2\|_{\infty} \left[\Big(\int_s^tf^2_{\alpha,\lambda}(t-u)\big\||\sqrt{\Lambda^*_{u}}|^2\big\|_\frac{p}{2} du\Big)^{\frac12} + \Big(\int_0^s \big(f_{\alpha,\lambda}(t-u)-f_{\alpha,\lambda}(s-u)\big)^2 \big\||\sqrt{\Lambda^*_{u}}|^2\big\|_\frac{p}{2} du\Big)^{\frac12}\right]\\
   	&\leq C_p \|\varsigma^2\|_{\infty} \sup_{u\ge 0}\Big\||\Lambda^*_{u}|\Big\|^{\frac12}_{p} \left[\Big(\int_0^{+\infty}f^{2\beta}_{\alpha,\lambda}(u) du\Big)^{\frac{1}{2\beta}} |t-s|^{\frac 12 (1-\frac{1}{\beta})} + \left(\int_0^{+\infty} \big(f_{\alpha,\lambda}(t-s+u)-f_{\alpha,\lambda}(u)\big)^2du\right)^{\frac12}\right] \\
   	&\leq C_{p,t_0, \varsigma, f_{\alpha,\lambda}} \cdot \left( 1 + \|\phi\|_{t_0}\Big\| |\Lambda^*_0 | \Big\|_p \right)|t-s|^{ \vartheta\wedge\frac{\beta -1}{2\beta}} \quad \text{where} \quad C_p  \equiv C^{BDG}_p .
   \end{align*}
   Finally, putting all these estimates together, for any $p\geqslant p_{eu}:=\frac{1}{\delta} \vee\frac{1}{\vartheta} \vee \frac{2\beta}{\beta-1}, t_0> 0$, the solution \(\Lambda^*\) satisfies (up to a $\mathbb P$-indistinguishability or a path-continuous version $\tilde{\Lambda}^*$ ):	
   \begin{equation}\label{eq:cont}
   	\E\, \left(|\Lambda^*_t-\Lambda^*_s| \right)^p \le C_{p,t_0, \varsigma, \beta, f_{\alpha,\lambda}}\cdot \left( 1 + \|\phi\|_{t_0}^{p}\mathbb{E}[|\Lambda^*_0|^{p}] \right)|t-s|^{p(\delta \wedge\vartheta\wedge \frac{\beta-1}{2\beta})}
   \end{equation}
   And thus, \( t \mapsto \Lambda^*_t \) admits a H\"older continuous modification (still denoted \( \Lambda^* \) in lieu of \( \tilde{\Lambda}^* \) up to a \( \mathbb{P} \)-indistinguishability), so that the process has the announced H\"older pathwise regularity, i.e. \( t \mapsto \Lambda^*_t \)
   has a \( \left( \delta \wedge \vartheta \wedge \frac{\beta - 1}{2\beta} - \eta \right) \)-H\"older pathwise continuous \( P \)-modification for sufficiently small \( \eta > 0 \). 
   
   \medskip
   \noindent The claim~(2) of the Theorem follows from [\cite{ZhangXi2010}, Theorem 2.10] or owing to Kolmogorov's continuity criterion (see \cite[Theorem 2.1, p. 26, 3rd edition]{RevuzYor} \footnote{Let $ X$ be a stochastic process with values in the Polish metric space $(S, \rho)$. If there exist positive constants $\alpha,\beta, c > 0 $ such that \(\mathbb E \left[ \rho( X_s, X_t)\right]^\alpha \le  c | s - t| ^{\beta+d}, \quad s,t \in \mathbb R^d\)
   	then $X$ admits a continuous modification and there exists a modification whose paths are H\"{o}lder continuous of order $\gamma$, for every $\gamma \in (0, \frac{\beta}{\alpha})$.} or \cite[Lemma 44.4, Section IV.44, p.100]{RogersWilliamsII}).
   
   \noindent That is, for any \( p > p_{eu} \), \( t_0 > 0 \), and some \( a \in \big(0,(\delta \wedge\vartheta\wedge \frac{\beta-1}{2\beta}) -\frac 1 p\big) \), we have:
   
   \begin{equation}\label{eq:Regul}
   	\mathbb{E}\left[\sup_{s\neq t\in [0,t_0]}\frac{|\Lambda^*_t-\Lambda^*_s|^p}{|t-s|^{ap}} \right]   <C_{a,p,t_0}  \left( 1 + \|\phi\|_{t_0}^{p}\mathbb{E}[|\Lambda^*_0|^{p}] \right)
   \end{equation}
   \noindent An alternative approach to proving the above estimate would involve the use of the Garsia-Rodemich-Rumsey inequality; see Lemma 1.1 in \cite{GRR1070} with \( \psi(u) = |u|^p \) and \( p(u) = |u|^{q + \frac{1}{p}} \) for \( q > \frac{1}{p} \). It states that for a continuous function \( f \) on \( \mathbb{R}_+ \), there exists a constant \( C_{p,q} > 0 \) such that for any \( x_2 > x_1 \geq 0 \),
   \[
   |f(x_2) - f(x_1)| \leq C_{p,q} \cdot |x_2 - x_1|^{q - \frac1p}
   \left( \int_{x_1}^{x_2} \int_{x_1}^{x_2} \frac{|f(s) - f(r)|}{|s - r|^{1 + pq }} \, dr \, ds \right)^{\frac1p}.
   \]
   Before stating the result, we recall the following definition. For \( a \in (0, 1] \), the \(a\)-H\"older coefficient of a real-valued function \( f \) or its \(a\)-H\"older seminorm on \([0, t_0]\) is defined by
   \[
   \|f\|_{C^{a}_{t_0}} := \sup_{0 \leq x_1 < x_2 \leq t_0} \frac{|f(x_1) - f(x_2)|}{|x_1 - x_2|^{a}}.
   \]
   Let \(a \in (0,1]\), we aim to apply the above GRR inequality to the sequence (see for example \cite[Theorem F.1 (GRR Lemma).]{JouPag22})  to the process \( \Lambda^* \) for large value of p, namely \(p>p_a:=\frac{1}{\delta \wedge \vartheta \wedge \frac{\beta - 1}{2\beta} - a}\) and \( q = \tfrac{1}{p} + a \). We now have:
   \[
   \| \Lambda^* \|_{C^a_{t_0}}^p
   = \sup_{0 \leq s< t \leq t_0} \frac{|\Lambda^*_t - \Lambda^*_s|^p}{|t - s|^{pa}}
   \leq C \cdot \int_0^{t_0} du \int_0^{t_0} \frac{|\Lambda^*_u - \Lambda^*_r|^p}{|u - r|^{pa + 2}} \, dr.
   \]
   Taking expectations on both sides of the preceding inequality, and then applying Fubini's theorem along with equation~\eqref{eq:cont}, we obtain:
   \begin{align*}
   	&\	\mathbb{E} \left[ \| \Lambda^* \|_{C^a_{t_0}}^p \right] = \mathbb{E} \left[  \sup_{0 \leq s< t \leq t_0} \frac{|\Lambda^*_t - \Lambda^*_s|^p}{|t - s|^{pa}} \right]
   	\\ &\hspace{1cm} \leq C \cdot C_{p,t_0, \varsigma, \beta, f_{\lambda}}  \cdot \left( 1 + \|\phi\|_{t_0}^{p}\mathbb{E}[|\Lambda^*_0|^{p}] \right) \cdot
   	\int_0^{t_0} du \int_0^{t_0} |u - r|^{p(\delta \wedge \vartheta \wedge \tfrac{\beta - 1}{2\beta}) - pa - 2} \, dr \nonumber \\
   	&\hspace{1cm} \leq C \cdot C_{p,t_0, \varsigma, \beta, f_{\lambda}} \cdot (1 + t_0)^{p(\delta \wedge \vartheta \wedge \tfrac{\beta - 1}{2\beta} - a)}\cdot \left( 1 + \|\phi\|_{t_0}^{p}\mathbb{E}[|\Lambda^*_0|^{p}] \right) \leq C_{p,t_0, \varsigma, \beta,\delta, \vartheta,a, f_{\lambda}} \cdot \left( 1 + \|\phi\|_{t_0}^{p}\mathbb{E}[|\Lambda^*_0|^{p}] \right).
   \end{align*}
   We then conclude that the bound holds uniformly for all \( 
   t_0 \geq 0 \).
   To extend the $L^p$-pathwise regularity  to every \( p \in (0, +\infty) \) such that \( \| |\Lambda^*_0| \|_p < +\infty \),
   one can take advantage of the form of the control to apply the splitting Lemma \cite[Lemma D.1]{JouPag22} with \( p < \bar{p} \), where \(\bar{p} = \frac{1}{\delta \wedge \vartheta \wedge \frac{\beta - 1}{2\beta} - a} \vee p_{\text{eu}}\),
   and considering the functional \( \Phi: \mathcal C ([0, t_0],\R)^2 \to \mathbb{R} \) defined by:
   
   \centerline{$\Phi(x,y) = \sup_{s \neq t \in [0, t_0]} \frac{|x(t) - x(s)|}{|t - s|^a}.$}
   \medskip
   \noindent The marginal bound~\eqref{eq:cont} can be extended likewise by considering 
   \(\Phi(x, y) = |x(t) - x(s)|.\)
   
   \smallskip
   \noindent  {\sc Step~2} {\em (Maximal inequalities).} 
   Now, we can establish the maximal inequality. Let \( p \) be given.
   \[ \text{For each }\; a\!\in \big(0,\delta \wedge\theta\wedge \frac{\beta-1}{2\beta}\big), \; \sup_{0\leq t \leq t_0} |\Lambda^*_t|^p \leq C\left( \sup_{0\leq t \leq t_0} |\Lambda^*_t-\Lambda^*_0|^p + |\Lambda^*_0|^p \right)  \leq C\left( \| \Lambda^* \|_{C^a_{t_0}}^p t_0^{pa} + |\Lambda^*_0|^p	\right) 
   \]
   Taking expectations on both sides of these inequality and then using the claim \ref{eq:Regul} yields:
   \begin{align*}
   	\mathbb{E} \left[ \sup_{0\leq t \leq t_0} |\Lambda^*_t|^p \right] &\leq C\left( t_0^{pa} \mathbb{E} \left[  \sup_{0 \leq s< t \leq t_0} \frac{|\Lambda^*_t - \Lambda^*_s|^p}{|t - s|^{pa}}  \right] + \mathbb{E} \left[  |\Lambda^*_0|^p \right]	\right) \\
   	&\leq  C_{p,t_0, \varsigma, \beta,\delta, \vartheta,a, f_{\lambda}} \left( 1 + \|\phi\|_{t_0}^{p}\mathbb{E}[|\Lambda^*_0|^{p}] \right)
   \end{align*}
   The marginal bound above, valid for all \(p > \bar{p}\), can be extended simirlarly to every \(p \in (0,+\infty) \) by applying the splitting argument from \cite[Lemma D.1]{JouPag22}, using the same \(\bar{p}\) as above and taking the functional \(\Phi(x,y) = \sup_{t \in [0,t_0]} |x(t)|.\) \hfill $\Box$

	\bibliographystyle{plainnat}
	\bibliography{Bibliography}

\begin{thebibliography}{43}
\providecommand{\natexlab}[1]{#1}
\providecommand{\url}[1]{\texttt{#1}}
\expandafter\ifx\csname urlstyle\endcsname\relax
  \providecommand{\doi}[1]{doi: #1}\else
  \providecommand{\doi}{doi: \begingroup \urlstyle{rm}\Url}\fi

\bibitem[A.~A.~Kilbas and Samko(1993)]{Pollard1948}
O.~I.~Marichev A.~A.~Kilbas and S.~G. Samko.
\newblock Fractional integrals and derivatives (theory and applications).
\newblock 1993.

\bibitem[Abi~Jaber et~al.(2019)Abi~Jaber, Larsson, and Pulido]{abi2019affine}
E.~Abi~Jaber, M.~Larsson, and S.~Pulido.
\newblock Affine volterra processes.
\newblock \emph{The Annals of Applied Probability}, 29\penalty0 (5):\penalty0
  3155--3200, 2019.

\bibitem[Bacry et~al.(2013)Bacry, Delattre, Hoffmann, and
  Muzy]{BacryDelattreHoffmannMuzy2013}
E.~Bacry, S.~Delattre, M.~Hoffmann, and J.~F. Muzy.
\newblock Some limit theorems for hawkes processes and application to financial
  statistics.
\newblock \emph{Stochastic Processes and their Applications}, 123\penalty0
  (7):\penalty0 2475--2499, 2013.

\bibitem[Bernstein(1929)]{Bernstein1929}
S.~Bernstein.
\newblock Sur les fonctions absolument monotones.
\newblock \emph{Acta Mathematica}, 52:\penalty0 1--66, 1929.

\bibitem[Billingsley(1999)]{Billingsley1999}
P.~Billingsley.
\newblock \emph{Convergence of Probability Measures}.
\newblock Wiley-Interscience, New York, 2nd edition, 1999.

\bibitem[Bingham et~al.(1989)Bingham, Goldie, and Teugels]{BiGoTe1989}
N.~H. Bingham, C.~M. Goldie, and J.~L. Teugels.
\newblock \emph{Regular {V}ariation}, volume~27 of \emph{Encyclopedia of
  Mathematics and its Applications}.
\newblock Cambridge University Press, Cambridge, 1989.

\bibitem[Callegaro et~al.(2020)Callegaro, Grasselli, and
  Pag\`es]{CallegaroGrasselliPages2020}
G.~Callegaro, M.~Grasselli, and G.~Pag\`es.
\newblock Fast hybrid schemes for fractional riccati equations (rough is not so
  tough).
\newblock \emph{Mathematics of Operations Research}, 46\penalty0 (1):\penalty0
  221--254, 2020.

\bibitem[Daley and Vere-Jones(2002)]{DaleyVereJones2002}
D.~J. Daley and D.~Vere-Jones.
\newblock \emph{An Introduction to the Theory of Point Processes: Elementary
  Theory and Methods}.
\newblock Probability and its Applications. Springer, 2nd edition, 2002.

\bibitem[Dawson and Fleischmann(1994)]{dawson1994super}
D.~A. Dawson and K.~Fleischmann.
\newblock A super-{B}rownian motion with a single point catalyst.
\newblock \emph{Stochastic Processes and their Applications}, 49\penalty0
  (1):\penalty0 3--40, 1994.

\bibitem[El~Euch and Rosenbaum()]{ElEuchR2018}
O.~El~Euch and M.~Rosenbaum.
\newblock Perfect hedging in rough {Heston} models.
\newblock \emph{Ann. Appl. Probab.}

\bibitem[El~Euch and Rosenbaum(2019)]{el2019characteristic}
O.~El~Euch and M.~Rosenbaum.
\newblock The characteristic function of rough {H}eston models.
\newblock \emph{Mathematical Finance}, 29\penalty0 (1):\penalty0 3--38, 2019.

\bibitem[El~Euch et~al.(2018)El~Euch, Fukasawa, and
  Rosenbaum]{ElEuchFukasawaRosenbaum2018}
O.~El~Euch, M.~Fukasawa, and M.~Rosenbaum.
\newblock The microstructural foundations of leverage effect and rough
  volatility.
\newblock \emph{Finance Stoch.}, 22\penalty0 (2):\penalty0 241--280, 2018.

\bibitem[El~Euch et~al.(2019)El~Euch, Gatheral, and
  Rosenbaum]{ElEuchGatheralRosenbaum2019}
O.~El~Euch, J.~Gatheral, and M.~Rosenbaum.
\newblock Roughening heston.
\newblock \emph{Risk}, pages 84--89, 2019.

\bibitem[Garsia et~al.(1970/71)Garsia, Rodemich, and Rumsey]{GRR1070}
A.~M. Garsia, E.~Rodemich, and H.~Rumsey, Jr.
\newblock A real variable lemma and the continuity of paths of some {G}aussian
  processes.
\newblock \emph{Indiana Univ. Math. J.}, 20:\penalty0 565--578, 1970/71.

\bibitem[Gatheral and Jacquier(2014)]{Gatheral2012arbitrage}
J.~Gatheral and A.~Jacquier.
\newblock Arbitrage-free svi volatility surfaces.
\newblock \emph{Quantitative Finance}, 14\penalty0 (1):\penalty0 59--71, 2014.

\bibitem[Gnabeyeu and Pag{\`e}s(2025)]{EGnabeyeu2025}
E.~Gnabeyeu and G.~Pag{\`e}s.
\newblock On a stationarity theory for stochastic volterra integral equations.
\newblock \emph{arXiv preprint arXiv.2511.03474}, 2025.

\bibitem[Gnabeyeu et~al.(2024)Gnabeyeu, Karkar, and
  Idboufous]{GnabeyeuKarkarIdboufous2024}
E.~Gnabeyeu, O.~Karkar, and I.~Idboufous.
\newblock Solving the dynamic volatility fitting problem: A deep reinforcement
  learning approach, 2024.

\bibitem[Gnabeyeu et~al.(2025)Gnabeyeu, Pag\`es, and
  Rosenbaum]{EGnabeyeuPR2025}
E.~Gnabeyeu, G.~Pag\`es, and M.~Rosenbaum.
\newblock \textnormal{On inhomogeneous affine Volterra processes: stationarity
  and applications to the Volterra Heston model.}
\newblock \emph{arXiv preprint arXiv.2512.09590}, 2025.

\bibitem[Gorenflo and Mainardi(1997)]{GorenfloMainardi1997}
R.~Gorenflo and F.~Mainardi.
\newblock Fractional calculus: Integral and differential equations of
  fractional order.
\newblock In \emph{Fractals and Fractional Calculus in Continuum Mechanics},
  pages 223--276. Springer Verlag, 1997.

\bibitem[Gripenberg et~al.(1990)Gripenberg, Norlund, and
  Saavalainen]{gripenberg1990}
B.~Gripenberg, S.~Norlund, and O.~Saavalainen.
\newblock \emph{Volterra Integral and Functional Equations}.
\newblock \emph{Encyclopedia of Mathematics and its Applications}. Cambridge
  University Press, Cambridge, UK, 1990.

\bibitem[Haan and Ferreira(2006)]{DeHaanFerreira2006}
L.~De Haan and A.~Ferreira.
\newblock \emph{\emph{Extreme Value Theory: An Introduction}}.
\newblock Springer Science \& Business Media, 2006.

\bibitem[Heston(1993)]{Heston1993}
S.~L. Heston.
\newblock A closed-form solution for options with stochastic volatility with
  applications to bond and currency options.
\newblock \emph{The Review of Financial Studies}, 6\penalty0 (2):\penalty0
  327--343, 1993.

\bibitem[Horst et~al.(2024)Horst, Xu, and Zhang]{HorstXuZhang2024}
U.~Horst, W.~Xu, and R.~Zhang.
\newblock Convergence of heavy-tailed hawkes processes and the microstructure
  of rough volatility.
\newblock \emph{arXiv preprint arXiv:2312.08784}, 2024.

\bibitem[Ikeda and Watanabe(1989)]{IkedaWatanabe1989}
N.~Ikeda and S.~Watanabe.
\newblock \emph{Stochastic Differential Equations and Diffusion Processes}.
\newblock North-Holland/Kodansha, Amsterdam/Tokyo, 1989.

\bibitem[Jacod and Shiryaev(2013)]{jacod2013limit}
J.~Jacod and A.~Shiryaev.
\newblock Limit theorems for stochastic processes.
\newblock \emph{Springer Science \& Business Media}, 288, 2013.

\bibitem[Jaisson and Rosenbaum(2015)]{JaissonRosenbaum2015}
T.~Jaisson and M.~Rosenbaum.
\newblock Limit theorems for nearly unstable hawkes processes.
\newblock \emph{Ann. Appl. Probab.}, 25\penalty0 (2):\penalty0 600--631, 2015.

\bibitem[Jaisson and Rosenbaum(2016)]{JaissonRosenbaum2016}
T.~Jaisson and M.~Rosenbaum.
\newblock Rough fractional diffusions as scaling limits of nearly unstable
  heavy tailed hawkes processes.
\newblock \emph{Ann. Appl. Probab.}, 26\penalty0 (5):\penalty0 2860--2882,
  2016.

\bibitem[{Jourdain} and {Pag{\`e}s}(2022)]{JouPag22}
B.~{Jourdain} and G.~{Pag{\`e}s}.
\newblock {Convex ordering for stochastic Volterra equations and their Euler
  schemes}.
\newblock \emph{\em Fin. \& Stoch.}, art. arXiv:2211.10186, November 2022.

\bibitem[Karatzas and Shreve(1991)]{karatzas1991}
I.~Karatzas and S.~E. Shreve.
\newblock \emph{\emph{Brownian motion and stochastic calculus}}, volume 113 of
  \emph{\emph{Graduate Texts in Mathematics}}.
\newblock Springer-Verlag, New York, second edition, 1991.

\bibitem[Kurtz and Protter(1991)]{Kurtz1991}
T.G. Kurtz and P.~Protter.
\newblock Weak limit theorems for stochastic integrals and stochastic
  differential equations.
\newblock \emph{The Annals of Probability}, 19\penalty0 (3):\penalty0
  1035--1070, jul 1991.

\bibitem[Mainardi(2014)]{Mainardi2014}
F.~Mainardi.
\newblock On some properties of the {M}ittag-{L}effler function
  {$E_\alpha(-t^\alpha)$}, completely monotone for {$t>0$} with {$0<\alpha<1$}.
\newblock \emph{Discrete Contin. Dyn. Syst. Ser. B}, 19\penalty0 (7):\penalty0
  2267--2278, 2014.

\bibitem[Mainardi and Gorenflo(2000)]{GorMain2000}
F.~Mainardi and R.~Gorenflo.
\newblock Fractional calculus: special functions and applications.
\newblock In \emph{Advanced special functions and applications ({M}elfi,
  1999)}, volume~1 of \emph{Proc. Melfi Sch. Adv. Top. Math. Phys.}, pages
  165--188. Aracne, Rome, 2000.

\bibitem[Mytnik and Salisbury(2015)]{mytnik2015uniqueness}
L.~Mytnik and T.S. Salisbury.
\newblock Uniqueness for {V}olterra-type stochastic integral equations.
\newblock \emph{arXiv preprint arXiv:1502.05513}, 2015.

\bibitem[Pag\`es(2024)]{Pages2024}
G.~Pag\`es.
\newblock Volterra equations with affine drift: looking for stationarity.
  {A}pplication to quadratic rough heston model.
\newblock 2024.

\bibitem[Protter(2005)]{Protter}
P.~E. Protter.
\newblock Stochastic integration and differential equations.
\newblock 2005.

\bibitem[Revuz and Yor(1999)]{RevuzYor}
D.~Revuz and M.~Yor.
\newblock \emph{Continuous martingales and {B}rownian motion}, volume 293 of
  \emph{Grundlehren der mathematischen Wissenschaften [Fundamental Principles
  of Mathematical Sciences]}.
\newblock Springer-Verlag, Berlin, third edition, 1999.

\bibitem[Richard et~al.(2021)Richard, Tan, and Yang]{RiTaYa2020}
A.~Richard, X.~Tan, and F.~Yang.
\newblock Discrete-time simulation of stochastic {V}olterra equations.
\newblock \emph{Stochastic Process. Appl.}, 141:\penalty0 109--138, 2021.

\bibitem[Rogers and Williams()]{RogersWilliamsII}
L.~C.~G. Rogers and D.~Williams.
\newblock \emph{Diffusions, {Markov} processes, and martingales. {Vol}. 2:
  {It{\^o}} calculus.}
\newblock 2nd ed. edition.

\bibitem[T.~Kailath and Zakai(1978)]{Kailath_Segall}
A.~Segall T.~Kailath and M.~Zakai.
\newblock Fubini-type theorems for stochastic integrals.
\newblock \emph{Sankhy{\=a}: The Indian Journal of Statistics}, pages
  138--143., 1978.

\bibitem[Veraar(2012)]{Veraar2012}
M.~Veraar.
\newblock The stochastic {F}ubini theorem revisited.
\newblock \emph{Stochastics}, 84\penalty0 (4):\penalty0 543--551, 2012.
\newblock ISSN 1744-2508, 1744-2516.

\bibitem[Walsh(1986)]{Walsh1986}
J.~B. Walsh.
\newblock An introduction to stochastic partial differential equations.
\newblock In P.~L. Hennequin, editor, \emph{{\'E}cole d'{\'E}t{\'e} de
  Probabilit{\'e}s de Saint Flour XIV--1984}, volume 1180 of \emph{Lecture
  Notes in Mathematics}, pages 265--439. Springer, Berlin, 1986.

\bibitem[Whitt(2007)]{Whitt2007}
W.~Whitt.
\newblock Proofs of the martingale {FCLT}.
\newblock \emph{Probability Surveys}, 4:\penalty0 268--302, 2007.

\bibitem[Zhang(2010)]{ZhangXi2010}
X.~Zhang.
\newblock Stochastic {V}olterra equations in {B}anach spaces and stochastic
  partial differential equation.
\newblock \emph{J. Funct. Anal.}, 258\penalty0 (4):\penalty0 1361--1425, 2010.

\end{thebibliography}
	
	\normalsize  
	\appendix
%	\section{Preliminaries: Analytical Tools and Properties of Hawkes Processes with Initial States}
	\section{Analytical tools and properties of Hawkes processes with initial states}\label{sect-PrelimStats}
	In this section, we introduce key tools including (functional) Fourier-Laplace transforms and results on resolvents of a borel function that will be important to our analysis. We next state key properties of Hawkes processes with general kernels and initial state condition.
	\subsection{Preliminaries: solvent core of a Borel function and Wiener-Hopf equations.}\label{subsect-tools} 
	Let us first introduce the Laplace transform of a Borel function $f:\R_+\to \R_+$ by
	\[\forall\, t\ge  0, \quad L_f(t)= \int_0^{+\infty} e^{-tu}f(u)du \!\in [0,\infty].\]
	This Laplace transform is non-increasing and if $L_f(t_0)<+\infty$ for some $t_0\ge 0$, then $L_f(t)\to 0$ as $t\to +\infty$.
	One can define the  Laplace transform of a Borel function $f:\R_+\to \R$ on $(0, +\infty)$ as soon as $L_{|f|}(t) <+\infty$ for every $t>0$  by the above formula. The Laplace transform can be extended to $\R_+$ as an $\R$-valued function  if $f \!\in {\cal L}^1_{\R_+}({\rm Leb}_1)$. 
	
	\medskip
	\noindent We consider the following definition of the resolvent of a Borel function, as introduced in \cite{Pages2024}. For every \(\lambda \in \mathbb{R}\), the {\em resolvent} \(R_{\lambda}\) associated with the function \(\varphi\), also known as the \(\lambda\)-resolvent of \(\varphi\), is defined as the unique solution (if it exists) to the deterministic Volterra equation:
	\begin{equation}\label{eq:Resolvent}
		\forall\,  t\ge 0,\quad R_{\lambda}(t) + \lambda \int_0^t \varphi(t-s)R_{\lambda}(s)ds = 1
	\end{equation}
	which can be written in terms of convolution: \(\; R_{\lambda}+\lambda \varphi*R_{\lambda} = 1.\;\)
	The solution  always satisfies $R_{\lambda}(0)=1$ and formally admits the following \textit{Neumann series expansion}:  
	\begin{equation}\label{eq:Resolvent3}
		R_{\lambda} = \sum_{k\ge 0} (-1)^k \lambda^k (\mbox{\bf 1}*\varphi^{k*}). 
	\end{equation}
	where, \(\varphi^{k*}\) denotes the \(k\)-th convolution of \(\varphi\) or the k-fold $*$
	product of \(\varphi\) with itself, with the convention in this formula, $\varphi^{0*}= \delta_0$ (Dirac mass at $0$). Recall that $ \varphi^{1*} = \varphi $ and $ \varphi^{k*}(t) = \int_0^t \varphi(t - s) \cdot \varphi^{(k-1)*}(s) \, ds.$
	
	\medskip
	\noindent {\bf Remark:} {\em Connexion with the canonical resolvent and the resolvent of the second kind.}
	We recall from  \cite{gripenberg1990} (see also \cite{abi2019affine}) that the one-dimensional functional resolvent of the second kind, of a function \( \varphi \in L^1_{\text{loc}}(\mathbb{R}^+) \), is the unique function \( \Psi_\varphi  \in L^1_{\text{loc}}(\mathbb{R}^+) \), solution to the linear Volterra
	equation \(\Psi_\varphi - \varphi = \varphi  \ast \Psi_\varphi\) while for some $\lambda \in \mathbb{R}$, the canonical resolvent of $\varphi$ with parameter $\lambda$ is the unique solution $E_\lambda \in L^1_{\text{loc}}(\mathbb{R}_+, \mathbb{R})$ of
	\(E_\lambda + \lambda \varphi \ast E_\lambda = \varphi\). This means that  \(E_0 = \varphi \quad \text{and} \quad E_{-1} = \Psi_\varphi.\)\\

	\noindent If $\varphi$ is regular enough (say continuous) the resolvent $R_{\lambda}$ is differentiable and one checks that $f_{\lambda}=-R'_{\lambda}$ satisfies, for every  $t>0$, \(-f_{\lambda}(t) +\lambda \big( R_{\lambda}(0)\varphi(t) - \varphi *f_{\lambda}(t)\big)=0\)
	that is $f_{\lambda}$ is solution to the equation
	\begin{equation}\label{eq:flambda-eq}
		f_{\lambda} +\lambda \varphi *f_{\lambda}=\lambda   \varphi.
	\end{equation}
	\noindent $\bullet$  We thus have: \(f_{\lambda} = \lambda E_{\lambda}\) so that $\Psi_\varphi = -f_{-1} =R^\prime_{-1}$.
	In particular, if $\lambda > 0$ and $R_{\lambda}$ turns out to be non-increasing, then $f_{\lambda}$ is non-negative and satisfies $0\le f_{\lambda} \le \lambda \varphi$. In that case one also has that $\int_0^{+\infty} f_{\lambda}(t)dt = 1 -R_{\lambda}(+\infty)$, so that $f_{\lambda} $ {\em is a probability density} if and only if $\displaystyle \lim_{t\to +\infty} R_{\lambda}(t) =0$.\\
	\noindent $\bullet$  Moreover, taking Laplace transforms of both sides of equation ~\eqref{eq:Resolvent} and then using that $L_{1}(t) =\frac1t$, yields:
	\begin{equation}\label{eq:Lapl_Resolvent}
		L_{R_\lambda}(t) = \frac{1}{t(1+\lambda L_\varphi(t))}.
	\end{equation}
	% \bigskip
	\begin{example}[\textit{$\lambda$-resolvent of a kernel and its Laplace Transform:} The case of the {\em Fractional integration kernel}]\label{Ex:frackernel}
		\begin{equation}\label{eq:frackernel}
			K(t) = K_{\alpha}(t) = \frac{u^{\alpha-1}}{\Gamma(\alpha)} \mbox{\bf 1}_{\R_+}(t),  \quad \alpha>0.
		\end{equation}
		This  family of  kernels corresponds to the fractional integrations of order $\alpha >0$. 
		It follows from the  easy identity  $K_{\alpha}*K_{\alpha'}= K_{\alpha+\alpha'}$,  that \(	R_{\alpha,\lambda}(t) = \sum_{k\ge 0} (-1)^k \frac{\lambda ^k t^{\alpha k}}{\Gamma(\alpha k+1)}= E_{\alpha}(-\lambda t^{\alpha} )= e_{\alpha}(\lambda^{1/\alpha}t) \; t\ge 0,\)
		where $E_{\alpha}$ denotes the standard   Mittag-Leffler function and  $e_{\alpha}$, the alternate Mittag-Leffler function.
		\begin{equation}\label{eq:MittagLeffler}
			E_{\alpha}(t) = \sum_{k\ge 0} \frac{t^k}{\Gamma(\alpha k+1)},\ t\!\in \R \quad \textit{and} \quad e_{\alpha}(t) := E_{\alpha}(-t^{\alpha}) = \sum_{k\ge 0} (-1)^k \frac {t^{\alpha k}}{\Gamma(\alpha k +1)}, \quad t\ge 0.
		\end{equation}
		Integral representations of the Mittag-Leffler function \( E_{\alpha} \) were first established in~\cite{Pollard1948}, followed by further results in~\cite{GorMain2000} and ~\cite{GorenfloMainardi1997}, where they were connected to the Laplace transform. For instance, in equation (F.12) on page 29 of~\cite{GorMain2000}, the Laplace transform of \( E_{\alpha}(-at^{\alpha}) \), with \( a \in \mathbb{C} \), is given by:
		\[
		L_{E_{\alpha}(-at^{\alpha})}(z) = \frac{z^{\alpha - 1}}{z^{\alpha} + a}, \quad z \in \mathbb{C}, \; \Re(z) > |a|^{1/\alpha}, \quad \alpha > 0.
		\]
		From this, we deduce the Laplace transform of \( e_{\alpha} \), which is given by \(L_{e_{\alpha}}(z) = \frac{z^{\alpha - 1}}{z^{\alpha} + 1}, \; z \in \mathbb{C}, \; \Re(z) > 1\) and \(\alpha > 0.\)
		Here, we define \( z^{\alpha} := |z|^{\alpha} e^{i\alpha \arg(z)} \), where \( -\pi < \arg(z) < \pi \), that is in the complex
		z-plane cut along the negative real axis.
		
		\medskip
		\noindent One shows (see~\cite{GorenfloMainardi1997} and Section 5 of \cite{Pages2024} further on) that $E_{\alpha}$ is increasing and differentiable on the real line with $\displaystyle \lim_{t\to +\infty}E_{\alpha} (t) =+\infty$ and $E_{\alpha}(0)=1$. In particular, $E_\alpha$ is an homeomorphism from $(-\infty, 0]$ to $(0,1]$. Hence, if $\lambda >0$, the function $f_{\alpha, \lambda}$ defined on $(0,+\infty)$ by
		\begin{equation}
			\label{eq:flambda} f_{\alpha, \lambda}(t)= - R'_{\alpha, \lambda}(t) = \alpha\lambda t^{\alpha-1} E'_{\alpha}(-\lambda t^{\alpha})  = \lambda t^{\alpha-1}\sum_{k\ge 0}(-1)^k\lambda^k \frac{t^{\alpha k}}{\Gamma(\alpha (k+1))}
		\end{equation}
		is a probability density -- called Mittag-Leffler density -- since $f_{\alpha, \lambda}>0$ and $\displaystyle \int_0^{+\infty} f_{\alpha, \lambda}(t)dt =  R_{\alpha,\lambda}(0) - R_{\alpha,\lambda}(+\infty) = 1$. Its Fourier-Laplace Transform is given by:
		\begin{align}
			L_{f_{\alpha, \lambda}}(z)=-L_{R'_{\alpha, \lambda}}(z) &= -L_{e'_{\alpha}(\lambda^{1/\alpha}\cdot)}(z) =-L_{e'_{\alpha}}(\frac{z}{\lambda^{1/\alpha}}) = -\left[\frac{z}{\lambda^{1/\alpha}}L_{e_{\alpha}}(\frac{z}{\lambda^{1/\alpha}}) - e_{\alpha}(0)\right] \notag \\
			&=\left[\frac{z}{\lambda^{1/\alpha}}\cdot \frac{(\frac{z}{\lambda^{1/\alpha}})^{\alpha - 1}}{(\frac{z}{\lambda^{1/\alpha}})^{\alpha} + 1} - 1 \right] = \frac{\lambda}{z^\alpha +\lambda}\label{eq:lapl_resolvent}
		\end{align}
		
	\end{example}
	\noindent Now, we outline the main results of these preliminaries, the first from \cite{EGnabeyeu2025}.
	
	\begin{proposition}[Wiener-Hopf and resolvent equations] \label{prop:W-H} Let $g, h: \R_+ \to \R$ be two  locally bounded Borel function, let $\varphi \! \in L^1_{loc}(Leb_{\R_+})$ and let $\lambda \!\in \R$.  Assume that the $\lambda$-resolvent $R_{\lambda}$ of $\varphi$ is differentiable on $(0, +\infty)$ with a derivative $R'_{\lambda}\!\in     L^1_{loc}(Leb_{\R_+})$, that both $R_{\lambda}$ and $R'_{\lambda}$ admit a finite Laplace transform on $\R_+$ and $\displaystyle \lim_{u\to +\infty} e^{-tu}R_{\lambda}(u) = 0$ for every $t>0$. Then,
		\begin{enumerate}
			\item[$(a)$]
			The Wiener-Hopf equation
			$\forall\, t\ge 0, \quad x(t) = g(t) -\lambda \int_0^t \varphi(t-s) x(s) ds$
			(also reading $x= g-\lambda \varphi*x$)  has   a solution given by:
			\begin{equation}\label{eq:Wiener-Hopf-solu}
				\forall\, t\ge 0, \quad x(t) = g(t) +\int_0^t R'_{\lambda}(t-s)g(s)ds\; \quad \text{that is,} \quad x= g- f_{\lambda}*g.
			\end{equation}	
			where $f_{\lambda}:= -R'_{\lambda}$. Or equivalently, \(x= g- \lambda E_{\lambda}*g,\) where $E_{\lambda}$ is the canonical resolvent of $\varphi$ with parameter $\lambda$. This solution is uniquely defined on $\R_+$ up to  $dt$-$a.e.$ equality.
			
			\item[$(b)$] The Wiener-Hopf or integral equation defined by
			\begin{equation}\label{eq:Wiener-Hopf22}
				\forall\, t\ge 0, \quad x(t) = h(t) - \int_0^t R'_{\lambda}(t-s) x(s) ds \quad \text{where} \quad f_{\lambda}= - R'_{\lambda}
			\end{equation}
			(also reading $x= h- R'_{\lambda}*x$)  has   a solution given by:
			\begin{equation}\label{eq:Wiener-Hopf-solu2}
				\forall\, t\ge 0, \quad x(t) = h(t) + \lambda \int_0^t \varphi(t-s)h(s)ds
			\end{equation}
			or, equivalently,  \(	x= h+ \lambda \varphi*h, \)
			This solution is uniquely defined on $\R_+$ up to  $dt$-$a.e.$ equality.
			
		\end{enumerate}
	\end{proposition}
	\noindent In Appendix~\ref{app:lemmata}, we provide a proof of this classical result for the reader's convenience.
	
		\begin{corollary}[Fredhom equations]\label{corl:fredhom}
		Let \( g \) be a measurable locally bounded function from \( \mathbb{R} \) to \( \mathbb{R} \) and \( \varphi : \mathbb{R}^+ \to \mathbb{R} \) be a real-valued function locally integrable. Then, there exists a unique locally bounded function \( f \) from \( \mathbb{R} \) to \( \mathbb{R} \) solution of \(	f(t) = g(t) + \int_0^t \varphi(t - s) \cdot f(s) \, ds, \quad t \geq 0,\)
		given by
		\[
		f(t) = g(t) + \int_0^t \psi_\varphi(t - s) \cdot g(s) \, ds, \quad t \geq 0, \quad \text{where} \quad \psi_\varphi = \sum_{k \geq 1} \varphi^{*k}.
		\]
	\end{corollary}
	\noindent {\bf Proof:} This is a Corollary of the proposition~\ref{prop:W-H}. It follows by appling the Wiener-Hopf equation of the first claim (a) in proposition~\ref{prop:W-H} with $\lambda=-1$. Alternatively see e.g. Lemma 3.  of \cite{BacryDelattreHoffmannMuzy2013}. \hfill $\Box$
	
	\medskip
	\noindent
	We now state the following proposition regarding the equivalence between a general representation of Volterra equations with affine drift, which is uniquely defined, and a form involving the $\lambda$-resolvent of the $\alpha$-fractional integration kernel. This result is an extension of [Proposition 4.9 \cite{ElEuchFukasawaRosenbaum2018}] which turns out to be the particular case of constant initial function \(\phi \equiv 1\) and \(\alpha \in (\frac12,1)\).
	
	\begin{proposition}[Equivalence Wiener-Hopf Transform for fractional Volterra integral equation]\label{prop:wiener_hopf}
		Let $\lambda > 0$ and let $\theta : \mathbb{R}_+ \to \mathbb{R}$ be a bounded Borel function (hence having   a well-defined finite Laplace transform on $(0,+\infty)$), $\sigma: \R_+\times \R \to  \R$  H\"older continuous in $x$, locally uniformly in $t\!\in [0,T]$, for every $T>0$, $ \phi \!\in {\cal L}^1_{\R_+}({\rm Leb}_1)$ a continuous function and \( (W_t)_{t \geq 0} \) a standard Brownian motion independent from the \( \mathbb{R} \)-valued random variable \( X_0 \), both defined on a probability space \( (\Omega, \mathcal{A}, \mathbb{P}) \). Assume
		$L^2_{\text{loc}, \mathbb{R}^+} (\text{Leb}_1) \ni K_{\alpha}:\R_+\to \R_+$ is the convolutive \(\alpha-\) fractional integration kernel for \(\alpha \in (\frac12,\frac32)\)  and its $\lambda$-resolvent $R_{\alpha,\lambda}$ is well-defined and differentiable on $(0, +\infty)$ with a derivative $R'_{\lambda}:=-f_{\alpha,\lambda}\!\in     L^1_{loc}(Leb_{\R_+})$, that both $R_{\alpha,\lambda}$ and $R'_{\alpha,\lambda}$ admit a finite Laplace transform on $\R_+$  Then, the process $(X_t)_{t \geq 0}$ is the solution of the following  stochastic volterra integral equation:
		\begin{equation}\label{eq:Volterrameanrevert2}
			X_t= X_0\big(\phi(t)- \int_0^t f_{\alpha, \lambda}(t-s)\phi(s)ds\big) +\frac{1}{\lambda}\int_0^t f_{\alpha,\lambda}(t-s)\theta(s)ds + \frac{1}{\lambda}\int_0^t f_{\alpha,\lambda}(t-s)\sigma(s,X_s)dW_s.
		\end{equation}
		if and only if it is the solution of
		\begin{equation}\label{eq:Volterrameanrevert}
			X_t = X_0\phi(t) +\int_0^t K_\alpha(t-s)(\theta(s)-\lambda X_s)ds + \int_0^t K_\alpha(t-s)\sigma(s,X_s)dW_s, \quad X_0\perp\!\!\!\perp W,
		\end{equation}
	\end{proposition}
	
	\noindent {\bf Proof:}	
	\smallskip This turns out to be a straightforward consequence of \cite[Proposition 3.2]{EGnabeyeu2025}, applied to the particular case of the \(\alpha-\) fractional integration kernel for \(\alpha \in (\frac12,\frac32)\).
	
	\subsection{A summary about Hawkes processes with initial state condition}
	\label{subsec:PrelimStationary}
	In this section, we state key properties of Hawkes processes with general kernels, time-varying baseline and initial state condition, including martingality and stationarity, using the tools introduced in Section~\ref{subsect-tools}.
	For clarity and conciseness, the proofs are omitted and postponed to Appendix \ref{app:lemmata}.\\
	The process $\Lambda:=\{\Lambda_t:t\geq 0\}$  is termed the \( \mathcal{F}_t \)-intensity of \( N \) if, for any interval \( (s, t] \) and $A\in \mathcal{F}_s$, the following holds almost surely: 
	\begin{equation}\label{eq:intensity}
		\E[N((s, t]) \mid \mathcal{F}_s] = \E\left[ \int_s^t \Lambda_u \, du \mid \mathcal{F}_s \right] \textit{or equivalently}\quad \E[(N_t-N_{s})\mathrm{1}_A]=\E[\int_s^t\Lambda_u\mathrm{1}_A du]. 
	\end{equation}
	The law of such a process is characterized by that predictable intensity, informally understood as the instantaneous probability of an event's arrival time. Specifically, this is expressed as:
	\[
	\P\left(N_t \text{ jumps in } [t, t+dt] \mid \mathcal{F}_t \right) = \Lambda_{t} \, dt,
	\]
	We denote by \( \mathcal{I}^\Lambda_\cdot = \int_0^\cdot \Lambda_s \, ds \) the integrated intensity.
	The Hawkes process $N$ has the compensator $\mathcal{I}^\Lambda$ which is its predictable quadratic variation and is given by \( \mathcal{I}^\Lambda_t = \int_0^t \Lambda_s \, ds \),  so that the compensated point process \footnote{ For a comprehensive introduction to point processes with stochastic intensities, the reader is referred to \cite{DaleyVereJones2002}.}
	\begin{equation}\label{eq:compensator} 
		M:=N-\mathcal{I}^\Lambda 
	\end{equation}
	is a $(\mathcal{F}_t)$-martingale by the very definition in equation \ref{HawkesDensityLinear} of the intensity \( \Lambda \), for any simple point process \( N \). The predictable process \(\mathcal{I}^\Lambda_t\) is also called the increasing process or the bracket of \(N\), and is denoted by \(\langle N \rangle_t\). For any predictable process \((H_t)_{t \geq 0}\) such that \(\mathbb{E} \left[ \int_0^t \left| H_s \right| N(ds) \right] < \infty\), it holds that \(\mathbb{E} \left[ \int_0^t H_s\, N(ds) \mid \mathcal{F}_t \right] = \mathbb{E} \left[ \int_0^t H_s\, \Lambda_s\, ds \mid \mathcal{F}_t \right]\). Taking \(H_s = \mathbf{1}_{[s, t]}\) for some interval \([s, t]\) yields the relation in equation~\eqref{eq:intensity}.
	
	\medskip
	\noindent
	We provide below a martingale representation theorem for the intensity process with initial state condition, in terms of the compensated point process with the resolvent defined in equation ~\eqref{eq:Resolvent}.
	\begin{theorem}[Martingale representation]\label{MartRep}
		The intensity process $\Lambda$ admits the representation
		\begin{equation}\label{SVR}
			\Lambda_t = Z_0 (t) + \mu(t) + (R^\prime_{-1}*(Z_0 + \mu ))_t + (R^\prime_{-1}\stackrel{M}{*}\mathbf{1})_t  ,\quad t\geq 0.
		\end{equation} 
		or equivalently,  \(\Lambda_t = Z_0(t) + \mu(t) - \int_0^t f_{-1}(t-s)(Z_0(s) + \mu(s) )ds - \int_0^t f_{-1}(t-s)dM_s ,\quad t\geq 0, \) where \( f_{\lambda} := - R'_{\lambda} \quad \forall \lambda \in \R\).
		Moreover if \(\; R'_{-1},\; \mathcal{I}^\varphi \in {\cal L}_{\text{loc}}^2(\mathbb{R}_+, \text{Leb}_1) \), the expected intensity (assume to be bounded) and the expected number of events are given respectively by
		\begin{equation} \label{eq:mart}
			\begin{split}
				\mathbb{E}[\Lambda_t] &= \mathbb{E}[Z_0 (t)] + \mu(t) + (R^\prime_{-1}*(\mathbb{E}[Z_0(\cdot)] + \mu(\cdot)))_t,  \implies \Lambda_t = \mathbb{E}[\Lambda_t] + \int_0^t R^{\prime }_{-1}(t-s)dM_s , %\quad \mathbb{E}[\Lambda_t^2] =\mathbb{E}[\Lambda_t]^2 +  \int_{0}^t  R_{-1}^{'2} (t-s) \mathbb{E}[\Lambda_s] ds
				\\
				\mathbb{E}[N_t] &= \mathbb{E}[\mathcal{I}^\Lambda_t] = \int_0^{t}(\mathbb{E}[Z_0(s)]+ \mu(s)) ds 
				+ ((\mathbb{E}[Z_0(\cdot)]+ \mu(\cdot)) *\mathcal{I}^{R^\prime_{-1}})_t
			\end{split}
		\end{equation}
		where \(	\mathcal{I}^{R^\prime_{-1}}_t := \int_0^t R^\prime_{-1}(s) ds = R_{-1}(t) - R_{-1}(0).\)
		In particular, if \(Z_0 \equiv 0\), the expected intensity, the expected squared intensity and the expected number of events are given respectively by, 
		\begin{equation}\label{eq:mart2}
			\mathbb{E}[\Lambda_t]= \mu(t) + (R'_{-1}*\mu)_t,\; \mathbb{E}[\Lambda_t^2] =\mathbb{E}[\Lambda_t]^2 +  (R^{\prime 2}_{-1} * \mathbb{E}[\Lambda_\cdot] )_t, \; 
			\mathbb{E}[N_t]= \mathbb{E}[\mathcal{I}^\Lambda_t]= \int_0^t \mu(s) ds + (\mu *\mathcal{I}^{R'_{-1}})_t
		\end{equation}	
	\end{theorem}
	\noindent
	We now investigate stationarity \footnote{Weak-stationarity in the sense of constant mean and stable autocovariance function in contrast to strong-stationarity where all finite-dimensional distributions are invariant under time shifts.} of Hawkes processes.
	For a simple point process \( N \), stationarity means that its distribution remains unchanged under time shifts. More formally, \( N \) is stationary if for any time shift \( t \), the process \( \theta_t N \) has the same distribution as the original process \( N \), where \( \theta_t \) is the shift operator defined by:
	\[\theta_t N(o) = N(t + o)  \quad \textit{for any} \quad o \in \mathcal{B}(\mathbb{R}).
	\]
	This condition implies that a stationary Hawkes process \( N \) exhibits stationary increments, and its intensity process $\Lambda:=\{\Lambda_t:t\geq 0\}$  is itself stationary, meaning that the distribution of \( \Lambda_t \) does not depend on \( t \), almost surely.
	A necessary condition for the process \( (\Lambda_t)_{t \in \mathbb{R}} \) to be in a stationary state is that its expected value (or verage intensity) \( \bar\Lambda := \mathbb{E}[\Lambda_t] \) is constant, positive, and finite. Therefore, before tackling the problem of the stationary regime of the Hawkes process whose intensity is given by equation~\eqref{HawkesDensityLinear2},
	a first step can be to determine when this equation has a constant mean (and even constant variance) i.e.
	$\mathbb{E}[\Lambda_t] = \mathbb{E}[\Lambda_0]$ for every $t \geq 0$. 
	\begin{proposition}[Stationarity of the first two moments]\label{prop:stat}
		Let \((N,\Lambda) \) be a Hawkes process.
		
		\begin{enumerate}
			\item If \(Z_0 \equiv 0\) (i.e., the process starts at time zero with no prior events), then the Hawkes process \((N,\Lambda) \) has constant first moment, if and only if
			
			\(\qquad \qquad \forall\, t\ge 0, \qquad \mu(t) = \mu(0)\left(1-\int_0^t \varphi(s)ds\right)\quad  dt-a.e. \quad \textit{and} \quad \mathbb{E}[\Lambda_t] = \mathbb{E}[\Lambda_0]=\mu(0).\)
			
			Moreover, if the second moment is assume to be constant, then the Hawkes process \((N,\Lambda) \) reduces to a homogeneous or stationary Poisson with intensity \(\mu (0)\)
			
			\item If \(Z_0 \not\equiv 0\) Then, the necessary and sufficient condition for the Hawkes process to have constant first moments is:
			\begin{itemize}
				\item \textit{Either} $\mu(t) = C^{\textit{ste}} = \mu > 0 $ \textit{and} $\int_0^{\infty} \varphi(s) \, ds < 0$, in which case $\mathbb{E}[\Lambda_t] = \bar\Lambda = \frac{\mu_0}{1 - \int_0^{\infty} \varphi(s) \, ds}$.\\
				And thus $$\Lambda_t \underset{t\to+\infty}{\to} \bar\Lambda + \int_0^t R^{\prime }_{-1}(t-s)dM_s $$
				\item \textit{Or} $\int_0^{\infty} \varphi(s) \, ds = 1$, and in this case $\bar\Lambda = +\infty$ (degenerate or nearly unstable case).
			\end{itemize}
		\end{enumerate}
		Therefore, when \( \mu > 0 \), the stability condition \( a = \mathit{\Phi}(0)= \int_0^{\infty} \varphi(s) \, ds < 1 \) is also necessary if we want Hawkes process \((N,\Lambda) \) to have a finite average intensity.
	\end{proposition}
	\noindent Note that this result is instrumental in establishing the convergence to an inhomogeneous fractional CIR process with a constant mean-reverting coefficient, as emphasized in the Remark devoted to~\eqref{eq:frac_CIR1}--~\eqref{eq:frac_CIR2}.
	
	\section{A note on the choice of the inhomogeneous intensity \(\mu^T(t)\)}
	\label{app:choose_mu}
	Following \cite{JaissonRosenbaum2015,JaissonRosenbaum2016,ElEuchFukasawaRosenbaum2018},  the suitably rescaled version of the intensity process asymptotically behaves as a diffusion process with constant mean-reversion parameter and with initial value equal to zero. As in \cite{ElEuchR2018, el2019characteristic}, we would like to obtain a particular diffusion with a time-dependent mean-reversion level and a non-zero starting value in the limit: That is more precesily, a diffusion of the form:
	\begin{equation}\label{eqDym}
		X_t = X_0\phi(t) +\int_0^t K(t-s)\lambda(\theta(s)- X_s)ds +  \int_0^t K(t-s)\varsigma(s) \sqrt{X_s}dW_s, \quad X_0\perp\!\!\!\perp W,
	\end{equation}
	where $K$ is the fractional integration kernel, $ \phi \!\in {\cal L}^1_{\R_+}({\rm Leb}_1)$ a continuous function and \( (W_t)_{t \geq 0} \) a standard Brownian motion independent from the \( \mathbb{R} \)-valued initial random variable \( X_0 \), both defined on a probability space \( (\Omega, \mathcal{A}, \mathbb{P}) \).
	Naturally, this expression should correspond to the weak limit of the intensity process given by equation~\eqref{SVR2}, which we recall below for clarity.
	\begin{equation}\label{SVR2_}
		\Lambda^T_t = \Lambda_0^{*T} \left( \mathit{\Phi}^T(t) + (R^{\prime T}_{-1} * \mathit{\Phi}^T)_t \right) + \mu^T(t) + \int_0^t R^{\prime T}_{-1}(t-s)\mu^T(s) \, ds + \int_0^t R^{\prime T}_{-1}(t-s) \, dM^T_s 
	\end{equation}
	From Proposition \ref{prop:wiener_hopf} in Section \ref{subsect-tools} , the dynamic in equation~\eqref{eqDym} is equivalent to:
	\begin{equation}
		X_t= X_0\big(\phi(t)- \int_0^t f_{\lambda}(t-s)\phi(s)ds\big) + \int_0^t f_{\alpha, \lambda}(t-s)\theta(s)ds + \frac{1}{\lambda}\int_0^t f_{\alpha, \lambda}(t-s)\varsigma(s) \sqrt{X_s}dW_s.
	\end{equation}
	Using the same heuristic arguments as in Section \ref{subsect:intuition}, and reasoning backward with \(\frac{\mu^T(tT)}{\mu^T} \equiv \frac{\theta(t)}{\theta_0}\), we observe that we should obtain this dynamic in the limit, provided that we work with a process \(\Lambda^{*T}_t\) with the following expression:
	{\small
	\begin{align*}
		\Lambda^{*T}_t &= \Lambda_0^{*T} \, (\mathit{\Phi}^T(tT) + (R^{\prime T}_{-1}*\mathit{\Phi}^T)(tT))\frac{1 - a_T}{\tilde\mu^T(\frac{t}{T})}  + \frac{\mu^T}{\tilde\mu^T(\frac{t}{T})}\int_0^{t} T(1 - a_T) R^{\prime T}_{-1}(T(t - s))\frac{\theta(s)}{\theta_0} \, ds \\
		&\quad  \hspace{1.5cm}+ (1 - a_T)\frac{\mu^T}{\tilde\mu^T(\frac{t}{T})} + \frac{1 }{\tilde\mu^T(\frac{t}{T})}\int_0^{t} \sqrt{\frac{\tilde\mu^T(\frac{s}{T})}{T(1 - a_T)}} T(1 - a_T) R^{\prime T}_{-1}(T(t - s)) \, \sqrt{\Lambda^{*T}_s} \, dW^T_{s}
	\end{align*}
	}
	This is equivalent to (inverse normalisation):
	\begin{equation}\label{SVR4} 
		\Lambda^T_t = \Lambda_0^{*T} \left( \mathit{\Phi}^T(t) + (R^{\prime T}_{-1} * \mathit{\Phi}^T)_t \right) + \mu^T + \frac{\mu^T }{\theta_0} \int_0^t R^{\prime T}_{-1}(t-s) \theta(\frac{s}{T})ds + \int_0^t R^{\prime T}_{-1}(t-s)dM^T_s ,\quad t\geq 0.
	\end{equation} 
	Therefore, identifying parameters in \eqref{SVR2_} and \eqref{SVR4} , this indicates that we should take $\mu^T(\cdot)$ such that
	\begin{equation}
		\label{eqaresoudre}  
		\mu^T(t) + \int_0^t R^{\prime T}_{-1}(t-s) \mu^T(s) ds  = \mu^T + \mu^T \int_0^t R^{\prime T}_{-1}(t-s)\bar\theta(\frac{s}{T})ds .
	\end{equation}
	where we set \(\bar\theta(\frac{\cdot}{T})= \frac{\theta(\frac{\cdot}{T}) }{\theta_0}.\)
	We may read the above equation~\eqref{eq:Volterrameanrevert}  ``pathwise'' as a Wiener-Hopf equation with $x(t)= X_t(\omega)$ and \(h(t) =  \mu^T + \mu^T \int_0^t R^{\prime T}_{-1}(t-s)\bar\theta(\frac{s}{T})ds.\)
	Then, applying Proposition~\ref{prop:W-H} (b), we get: \( \mu^T(t) = h(t) - (\varphi^T*h)_t \).
	From the second term in the right-hand side of the above equation, we get: 
	\begin{align*}
		\frac{1}{\mu^T}(\varphi^T*h)_t &=\big(\varphi^T* 1 \big)_t
		+ \int_0^t \varphi^T(t-s) \int_0^s  R^{\prime T}_{-1}(s-u)\bar\theta(\frac{u}{T}) du ds  \\
		&= \int_0^t \varphi^T(t-s) ds +  \int_0^t \int_0^{t-u}  R^{\prime T}_{-1}(s)\varphi^T(t-u-s)    ds \bar\theta(\frac{u}{T})du \\
		&= \int_0^t \varphi^T(t-s)  ds +  \int_0^t  \big(R^{\prime T}_{-1}(t-u) - \varphi^T(t-u)\big)\bar\theta(\frac{u}{T})  du \\
		&=  \big(R^{\prime T}_{-1}*\bar\theta(\frac{\cdot}{T})\big)_t + \int_0^t \varphi^T(t-u)\big(1-\bar\theta(\frac{u}{T})\big)  du .
	\end{align*}
	Consequently, after simplification, the following equality should hold for a well-chosen $\mu^T(t)$:
	
	\begin{equation}
		\label{eq:mu}
		\frac{\mu^T(t)}{\mu^T} = 1 -\int_0^t \varphi^T(t-u)\big(1-\bar\theta(\frac{u}{T})\big)  du = 1 -\int_0^t \varphi^T(t-u)\big(1-\frac{\theta(\frac{u}{T}) }{\theta_0}\big)  du .
	\end{equation}
	Note that, this is different from the equation provided in \cite{ElEuchR2018} since the initial condition or initial random variable does not comes from the method described above.
	
	\section{Supplementary materials and proofs. }\label{app:lemmata}
	
	\subsection{Proof of Proposition~\ref{prop:W-H}, Theorem~\ref{MartRep} and Proposition~\ref{prop:stat}}
		
	\noindent {\bf Proof of Proposition~\ref{prop:W-H}:}
	
	\smallskip
	\noindent  {\sc Step~1.}{\em Proof of the first claim (a):}
	First  note that, by an integration by parts, for every $t>0$, \(	tL_{R_{\lambda}} (t)= 1 +L_{R'_{\lambda}}(t).\)
	On the other hand it follows from~\eqref{eq:Resolvent}  that $tL_{R_\lambda}(t)(1+\lambda L_\varphi(t))=1 $, $t>0$. 
	Consequently, 
	\[L_x (t)= \frac{L_g(t)}{1+\lambda L_{\varphi}(t)}= tL_g L_{R_{\lambda}}(t)= L_g(t)  \big(1 +L_{R'_{\lambda}}(t)\big)= L_{g+ R'_{\lambda}*g}(t)\]
	which completes the proof since Laplace transform is injective. \hfill $\Box$
	
	\medskip
	\noindent  {\sc Step~2.}{\em Proof of the second claim (b):}
	Convoluting $x(t) + \int_0^t R'_{\lambda}(t-s) x(s) ds$ by $\varphi$ together with the fact that $R'_{\lambda}*\varphi = - \varphi - \frac{R'_{\lambda}}{\lambda}$(see equation ~\ref{eq:flambda-eq} ), we obtain:
	\begin{align*}
		\int_0^t h(s) \varphi(t-s) &\; ds = \int_0^t x(s) \varphi(t-s) ds + \int_0^t \int_0^s R'_{\lambda}(s-u) x(u) du \varphi(t-s) ds\\ 
		&= \int_0^t x(s) \varphi(t-s) ds + \int_0^t \int_0^{t-u} R'_{\lambda}(t-u-s) \varphi(s)  ds x(u) du \\
		&= \int_0^t x(s) \varphi(t-s) ds + \int_0^t \big(- \frac{R'_{\lambda}}{\lambda}(t-u)-\varphi(t-u)\big) x(u) du  
		= - \frac{1}{\lambda}\int_0^t R'_{\lambda}(t-s) x(s) ds.
	\end{align*}
	Replacing back in equation ~\eqref{eq:Wiener-Hopf22} gives the results.
	
	\medskip
	\noindent {\bf Remark:}
	We could have rather used Laplace transform in equation ~\eqref{eq:Wiener-Hopf22} and deduce that:
	
	\(\qquad \qquad \qquad \qquad L_x (t)= \frac{L_h(t)}{1+ L_{R'_{\lambda}}(t)}= L_h(t)  \big(1 + \lambda L_{\varphi}(t)\big)= L_{h+ \lambda\varphi*h}(t).\)
	
	\noindent where the penultimalte equality come from applying laplace transform to $R'_{\lambda}*\varphi = - \varphi - \frac{R'_{\lambda}}{\lambda}$(see equation ~\ref{eq:flambda-eq} ). We next conclude by the injectivity of Laplace transform.\hfill $\Box$
	
	\medskip
	\noindent {\bf Proof of Theorem~\ref{MartRep}:}
	Combining equations ~\eqref{HawkesDensityLinear2} and ~\eqref{eq:compensator} , we write:
	\begin{equation}\label{eq:int}
		\Lambda_t = Z_0(t) + \mu(t) + \int_0^t \varphi(t-s)\Lambda_sds + \int_0^t \varphi(t-s)dM_s ,\quad t\geq 0,
	\end{equation}
	Integrating both sides of (\ref{eq:int}) and then using that $
	\int_0^{\cdot} (f * g)(s) \, ds = \left( \int_0^{\cdot} f(s) \, ds \right) * g
	$, together with stochastic Fubini's theorem,\footnote{Sufficient conditions for interchanging ordinary integration (with respect to a finite measure) and stochastic integration (with respect to a square-integrable martingale) are given in \cite[Thm.~1]{Kailath_Segall}; see also \cite[Thm.~IV.65]{Protter} for further details.}
	the cumulated or integrated intensity can be represented as: 
	For all $t\geq 0$, we have
	\begin{equation}\label{eq:cum_int}
		\int_0^{t} \Lambda_sds= \int_0^{t}(Z_0(s)+ \mu(s)) ds+\int_0^t \varphi(t-s)(\int_0^s \Lambda_u du)ds+\int_0^t (\int_0^{t-s}\varphi(u)du) dM_s.
	\end{equation}
	The first two terms on the right side of equation (\ref{eq:cum_int}) specify the average direct and indirect impact of external events on the arrivals of future events. 
	The third term specifies the random variations in the number of child events on the arrivals of future events. 
	It follows by Fubini's theorem that,
	\begin{equation}\label{eq:ex_cum_int}
		\mathbb{E}[\mathcal{I}^\Lambda_t]= \mathbb{E}[\int_0^{t} \Lambda_sds ]= \int_0^{t}(\mathbb{E}[Z_0(s)]+ \mu(s)) ds+\int_0^t \varphi(t-s)\mathbb{E}[\int_0^s \Lambda_u du]ds.
	\end{equation}
	Back to equation ~\eqref{eq:int} above, it can be interpreted pathwise as a Wiener-Hopf equation with $x(t) = \Lambda_t(\omega)$, \(\lambda=-1\) and \(g(t) = Z_0(t) + \mu(t) + \int_0^t \varphi(t-s)dM_s = Z_0(t) + \mu(t) + (\varphi \stackrel{M}{*} \mathbf{1})_t.\)
	This leads by applying Proposition \ref{prop:W-H} (a) to the following expression for $\Lambda_t$:
	\begin{align*}
		\Lambda_t &= g(t) + \int_0^t R'_{-1}(t-s) g(s) \, ds \\&= Z_0(t) + \mu(t) + \int_0^t \varphi(t-s)dM_s + \int_0^t R'_{-1}(t-s) \left[Z_0(s) +\mu(s) +\int_0^s \varphi(s-r)dM_r \right] \, ds.
	\end{align*}
	Using commutativity and associativity (via stochastic Fubini's theorem) of convolution together with the fact that $R'_{-1} \ast \varphi =R'_{-1}-\varphi$ (equation ~\eqref{eq:flambda-eq} of the Remark in section~\ref{subsect-tools} ), we obtain:
	\begin{align*}
		&\ \int_0^t R'_{-1}(t-s) \int_0^s \varphi(s-r) dM_r ds= \int_0^t \int_0^t \mathrm{1}_{r\leq s}  R'_{-1}(t-s) \varphi(s-r)dsdM_r\\
		&= \int_{0}^t \int_0^{t-r} R'_{-1}(t-r-s)  \varphi(s)ds dM_r = \int_{0}^t  (R'_{-1} \ast \varphi)(t-r) dM_r= \int_{0}^t  R'_{-1} (t-r) dM_r-\int_{0}^t  \varphi (t-r) dM_r
	\end{align*}
	\[\text{We then have: }\;\qquad\qquad \Lambda_t = Z_0(t) + \mu(t) + \int_0^t R'_{-1}(t-s) (Z_0(s) + \mu(s)) \, ds + \int_{0}^t  R'_{-1} (t-s) dM_s.\qquad\qquad\]
	Note that, in the particular setting \(Z_0 \equiv 0\), the associated martingale representation for the intensity process has been established by \cite{BacryDelattreHoffmannMuzy2013} using the variation of constants formula for linear Volterra integral equations as given in \cite[p.36, Equation (1.2)]{gripenberg1990}
	
	\noindent As, for every $t \in [0, T ],$ $T > 0$,
	$\mathbb{E}\left[\left(\int_{0}^t  R'_{-1} (t-s) dM_s\right)^2\right] = \int_{0}^t  R_{-1}^{'2} (t-s) \mathbb{E}[\Lambda_s] ds < C_T \int_{0}^t  R_{-1}^{'2}(s)ds $
	we have, for every $t \geq 0$,
	$\mathbb{E}\left[\int_{0}^t  R'_{-1} (t-s) dM_s\right]=0$
	so that:
	$$\mathbb{E}[\Lambda_t] =  \mathbb{E}[Z_0(t)] + \mu(t) + \int_0^t R'_{-1}(t-s) (\mathbb{E}[Z_0(s)] + \mu(s))  \, ds. $$ 
	Thus, the martingale representation theorem provides an exact representation of the mean (and even the second moment) of the intensity of the Hawkes
	process \((N,\Lambda)\) in term of the resolvent as stated above.
	Alternatively, we could have noticed that, as $\Lambda_t = Z_0(t) + \mu(t) + \int_{(0,t)} \varphi(t-s)N(ds), t>0 $ from \ref{HawkesDensityLinear2}, taking expectatiuon on both side of that inequality lead to:
	$$\mathbb{E}[\Lambda_t] =  \mathbb{E}[Z_0(t)] + \mu(t) + \int_0^t \varphi(t-s)\mathbb{E}[\Lambda_s] ds $$
	This still can be interpreted as a Wiener-Hopf equation with $x(t) = \mathbb{E}[\Lambda_t]$, \(\lambda=-1\) and $
	g(t) = \mathbb{E}[Z_0(t)] + \mu(t)$.
	By applying Proposition~\ref{prop:W-H} (a) , this lead to the following expression for $\mathbb{E}[\Lambda_t]$:
	$$\mathbb{E}[\Lambda_t] =  \mathbb{E}[Z_0(t)] + \mu(t) + \int_0^t R'_{-1}(t-s) (\mathbb{E}[Z_0(s)] + \mu(s))  \, ds$$
	so that in the particular settings where \(Z_0 \equiv 0\), we have: \(\mathbb{E}[\Lambda_t] = \mu(t) + \int_0^t R'_{-1}(t-s) \mu(s)  \, ds\) i.e. \(\Lambda_t = \mathbb{E}[\Lambda_t] + \int_{0}^t  R'_{-1} (t-s) dM_s\)
	and for the second moment,
	\begin{align*}
		\mathbb{E}[\Lambda_t^2] &= \mathbb{E}[\Lambda_t]^2 + \mathbb{E}\left[\left(\int_{0}^t  R'_{-1} (t-s) dM_s\right)^2\right] +2 \mathbb{E}\left[\mathbb{E}[\Lambda_t]\int_{0}^t  R'_{-1} (t-s) dM_s\right]\\
		&=\mathbb{E}[\Lambda_t]^2 +  \int_{0}^t  R_{-1}^{'2} (t-s) \mathbb{E}[\Lambda_s] ds
	\end{align*}
	
	\noindent Let us now rewrite expectation of the cumulated intensity ~\eqref{eq:ex_cum_int} (Note that, we can also derive it directly from the equation~\eqref{SVR}). 
	For all $t\geq 0$, we have
	\begin{equation}\label{eq:cum_intensity}
		\mathbb{E}[\mathcal{I}^\Lambda_t]= \int_0^{t}(\mathbb{E}[Z_0(s)]+ \mu(s)) ds+\int_0^t \varphi(t-s)\mathbb{E}[\mathcal{I}^\Lambda_s]ds.
	\end{equation}
	This together with the Corollary~\ref{corl:fredhom} and regular Fubini's theorem in the second equality yields:
	\begin{align*}
		\mathbb{E}[\mathcal{I}^\Lambda_t] &=\int_0^{t}(\mathbb{E}[Z_0(s)]+ \mu(s)) ds + \int_0^t R'_{-1}(t-s)\int_0^{s}(\mathbb{E}[Z_0(u)]+ \mu(u)) duds.\\
		&=\int_0^{t}(\mathbb{E}[Z_0(s)]+ \mu(s)) ds + \int_0^t (\mathbb{E}[Z_0(t-s)]+ \mu(t-s)) (\int_0^{s}R'_{-1}(u)du)ds.
	\end{align*} 
	This complete the proof and we are done.\hfill $\Box$
	
	\medskip     
	\noindent {\bf Proof  of Proposition~\ref{prop:stat}.}
		
		\smallskip
		\noindent  {\sc Step~1} {\em (The case \(Z_0 \equiv 0\) ):}
		Using Equation~\eqref{eq:mart2} of Theorem \ref{MartRep}, we derive that, for every $t \geq 0$,
		$$\mathbb{E}[\Lambda_t]= \mu(t) + (R^\prime_{-1}*\mu)_t \quad \text{and} \quad {\rm Var}(\Lambda_t)=\mathbb{E}[\Lambda_t^2] -\mathbb{E}[\Lambda_t]^2 =  \int_{0}^t  R_{-1}^{'2} (t-s) \mathbb{E}[\Lambda_s] ds.$$
		Hence $\mathbb{E}[\Lambda_t]$ is constant if and only if: $\forall t\geq 0 \quad \mu(t) + (R^\prime_{-1}*\mu)_t = \mu(0)$.
		And then, applying Proposition~ \ref{prop:W-H} (b), we deduce that:
		
		\(\qquad \qquad \forall\, t\ge 0, \qquad \mu(t) = \mu(0)\left(1-\int_0^t \varphi(s)ds\right)\quad  dt-a.e. \quad \textit{and} \quad \mathbb{E}[\Lambda_t] = \mathbb{E}[\Lambda_0]=\mu(0).\)
		
		\noindent Conversely one checks that for this value of $\mu(t)$, the expectation of $\Lambda_t$ is constant.
		Assuming constant mean, the above variance becomes: \(\forall\, t\ge 0, \quad {\rm Var}(\Lambda_t) = \mathbb{E}[\Lambda_0] \int_{0}^t  R_{-1}^{'2} (t-s)  ds = \mathbb{E}[\Lambda_0] \int_{0}^t  R_{-1}^{'2} (s)  ds.\)
		Thus ${\rm Var}(\Lambda_t)$ constant implies $\varphi \equiv 0$, (Take the derivative and owe to the injectivity of Laplace transform).
		It boils down that \((N,\Lambda) \) is a homogenous poisson with intensity \(\mu (0)\)
		
		\medskip
		\noindent  {\sc Step~2} {\em (The case \(Z_0 \not\equiv 0\) ):}
		Set \(\mu(0) = \mu \in \R_+ \), thus $\Lambda_0 = Z_0(0) + \mu \in L^1(\R_+ ) $ so that:\\ \begin{equation}\label{Expectation0}
			\mathbb{E}[\Lambda_0] = \mathbb{E} [ Z_0(0)] + \mu =
			\mathbb{E} \left[ \int_{-\infty}^0 \varphi(-s) \, dN_s \right] + \mu = \int_{-\infty}^0 \varphi(-s) \, \mathbb{E} [ \Lambda_s]ds + \mu
		\end{equation} 
		Using equation~\eqref{eq:mart2} of Theorem \ref{MartRep}, we derive that, for every $t \geq 0$,
		$$\mathbb{E}[\Lambda_t]= \mathbb{E}[Z_0 (t)] + \mu(t) + (R^\prime_{-1}*(\mathbb{E}[Z_0(\cdot)] + \mu(\cdot)))_t$$
		Equation~\eqref{Expectation0} lead to $\mathbb{E}[\Lambda_0] = \frac{\mu}{1-\int_0^{+\infty} \varphi(s)ds}.$
		Hence $\mathbb{E}[\Lambda_t]$ is constant if and only if: $$\forall t\geq 0 \quad \mathbb{E}[Z_0 (t)] + \mu(t) + (R^\prime_{-1}*(\mathbb{E}[Z_0(\cdot)] + \mu(\cdot)))_t = \mathbb{E}[\Lambda_0] = \mathbb{E} [ Z_0(0)] + \mu$$
		And then, applying Proposition ~\ref{prop:W-H} (b), we deduce that:
		$$\forall\, t\ge 0, \qquad \mu(t) = \mathbb{E}[\Lambda_0] \left(1-\int_0^t \varphi(s)ds\right) - \mathbb{E}[Z_0 (t)] \quad  dt-a.e. \quad \textit{and} \quad \mathbb{E}[\Lambda_t] = \mathbb{E}[\Lambda_0]= \frac{\mu}{1-\int_0^{+\infty} \varphi(s)ds}$$
		Moreover for \(t\geq0\), equation~\eqref{ExpectationZ} gives
		$ \mathbb{E}[Z_0 (t)] = \mathbb{E}[\Lambda_0]\int_t^{+\infty} \varphi(s)ds $ so that \\$\mu(t) = \mathbb{E}[\Lambda_0] \left(1-\int_0^{+\infty} \varphi(s)ds\right) = \mu \quad dt-a.e.$
		Consequently, the Hawkes process has at least constant mean i.e. $\mathbb{E}[\Lambda_t] = \mathbb{E}[\Lambda_0]$ for every $t \geq 0$ if 
		$$\forall\, t\ge 0, \qquad \mu(t) \equiv C^{\textit{ste}} =\mu(0) = \mu \quad  dt-a.e. \quad \textit{and} \quad \mathbb{E}[\Lambda_t] = \mathbb{E}[\Lambda_0]= \frac{\mu}{1-\mathit{\Phi}(0) }.$$
		Conversely one checks that for this value of $\mu(t)$, the expectation of $\Lambda_t$ is constant.
		This leads to the results stated in proposition~\ref{prop:stat} (2), since the equation must be true for any t.

		\medskip
		\noindent Assume that $Z_0(t)$ is a square-integrable stochastic process and $R^\prime_{-1} \in L^2_{\text{loc}}(\mathbb{R}_+)$.
		Recall from Theorem \ref{MartRep}, the representation: \(\Lambda_t = Z_0(t) + \mu(t) + (R^\prime_{-1} * (Z_0 + \mu))_t + (R^\prime_{-1} \stackrel{M}{*} \mathbf{1})_t.\)
		Then using elementary computations and It\^o's type Isometry for stochastic integral with respect to a local martingale, the second moment of $\Lambda_t$ is given by: \(\mathbb{E}[\Lambda_t^2] = \mathbb{E}\left[ \left( Z_0(t) + \mu(t) + (R^\prime_{-1} * (Z_0 + \mu))_t \right)^2 \right]
		+ \int_0^t \left( R^\prime_{-1}(t - s) \right)^2 \mathbb{E}[\Lambda_s] \, ds.\)
		Hence, the variance is:
		\[
		\operatorname{Var}[\Lambda_t] = \mathbb{E}\left[(\Lambda_t-\mathbb{E}[\Lambda_t])^2\right] = \operatorname{Var}\left(Z_0(t)\right) +
		\mathbb{E}\left[(R^\prime_{-1} * (Z_0(\cdot)-\mathbb{E}[Z_0(\cdot)]))^2_t\right] 
		+ \int_0^t \left( R^\prime_{-1}(t - s) \right)^2 \mathbb{E}[\Lambda_s] \, ds.
		\]
		Obtaining a constant variance over time requires stringent conditions on the process \(Z_0 \). In the remark below, we analyse the particular case \(Z_0 \equiv 0\), which eliminates the dependence on the initial randomness and on its convolutional propagation. \hfill $\Box$
		
	\subsection{Proof of Proposition~\ref{Lem:Lhawkes}, Proposition~\ref{prop:stability} and Corollary \ref{Corol:LimitMeasure}}
	\noindent {\bf Proof of Proposition \ref{Lem:Lhawkes}.}
		\smallskip
		\noindent  {\sc Step~1} {\em (Convergence of Laplace transform).} 
		By the properties of Laplace transform, we have:
		{\small
		\begin{align}
			&\, \mathcal{L}\left(\int_0^{t} T(1 - a_T) R^{\prime T}_{-1}(Tu) \, du \right)(z)= \frac{1}{z} \mathcal{L}\left( T(1 - a_T) R^{\prime T}_{-1}(T\cdot)  \right)(z) = \frac{T(1 - a_T)}{z} \mathcal{L}\left(  R^{\prime T}_{-1}(T\cdot) \right)(z) \notag \\
			&\hspace{2cm}=\frac{(1 - a_T)}{z} \mathcal{L}\left(  R^{\prime T}_{-1} \right)(\frac{z}{T}) = \frac{(1 - a_T)}{z}\left[\frac{z}{T}L_{R^{T}_{-1}}(\frac{z}{T}) - R^{ T}_{-1}(0)\right] 
			= \frac{(1 - a_T)}{z} \frac{ L_{\varphi^{T}} (\frac{z}{T})}{1- L_{\varphi^{T}} (\frac{z}{T})}\notag
		\end{align}
		}
		where the penultimate equatily follows from equation ~\eqref{eq:Lapl_Resolvent}.
		We then deduce that:
		{\small
		\begin{equation}\label{eq:lapltrans}
			\mathcal{L}\left(\int_0^{t} T(1 - a_T) R^{\prime T}_{-1}(Tu) \, du \right)(z)  = \frac{(1 - a_T)}{z} \frac{ a_T L_{\varphi} (\frac{z}{T})}{1- a_T L_{\varphi} (\frac{z}{T})}  = \frac{1}{z} \frac{ a_T L_{\varphi} (\frac{z}{T})}{1
				+ \frac{a_T}{1 - a_T}(1- L_{\varphi} (\frac{z}{T})) }
		\end{equation}
		}
		\noindent The monotonicity of \( L_\varphi \) and the assumption that \( a_T \to 1 \) induce that \(\lim_{T \to \infty} a_T \cdot L_\varphi\left( \frac{z}{T} \right) = 1, \quad z > 0.\)
		To study the asymptotics of the last denominator in equation~\eqref{eq:lapltrans}, we rely on the fact that the tail distribution \( \Phi: t \to \int_t^\infty \varphi(s) \, ds \), is regularly varying with index  \(  -\alpha \).	
		Now, by an integration by parts and using that \( \|\varphi\|_1 = 1 \), we get \(L_{\varphi}(z) = \int_0^{\infty} \varphi(x) e^{-zx} \, dx = 1 - z \int_0^{\infty} \int_x^{\infty} \varphi(u) \, du e^{-zx} \, dx.\)
	  Therefore, by  change of variables,
	 \[L_{\varphi}(z) = 1 - z \int_0^{\infty} \mathit{\Phi}(x) e^{-zx} \, dx =1 -  \int_0^{\infty} \mathit{\Phi}(\frac{x}{z}) e^{-x} \, dx  = 1 - \frac{z^\alpha}{\alpha} \int_0^{\infty} \alpha\left( \frac{x}{z} \right)^{\alpha} \mathit{\Phi}(\frac{x}{z}) x^{-\alpha}e^{-x} \, dx.\]
		\noindent Hence, using Assumption \ref{assum:Lhawkes} (2) together with the dominated convergence theorem, we obtain :
		
		\[L_{\varphi}(z) = 1 - \frac C\alpha \left(\int_0^{\infty} x^{-\alpha}e^{-x} \, dx \right)  z^\alpha + \underset{z \to 0}{o}(z^\alpha) = 1 - \frac C\alpha \Gamma(1 - \alpha)  z^\alpha + \underset{z \to 0}{o}(z^\alpha).\]
		Indeed, one could alternatively apply \cite[Corollary 8.1.7, p.~334]{BiGoTe1989} to establish that \(1 - L_{\varphi}(z) \underset{z \to 0^+}{\rightarrow} \Gamma(1 - \alpha) \, \Phi\left(\tfrac{1}{z}\right).\)
		Furthermore, under Assumption~\ref{assum:Lhawkes}~(2), which equivalently asserts that $
		\Phi\left(\tfrac{1}{z}\right) \underset{ z \to 0^+}{\sim} \frac{C}{\alpha} z^{\alpha},$
		we deduce the asymptotic expansion:
		\(1 - L_{\varphi}(z) = \frac{C}{\alpha} \, \Gamma(1 - \alpha) \, z^{\alpha} + o(z^{\alpha}) \quad \text{as } z \to 0.\)
		From this, we easily deduce that for \( z > 0 \),
		\begin{equation}
			\mathcal{L}\left(\int_0^{t} T(1 - a_T) R^{\prime T}_{-1}(Tu) \, du \right)(z)  \underset{T \to +\infty}{=} \frac{1}{z} \frac{ a_T }{1
				+ a_T \frac{\frac{C \Gamma(1-\alpha)}{\alpha}}{T^\alpha (1 - a_T)}z^\alpha }
		\end{equation}
		Hence, Owing to Assumption \ref{assum:Lhawkes} (3), \( \mathcal{L}\left(\int_0^{t} T(1 - a_T) R^{\prime T}_{-1}(Tu) \, du \right)(z)  \underset{T \to +\infty}{\to} \frac{1}{z} \frac{\lambda}{\lambda + z^\alpha} =\frac{1}{z} L_{f_{\alpha, \lambda}}(z),\)
		where one-to-one correspondence between functions and their Laplace transforms, together with equation ~\eqref{eq:lapl_resolvent} of Example~\ref{Ex:frackernel} (which give the Laplace transform of the Mittag-Leffler density function \( f_{\alpha, \lambda} \)) yields the last equality.
		
		\medskip
		\noindent {\sc Step~2} {\em (The convergence of the finite measure with density \(T(1 - a_T) R^{\prime T}_{-1}(T\cdot)\)).} 
		By the properties of Laplace Transform, this can be rewritten as follows:
		\begin{equation}
			\mathcal{L}\left(\int_0^{t} T(1 - a_T) R^{\prime T}_{-1}(Tu) \, du \right)(z)  \underset{T \to +\infty}{\to} \frac{1}{z} L_{f_{\alpha, \lambda}}(z) = \mathcal{L}\left(\int_0^{t} f_{\alpha, \lambda} (s) \, ds \right)(z) ,
		\end{equation}
		Thus, still by the pointwise convergence of the Laplace transform yields the weak convergence of the finite measure with density function \(T(1 - a_T) R^{\prime T}_{-1}(T\cdot)\) to the finite measure
		with density function $f_{\alpha, \lambda}$.
		
		\noindent The last claim of the Proposition is deduced easily by the injectivity of the Laplace transform and the dominated convergence theorem. \hfill$\Box$
		%\qed
		
		\medskip	
		\noindent {\bf Proof of Proposition~\ref{prop:stability}:}
		For any \( T \geq 0 \), we can write
		\begin{align*}
			\int_0^{t_0} \left| G(t) - \int_0^t g(t - s) m^*(ds)  \right| dt 
			&\leq \int_0^{t_0} |G(t) - G^T(t)| \, dt  + \int_0^{t_0} \left| \int_0^{t} (g^T(t - s) - g(t - s)) \, m^{T}(ds) \right| dt \\
			&\quad + \int_0^{t_0} \left| \int_0^{t} g(t - s) \, m^{T}(ds) - \int_0^t g(t - s) \, m^{*}(ds) \right| dt.
		\end{align*}
		The first term vanishes as \( T \to \infty \) by assumption. For the second term, using regular Fubini theorem:
		\begin{align*}
			&\	\int_0^{t_0} \left| \int_0^{t} (g^T(t - s) - g(t - s)) \, m^{T}(ds) \right| dt 
			= \int_0^{t_0} \int_0^{t_0} \mathbf{1}_{\{s \leq t\}} |g^T(t - s) - g(t - s)| dt \,m^{T}(ds) \\
			&\hspace{.5cm}= \int_0^{t_0} m^{T}(ds) \int_s^{t_0} |g^T(t - s) - g(t - s)| dt = \int_0^{t_0} m^{T}(ds) \int_0^{t_0 - s} |g^T(u) - g(u)| du \\
			&\hspace{7cm}\leq \|g^T - g\|_{L^1([0, t_0])} \cdot \int_0^{t_0} m^{T}(ds),
		\end{align*}
		which goes to 0 as \( T \to \infty \), by the \( L^1 \)-convergence of \( g^T \to g \) and integrability of \( m^{T} \to m^{*} \). For the third and last term, by the first claim of Corollary \ref{Corol:LimitMeasure} below, we deduce its convergence to 0 as \( T \to \infty \).
		Finally, \(\int_0^{t_0} \left| G(t) - \int_0^t g(t - s) m^*(ds)  \right| dt = 0,\) and hence \(G(t) = \int_0^t g(t - s)  m^*(ds)  \quad \text{for } t \leq t_0, \text{ by the continuity of } G \text{ and } g.\)  \hfill$\Box$
		
		\medskip		
		\noindent {\bf Proof of Corollary \ref{Corol:LimitMeasure}:}
		\noindent  {\sc Step~1} {\em (The convergence of the equation with \(\theta\) ):} Owing to the commutative property of convolution, we have $\int_0^{t} T(1 - a_T) R^{\prime T}_{-1}(T(t-s))\theta(s) \, ds = \int_0^{t} T(1 - a_T) R^{\prime T}_{-1}(Ts)\theta(t-s) \, ds, $ and
		since \( \theta \) is continuous, we deduce that as \(T \to +\infty\):
		\[
		\int_0^t T(1 - a_T) R^{\prime T}_{-1}(Ts) \theta(t - s) \, ds = \int_0^t \theta(t - s) \, m^{T}(ds) \xrightarrow{\text{a.s.}} \int_0^t \theta(t - s) \, m^*(ds) = \int_0^t  f_{\alpha,\lambda}(s) \theta(t - s) \, ds.
		\]
		More precisely, the proof consists in bounding the difference between those integrals, which vanishes in the limit due to the continuity of \( \theta \) and integrability of \( f^{\alpha,\lambda} \).
		This convergence implies that  almost surely,
		\[\sup_{t\geq0}|\int_0^{t} T(1 - a_T) R^{\prime T}_{-1}(T(t-s))\theta(s) \, ds- \int_0^t f_{\alpha,\lambda}(t-s)\theta(s)ds|\rightarrow 0 \; \text{as}\;T \to +\infty. \]
		Consequently, and in particular, we deduce the weak convergence in the space \( \mathcal{C}([0,t_0]; \mathbb{R}) \).
		
		\medskip
		\noindent {\sc Step~2} {\em (The convergence of the equation with \(\mathit{\Phi}\) ).} 
		we have using Fubini's Lemma in the third equality:
		{\small
		\begin{align*}
			&\,(R^{\prime T}_{-1}*\mathit{\Phi}^T)(t) = \int_0^t \left(\int_{t-s}^{+\infty} a_T \varphi^T(r) dr\right) R^{\prime T}_{-1}(s) ds = \int_0^t a_T \left(1 - \int_0^{t-s} \varphi(r) dr\right) R^{\prime T}_{-1}(s) ds= a_T \int_0^t R^{\prime T}_{-1}(s) ds \\ 
			&- \int_0^t \int_0^{t-s} \varphi^T(t-s-r) dr R^{\prime T}_{-1}(s) ds = a_T \int_0^t R^{\prime T}_{-1}(s) ds - \int_0^t \left(\int_0^{t-r} \varphi^T(t-s-r) R^{\prime T}_{-1}(s) ds\right) dr\\
			&= a_T \int_0^t R^{\prime T}_{-1}(s) ds - \int_0^t \big(R^{\prime T}_{-1}(t-r)-\varphi^T(t-r)\big) dr  
			= (a_T - 1) \int_0^t R^{\prime T}_{-1}(s) ds - \int_0^t \varphi^T(s) ds.
		\end{align*}
	    }
		Replacing back in the original equation, it turns out that:
		\begin{align*}
			\mathit{\Phi}^T(t) + (R^{\prime T}_{-1}*\mathit{\Phi}^T)(t) &= \int_t^{+\infty} \varphi^T ds + (a_T - 1) \int_0^t R^{\prime T}_{-1}(s) ds - \int_0^t \varphi^T(s) ds\\ 
			&= a_T \int_0^{+\infty} \varphi ds + (a_T - 1) \int_0^t R^{\prime T}_{-1}(s) ds= a_T + (a_T - 1) \int_0^t R^{\prime T}_{-1}(s) ds.
		\end{align*}
		so that
		$\mathit{\Phi}^T(Tt) + (R^{\prime T}_{-1}*\mathit{\Phi}^T)(Tt) = a_T - (1-a_T) \int_0^{Tt} R^{\prime T}_{-1}(s) ds = a_T - \int_0^{t} T (1-a_T) R^{\prime T}_{-1}(Ts) ds$, then:
		{\small 
			\begin{align*}
				T^\alpha \frac{ (1 - a_T)}{\tilde{\mu}^T \left(\frac{t}{T}\right)}
				\Big( \mathit{\Phi}^T(tT) + (R^{\prime T}_{-1} * \mathit{\Phi}^T)(tT) \Big) 
				&= T^\alpha \frac{ a_T (1 - a_T)}{\tilde{\mu}^T \left( \frac{t}{T} \right)}  - \frac{1}{a_T} \int_0^{t} T (1-a_T) R^{\prime T}_{-1}(T(t-s)) 
				T^\alpha \frac{ a_T (1 - a_T)}{\tilde{\mu}^T \left( \frac{s}{T} \right)} ds \\
				&+ T^\alpha(1 - a_T) \int_0^{t} T (1-a_T) R^{\prime T}_{-1}(T(t-s)) 
				\left( \frac{1}{\tilde{\mu}^T \left( \frac{s}{T} \right)} 
				- \frac{1}{\tilde{\mu}^T \left( \frac{t}{T} \right)} \right) ds
			\end{align*}
		} 
		\noindent By assumption, the first term converges to \(\phi\) as \(T \to +\infty\), the second to \(\int_0^t f_{\alpha,\lambda}(t-s)\phi(s)ds\) owing to the stability result in Proposition~\ref{prop:stability} for the density \(
		m^{T}(ds) \) and the third to 0 thanks to the same Lemma. \\
		For a formal proof of the convergence to 0 of the third and last term, the reader is invited to refer to Lemma~\ref{Lm:cvgcezero}.
		This yields finally: 
		\[T^\alpha\frac{ (1 - a_T)}{\tilde{\mu}^T (\frac{t}{T})}\Big(\mathit{\Phi}^T(tT) + (R^{\prime T}_{-1}*\mathit{\Phi}^T)(tT)\Big) \underset{T \to +\infty}{\rightarrow} \phi(t)- \int_0^t f_{\alpha,\lambda}(t-s)\phi(s)ds.\]
		
		\noindent Note that, if \(\phi \equiv 1\), then $\mathit{\Phi}(Tt) + (R^{\prime T}_{-1}*\mathit{\Phi})(Tt) \underset{T \to +\infty}{\rightarrow} 1- \int_0^{t}f_{\alpha, \lambda}(u) \, du = R_{\alpha, \lambda}(t). $
		
		\noindent Finally, the last claim of the corollary follows from the first point of Assumption~\ref{assum:Lhawkes} and the uniform (non-degeneracy) bound on~$\tilde{\mu}$ stated in the remark and discussions following the same assumption.\\
		This complete the proof and we are done.\hfill $\Box$
\end{document}